\documentclass[reqno, 11 pt]{amsart}
\usepackage{mathtools}
\usepackage{amssymb}   
\usepackage{amsmath}   
\usepackage{euscript}   
\usepackage{hyperref} 
\usepackage{amsthm} 
\usepackage{cleveref} 
\usepackage{dsfont}  
\usepackage{amsfonts}    
\usepackage{marginnote}
\usepackage{latexsym} 
\usepackage{amsopn} %per definire nuovi operatori "operator name"
\usepackage{geometry}
\usepackage{amscd} %per fare diagrammi rettangolari (no frecce diagonali) e mettere le label sulle frecce 
\usepackage{mathrsfs}
\usepackage{caption}
\usepackage{aurical}  
\usepackage[english]{babel}
\usepackage[utf8]{inputenc}
\usepackage{esint} %per i simbOli di integrale
\usepackage{graphicx}
\usepackage[dvipsnames]{xcolor}
\usepackage{multicol}
\usepackage{nomencl}
\makeglossary
\makenomenclature  
\usepackage{refcount}
\usepackage{float}
\usepackage{times}
\usepackage[titletoc]{appendix}
\newcommand{\comment}[1]{}

\definecolor{airforce}{rgb}{0.36, 0.74, 0.86}
 	\definecolor{blue-green}{rgb}{0.0, 0.57, 0.87}

\newcommand{\eps}{\varepsilon}
\newcommand{\e}{{\varepsilon}}

\renewcommand{\ge}{\geqslant}
\renewcommand{\geq}{\geqslant}
\renewcommand{\le}{\leqslant}
\renewcommand{\leq}{\leqslant}
\renewcommand{\Re}{\operatorname{Re}}
\renewcommand{\Im}{\operatorname{Im}}

\newcommand{\id}{\mathrm{Id}}

\newcommand{\opbw}{{Op^{\mathrm{BW}}}}
\newcommand{\opb}{{Op^{\mathrm{B}}}}

\providecommand{\vect}[2]{{\bigl[\begin{smallmatrix}#1\\#2\end{smallmatrix}\bigr]}}   
\providecommand{\sm}[4]{{\bigl[\begin{smallmatrix}#1&#2\\#3&#4\end{smallmatrix}\bigr]}}

\newtheorem{theorem}{Theorem}[section]
\newtheorem*{thm*}{Theorem}
\newtheorem{proposition}[theorem]{Proposition}
\newtheorem{lemma}[theorem]{Lemma}

\newtheorem{cor}[theorem]{Corollary}
\newtheorem*{cor*}{Corollary}
\newtheorem{remark}[theorem]{Remark}
\newtheorem{ex}{Example}

\newtheorem{definition}[theorem]{Definition}

\numberwithin{equation}{section}
\newcommand{\ii}{{\rm i}}

\newcommand{\ov}{\overline}

\newcommand{\C}{{\mathbb C}}

\newcommand{\N}{{\mathbb N}}

\newcommand{\R}{{\mathbb R}}

\newcommand{\T}{{\mathbb T}}
\newcommand{\Z}{{\mathbb Z}}

\usepackage{bm}

\newcommand{\diag}{\mathop{\mathrm{diag}}}

\newcommand{\nnorm}[1]{{\left\vert\kern-0.25ex\left\vert\kern-0.25ex\left\vert #1 
    \right\vert\kern-0.25ex\right\vert\kern-0.25ex\right\vert}}
\usepackage{framed,enumitem}

\providecommand{\vect}[2]{{\bigl[\begin{smallmatrix}#1\\#2\end{smallmatrix}\bigr]}} 
  
\providecommand{\sm}[4]{{\bigl[\begin{smallmatrix}#1&#2\\#3&#4\end{smallmatrix}\bigr]}}

\makeatletter

\renewcommand{\tocsection}[3]{%
\indentlabel{\@ifnotempty{#2}{\bfseries\ignorespaces#1 #2\quad}}\bfseries#3}
\def\l@subsection{\@tocline{2}{0pt}{2.5pc}{5pc}{}}
\def\l@subsubsection{\@tocline{3}{0pt}{4.5pc}{5pc}{}}
\renewcommand\tocchapter[3]{%
  \indentlabel{\@ifnotempty{#2}{\ignorespaces#2.\quad}}#3%
}
\begin{document} 
 
\title[Long time strong Sobolev instability for quasilinear NLS]{Long time strong Sobolev instability for quasilinear NLS}
\date{\today}

\author{Yuri Cacchiò}
\address{\scriptsize{Department of Mathematics, University of Vienna, Oskar-Morgenstern-Platz 1, 1090, Wien, Austria}}
\email{yuri.cacchio@univie.ac.at}

\author{Filippo Giuliani}
\address{\scriptsize{Dipartimento di Matematica, Politecnico di Milano, Piazza Leonardo Da Vinci 32, 20133, Milano, Italy}}
\email{filippo.giuliani@polimi.it}

\author{Felice  Iandoli}
\address{\scriptsize{Dipartimento di Matematica ed Informatica, 
Universit\`a della Calabria, Ponte Pietro Bucci,  87036, Rende, Italy}}
\email{felice.iandoli@unical.it}

\author{Raffaele Scandone}
\address{\scriptsize{Dipartimento di Matematica e Applicazioni ``R. Caccioppoli'', 
Universit\`a degli Studi di Napoli Federico II, Complesso Universitario Monte S.~Angelo, via Cintia, 80126, Napoli, Italy}}
\email{raffaele.scandone@unina.it}



\begin{abstract}
We prove long-time strong Sobolev instability for a class of quasilinear Schr\"odinger equations on \(\mathbb T^2\) arising, via the Madelung transform, in the description of Euler–Korteweg capillary fluids. More precisely, for every \(s\ge 7\), arbitrarily small \(\mu>0\), and arbitrarily large \(\mathcal K>0\), we construct a solution such that
$$ \|u(0)\|_{H^s}<\mu, \qquad \|u(T)\|_{H^s}>\mathcal K $$
at some finite time \(T\), for which we provide an explicit exponential upper bound.
The instability mechanism is generated by a finite-dimensional resonant Toy Model associated with the cubic semilinear part of the equation, in the spirit of Colliander–Keel–Staffilani–Takaoka–Tao. The main technical difficulty is to justify this dynamics over the long time scale of the energy transfer in the presence of quasilinear derivative losses. We achieve this by combining a weak Birkhoff normal form with paradifferential calculus, microlocal symmetrization, and modified energy estimates.
Our result provides a rigorous quasilinear realization of a resonant energy-transfer mechanism, and suggests a connection with weakly turbulent dynamics in capillary fluids.

\end{abstract}
    
\maketitle

\setcounter{tocdepth}{1}
\tableofcontents

%%%%%%%%%%%%%%%%%%%%%%%%%%%%%%%%%%%%%%%%%%%%%%%%%%%%%%%

\section{Introduction}
\subsection{Setup and main result}
We consider a class of quasilinear  Schr\"odinger equations in dimension two, with periodic boundary conditions, and we investigate the orbital instability of the flow in high regularity Sobolev spaces. %\red{This issue is deeply related with the weak turbulence theory for quantum systems, as discussed below in more details.}\blue{Questa frase qui e' un po' tricky, non si capisce bene e non credo aggiunga interesse, e' una cosa che va spiegata bene perche' la relazione non e' immediata, soprattutto perche' qui non si ha gwp e quindi nessuna norma e' controllata uniformemente in tempo} 
The model under analysis is the following cubic, defocusing Schr\"odinger equation
\begin{equation*}\label{NLS}\tag{NLS}
\left\{\begin{aligned}
&\ii\partial_t u+\Delta u
+\big[\Delta (h(|u|^2))\big]h'(|u|^2)u-|u|^2u=0\\
&u(0,x)=u_0(x),\end{aligned}\right.
\end{equation*}
in the unknown $\C\ni u:=u(t,x)$, $t\in\R$, $x\in \T^2$, where $h(x)$ is a function in $\mathcal{C}^{\infty}(\R;\mathbb{R})$ with a zero of order at least two at the origin. {We will be interested in the long time dynamics of smooth initial data $u_0$.}

Quasilinear Schr\"odinger equations, such as \eqref{NLS}, provide effective models for wave propagation in nonlinear dispersive media where the constitutive response is coupled to field gradients \cite{due}. These models find significant applications in nonlinear optics, specifically in describing self-steepening mechanisms and the dynamics of ultra-short pulses \cite{tre}, as well as in the study of superfluid films and Bose-Einstein condensates when density-dependent dispersion or beyond-mean-field effects are taken into account \cite{quattro,uno}. Moreover, quasilinear Schr\"odinger equations arise in the description of capillarity effects in compressible fluids, we refer to Section \ref{sec:hydro} below for further discussion and perspectives.

The local well-posedness of \eqref{NLS} in $H^s(\T^d)$ with $s>\frac{d}{2}+2$ was proved in \cite{iandoli} (see also \cite{IN2023}), improving upon the earlier work of Feola and Iandoli \cite{FIJMPA}. In contrast, the literature concerning the Euclidean space $\mathbb{R}^d$ is much broader. After Poppenberg's $1$-dimensional result \cite{due}, Kenig, Ponce, and Vega \cite{KPV3, KPV2, KPV} established pioneering results in arbitrary dimension, for a broader class of equations. Subsequently, Marzuola, Metcalfe, and Tataru \cite{MMT,MMT2,MMT3} improved these works by lowering the required regularity of the initial data.

 In the presence of suitable convolution potentials, Feola-Gr\'ebert-Iandoli \cite{FGI} proved a long-time stability result for \eqref{NLS}. More precisely, if $\| u_0\|_{H^s}=\e$ for $\e>0$ sufficiently small, the unique solution to the Cauchy problem \eqref{NLS} has a lifespan of order at least $\e^{-4}$ and its norm remains of order $\e$. In these settings, the quasilinear terms act as small perturbations, preserving the stability of the equilibrium over long time scales. 

On the other hand, in the absence of external potentials which rule out nontrivial resonances, one expects resonant and nonlinear effects to trigger long time instability phenomena. The present paper is devoted to establishing this mechanism.

\medskip

We remark that, for sufficiently regular initial data, the following quantities
\begin{itemize}
	\item total mass
	\begin{equation}\label{mass}
	\mathcal{M}:=\int_{\T^2} |u|^2\,dx,
	\end{equation}
\item total energy 
	\begin{equation}\label{energy}
		H=\int_{\T^2} |\nabla u |^2\,dx+\frac{1}{2}\int_{\T^2} |u |^4\,dx +\frac{1}{2}\int_{\T^2} |\nabla (h(|u|^2)) |^2 dx,
	\end{equation}
\end{itemize}
are conserved along the flow of \eqref{NLS}. In fact, $H$ is a Hamiltonian function for \eqref{NLS}, see Section \ref{sec:ham}.

%\red{In view of the above conservation laws, the $H^1$-norm of a solution is uniformly bounded in time.} \blue{Bisogna stare attenti: questo e' vero fintanto che le soluzioni sono abbastanza regolari ed esistono, purtroppo non abbiamo global well posedness. E' vero pero' che per le nostre soluzioni alcune norme di Sobolev basse, comresa $H^1$ restano molto piccole mentre altre norme alte crescono.} 
 While the aforementioned conservation laws provide an \emph{a priori} bound on the $H^1$-norm as long as the solution remains sufficiently regular, the lack of a global well-posedness theory for \eqref{NLS} prevents us to deduce uniform in time upper bounds. 
 Nevertheless, we prove the existence of solutions with very long lifespan whose low Sobolev norms, including the $H^1$
 norm, remain arbitrarily small, while their higher Sobolev norms undergo an arbitrarily large inflation over time. This behavior is closely connected with transfers of energy among modes of substantially different scales, commonly referred to as energy cascades.

\smallskip

Such mechanisms plays a central role in the \emph{weak turbulence} theory, whose general aim is to study the long-time instability phenomena in dispersive models (such as those arising from quantum mechanics) produced by the interactions between nonlinear waves: we refer to the monograph \cite{Nazarenko} by Nazarenko and references therein for a comprehensive discussion. Formal theoretical physics arguments suggest that energy transfer is generic with respect to a suitable randomization of the initial data. For the \emph{semilinear} NLS, a relevant step in this direction has been obtained by Deng-Hani \cite{DH2023}, who provide a rigorous derivation of the wave kinetic equation, an effective description for the interaction between Fourier modes for randomized data, under certain scaling limits. Still, current techniques only allows to construct specific solutions which exhibit norm inflation. For semilinear models, there are nowadays various results in this direction, while in the quasilinear setting few results are available; we refer to Section \ref{se:comparison} for a comprehensive discussion of the literature. In particular, instability mechanisms for the quasilinear equation \eqref{NLS} were still unknown, and our main result provides a first positive answer to this question.

\begin{theorem}\label{th:main}
Fix $s \ge 7$. 
There exist absolute constants $C,\gamma>0$ such that for any sufficiently small $\mu > 0$ and any sufficiently large $\mathcal{K}>0$, the following holds.
There exists an initial datum $u_0\in H^s(\T^2)$ and a time $T=T(\mu,\mathcal{K})>0$ satisfying 
\begin{equation}\label{time_bound}
    T \le \exp\left( C \left(\frac{\mathcal{K}}{\mu}\right)^{\frac{\gamma}{s-1}} \right), 
\end{equation}
 such that \eqref{NLS} admits a unique solution $u\in\mathcal{C}([0,T],H^s(\T^2))$ satisfying $u(0)=u_0$ and 
$$\|u_0\|_{H^s}<\mu,\qquad \|u(T)\|_{H^s}>\mathcal{K}.$$
\end{theorem}
%\red{$\gamma$ non e' una absolute constant, dipende da $s$, infatti la norma $H^1$ nel toy model e' preservata, quindi se $s\to 1$ l'upper bound nel tempo esplode. }{\color{orange} HAi ragione. Inserito il $\gamma$ esplicito che avevo derivato nella dimostrazione }

Theorem \ref{th:main} provides the existence of small data solutions escaping arbitrarily far from the rest equilibrium $u=0$. This can be seen, from a dynamical point of view, as a result of strong topological instability: Hani \cite{H} named this phenomenon as \emph{long time strong instability}.

\begin{remark}
    Once an unstable trajectory is constructed on $\T^2$, it can also be embedded into higher dimensions. In dimension one, the picture is instead qualitatively different. In the semilinear setting (namely, $h\equiv 0$), the corresponding NLS equation is completely integrable, and all Sobolev norms remain uniformly bounded in time. In the quasilinear setting, the exact integrable structure is no longer available, and whether unstable solutions exist remains an open problem. Still, exploring this direction would require new technical ideas.
\end{remark}

\begin{remark}
The specific structure of \eqref{NLS}  restricts the possible  quasilinear nonlinearities. For a smooth function \(h\), the corresponding quasilinear term is either cubic or septic (or higher-order). Quintic terms could be considered within a more general class of Hamiltonian Schr\"odinger equations, such as those studied in \cite{FIJMPA,KPV3,KPV2,KPV}. We do not pursue this setting here, since \eqref{NLS} is particularly relevant for the physical and hydrodynamic applications discussed in Section \ref{sec:hydro}. We restrict our analysis to septic (or higher-order) terms, corresponding to the assumption that $h$ has a zero of order at least two at the origin. The cubic case instead is not covered by our method. In this regime, the quasilinear interactions would have to be incorporated directly into the effective resonant finite-dimensional dynamics described in Section \ref{sec:toy}. Constructing such a resonant dynamics for quasilinear equations appears to be extremely challenging at present.
\end{remark}

\begin{remark}
By the local well-posedness and the continuous dependence on initial data for \eqref{NLS}, the strict inequalities in Theorem \ref{th:main} persist for any initial data in a sufficiently small neighborhood of the constructed $u_0$. Our approximation argument gives a quantitative version of this observation: the explicit neighborhood in \eqref{eq:open} consists of initial data whose solutions exist  up to the same time and exhibit the prescribed Sobolev norm inflation.
\end{remark}

\begin{remark}
The lower bound $ s\geqslant 7$ could be slightly improved, but our methods appear to require at least $s>4$, due to the quasilinear structure of \eqref{NLS} -- see the discussion in Subsection \ref{sec:s_0} of Appendix \ref{app:para}.
\end{remark}

\subsection{Comparison with previous literature}\label{se:comparison}
The question of whether dispersive PDEs on compact manifolds could exhibit unbounded growth of Sobolev norms was raised by Bourgain \cite{Bourgain1996} and Kuksin \cite{Kuksin1997}. A breakthrough result in this direction was achieved by Colliander, Keel, Staffilani, Takaoka, Tao \cite{CKSTT}, who provided the first rigorous proof of \emph{long time strong instability} for the defocusing cubic NLS on $\T^2$. Following this seminal work, an extensive literature has emerged. Guardia-Kaloshin \cite{GK2015} established polynomial upper bounds on the instability time. Subsequently, the weak turbulence mechanism was generalized to several other settings. It was extended to quintic NLS by Haus, Procesi \cite{HP2015} and then to generic analytic NLS by Guardia, Haus, and Procesi \cite{GHP2016}. It was extended to neighborhoods of finite-gap tori by Guardia-Hani-Haus-Maspero-Procesi \cite{GHHMP}. For the cubic NLS on the product space $\T^2\times \R$, Hani-Pausader-Tzvetkov-Visciglia \cite{HPTV2015} obtained the existence of trajectories with unbounded growth, exploiting the dispersion in the $\R$-direction. We point out that the analogous result is unknown on any compact manifold. The framework of irrational tori has been considered by Giuliani-Guardia \cite{GG2022}, Staffilani-Wilson \cite{SW2020}, and Giuliani \cite{G2025}, which also considers NLS with convolution potentials.

Concerning quasilinear NLS, the situation is more involved and few results are available. Indeed, establishing the growth of $H^s$ norms requires controlling the dynamics in regimes where the quasilinear terms can no longer be treated as small perturbations. This feature, combined with the loss of derivatives related to the quasilinear PDEs, poses a technical obstacle for the study of norm inflation. Recently, Maspero-Murgante \cite{MM2024} provided one-dimensional energy cascades for a \emph{fractional} quasilinear NLS. Furthermore, Langella-Maspero-Murgante-Terracina \cite{LMMT2026} have established the transfer of energy for pure-gravity water waves, relying on a nonlinear transport structure.

In the present paper, we address the topological instability for the full quasilinear model \eqref{NLS}, which does not rely on nonlinear transport: in particular, the energy cascade is caused by the interaction of dispersive waves, according to the prediction of weak turbulence theory. 

The fundamental obstacle lies in the fact that the higher-order nonlinearities cannot be treated as small perturbations in any suitable topology. Specifically, tracking the evolution of data that eventually escape from any bounded set in $H^s$ prevents the nonlinear terms from being considered  small in such topologies. On the other hand, the presence of second-order derivatives in the quasilinear term $[\Delta ( h(|u|^2) )] h'(|u|^2) u$ prevents working in weak topologies (such as $\mathcal{F}\ell^1$), which are exploited in the semilinear literature \cite{CKSTT, GHP2016, GK2015}. This rigidity precludes the use of standard Birkhoff normal forms, requiring the development of a different strategy, which we describe
below.

\subsection{Hydrodynamic interpretation}\label{sec:hydro} 
Let $u=|u|e^{\ii\theta}$ be a solution to \eqref{NLS}. Then its Madelung transform 
\begin{equation}\label{eq:madelung}
\rho=|u|^2,\quad v=2\nabla\theta,
\end{equation}
formally solves the Euler-Korteweg system
\begin{equation}\label{eq:EK}
	\begin{cases}
		\partial_t \rho+\operatorname{div}(\rho v)=0, \\
		\partial_t (\rho v)+\operatorname{div}(\rho v \otimes v)+ \nabla p(\rho)=\rho\nabla \Big(\operatorname{div}(k(\rho)\nabla\rho)-{\textstyle{\frac{1}{2}}}k'(\rho)|\nabla\rho|^2\Big),
	\end{cases}
\end{equation}
where $p(\rho)=\rho^2$ is the pressure term, and $k(\rho)=\frac{1}{\rho}+2h'(\rho)^2$ is the so-called capillarity coefficient. System \eqref{eq:EK}, in which the classical compressible Euler equation are augmented by a third order nonlinear term, describes capillarity effect in diffuses interfaces \cite{Kort}. The above connection was first observed in the seminal work by Madelung \cite{Madelung}, and then widely exploited in the mathematical literature \cite{AMQ3,AMQ2,AMHZ-Rims,CDS}. We refer to the review paper \cite{AM2026} and references therein for a detailed discussion. The particular choice $k(\rho)=\frac{1}{\rho}$ in \eqref{eq:EK} gives the so-called quantum hydrodynamic (QHD) system, used to describe the dynamics of compressible fluid exhibiting quantum effects, see e.g.~\cite{Khal}. In this case, the corresponding NLS become semilinear, namely $h=0$.

Local well-posedness for \eqref{eq:EK} on $\T^d$, at sufficiently high Sobolev regularity, has been proved by Berti-Maspero-Murgante \cite{BMM1}. Moreover, Feola-Iandoli-Murgante \cite{FIM2021} established the long-time stability of the QHD system on irrational tori.

The existence of solutions to \eqref{eq:EK} exhibiting energy cascade is a challenging question. A main issue is related to the rigorous application of the Madelung transform in presence of vacuum, namely when $u$ vanishes at some point. First of all, the velocity field $v=2\nabla\theta$ is not well-defined, and \eqref{eq:madelung} must be suitably generalized in order to deal with meaningful hydrodynamic quantities, see \cite{AM2026} for a comprehensive discussion. Moreover, one can not expect to transfer information on high Sobolev norms from \eqref{NLS} to \eqref{eq:EK}: for example, the Nemitskii operator $u\mapsto\sqrt{\rho}=|u|$ is bounded on $H^s$ only when $s<\frac32$, see e.g.~\cite{BM91}. For the QHD system, Giuliani-Scandone \cite{GS2025} proved the existence of solutions with almost constant mass density, displaying arbitrarily large growth of Sobolev norms. To this aim, they construct solutions to the semilinear NLS which exhibits the required norm inflation and are \emph{small amplitude} perturbation of a \emph{plane wave}. This allows to avoid the issues with vacuum, but introduces further technical difficulties, see also \cite{GHHMP}. Extending the results in \cite{GS2025} to \eqref{eq:EK} is a difficult task, in view of the quasilinear nature of the corresponding Schr\"odinger equation. In this perspective, our analysis lay the technical bases for further understanding of instability mechanisms for compressible, capillary fluids.

\subsection{Strategy of the proof}
The proof of Theorem \ref{th:main} is based on a long-time approximation of special unstable trajectories of a finite-dimensional resonant system by solutions of the full quasilinear \eqref{NLS} equation. The resonant system is the same Toy Model introduced by Colliander, Keel, Staffilani, Takaoka and Tao in \cite{CKSTT}, and originates entirely from the semilinear cubic part of the equation. The main difficulty is therefore to prove that these unstable orbits remain a valid approximation of the full quasilinear PDE for a time interval long enough for the transfer of energy to take place.

We first recall the finite-dimensional mechanism. Following \cite{CKSTT} and its subsequent refinements \cite{GHP2016}, we select a finite set of Fourier modes
$\Lambda\subset \Z^2$
 satisfying suitable closure and non-degeneracy properties with respect to the quartic resonances of the cubic NLS. Restricting the resonant quartic dynamics to \(\Lambda\), and then reducing the degrees of freedom by exploiting the symmetry of the cubic NLS interactions, yields the Toy Model of \cite{CKSTT}. This system possesses special trajectories along which the energy is transferred  from modes carrying a relatively small Sobolev weight to modes carrying a much larger one. After exploiting the cubic scaling of the Toy Model, we obtain trajectories \(v^\lambda\), supported on \(\Lambda\), which display an arbitrarily large growth of the \(H^s\)-norm over a time interval of size $
T=\lambda^2 T_0 $, where $T_0>0$ depends on the prescribed growth, but it is independent of $\lambda$.
The scaling parameter $\lambda$ plays a fundamental role below.

The first step towards comparing these trajectories with the full equation is to introduce coordinates in which the dynamics of the modes in $\Lambda$ is sufficiently decoupled from the remaining Fourier modes. A standard Birkhoff normal form would be unsuitable here. Indeed, although the quasilinear part of the Hamiltonian starts at higher degree, we do not want to modify much the paradifferential structure of \eqref{NLS} after the conjugation with the Birkhoff map.  We therefore perform only a weak version of the Birkhoff normal form. More precisely, at quartic order we eliminate only the monomials containing exactly one mode outside $\Lambda$, namely those which directly couple the Toy-Model dynamics to its normal directions. The corresponding generating Hamiltonian involves only finitely many Fourier modes. Consequently, the time one flow map $\Phi$, and its inverse, differs from the identity by finite-rank operators. This is crucial, since conjugating the equation by this map leaves the \eqref{NLS} paradifferential structure essentially unchanged, up to smoothing contributions. In the new coordinates, the equation consists of the truncated Hamiltonian whose restriction to $\Lambda$ produces the Toy Model, plus higher-order terms that we need to show to be perturbative.

We then turn to the main part of the proof, namely the approximation argument. Let \(z(t)\) solve the transformed NLS equation and let \(v^\lambda(t)\) be the prescribed Toy-Model trajectory. Setting $w:=z-v^\lambda$,
we derive an equation for the difference. A direct energy estimate cannot be closed: the quasilinear remainder contains two derivatives, while the reference trajectory itself develops large high Sobolev norms. We therefore paralinearize the difference equation. The order-two quasilinear contribution is incorporated into the principal operator acting on \(w\), rather than being treated as an external forcing. We then perform microlocal changes of variables which diagonalize and symmetrize this principal operator. After this reduction, the leading second-order symbol becomes diagonal and skew-adjoint at principal level, while the first-order part has the appropriate real structure.

The standard \(H^s\) energy still does not allow to close the argument. We introduce instead a modified paradifferential energy, adapted to the symmetrized operator. Its weight is chosen so that the highest-order commutators, generated when differentiating the energy, cancel at the leading order. 

Despite obtaining an energy estimate for the difference equation, an additional difficulty is to propagate the stability estimate in \(H^s\) up to the long time scale on which the Toy Model exhibits energy transfer. This requires keeping the low Sobolev norms sufficiently small and controlling the loss of derivatives falling on the reference trajectory \(v^\lambda\).\\
To handle the former, we choose $\lambda$ in such a way that certain low norms remain small over long time, while the $H^s$ norm is allowed to display an arbitrarily large growth.\\
To deal with the latter,
a key ingredient is the use of Sobolev norms \(\|\cdot\|_{s,M}\) (see \eqref{norma_sM}) adapted to the Fourier geometry of the set \(\Lambda\). This set is contained in a high-frequency annulus, whereas the normal modes may also lie in its interior, at frequencies much smaller than those of the Toy-Model orbit. This causes a difficulty in the stability estimates: when a low-frequency normal mode interacts with modes in \(\Lambda\), the output may have frequency of order \(M\), and in the standard Sobolev topology the corresponding \(H^s\)-weight cannot be assigned entirely to the normal component. One would then be forced to place high derivatives on the reference solution \(v^\lambda\), whose high Sobolev norms grow along the cascade. The modified weight
\[
\langle n\rangle_M^s=(M^2+|n|^2)^{s/2}
\]
removes this mismatch by assigning a weight of order \(M^s\) also to the low frequencies inside the annulus. This allows the full \(s\)-derivative weight in the cubic stability estimate to fall on the perturbation \(w\), while the factors \(v^\lambda\) are controlled only in weak norms; this mechanism is crucial in Lemma \ref{lem:vector_field_diff}. The paradifferential analysis is therefore developed uniformly in this adapted topology, with the corresponding product, composition, paralinearization and smoothing estimates collected in Appendix \ref{app:para}.

Eventually, a bootstrap argument, combined with the modified energy estimate and Gronwall's inequality, then shows that
$$
\sup_{0\leq t\leq T}
\|z(t)-v^\lambda(t)\|_{s,M}
$$
remains much smaller than the size of the Toy-Model orbit up to the full transfer time \(T\).

Finally, the smallness of the trajectory in a weak Fourier norm guarantees that the Birkhoff map \(\Phi\) remains well defined throughout \([0,T]\). Hence, $u(t)=\Phi(z(t))$
is a solution of the original quasilinear NLS on the same interval. Since \(\Phi\) is close to the identity in \(H^s\), the large Sobolev growth of the Toy-Model trajectory is transferred first to \(z(t)\) and then to \(u(t)\).

\subsection*{Organization of the paper}
In Section \ref{sec:funct} we introduce the functional setting and the notation used throughout the paper. In Section \ref{sec:toy} we introduce the Lambda set and construct unstable trajectories for the corresponding Toy model. Section \ref{sec:wb} is devoted to the Birkhoff normal form procedure. In Section \ref{sec:energy} we provide suitable energy estimates, after introducing the necessary definition and tools from paradifferential calculus. In Section \ref{sec:approximation} we provide a suitable approximation argument, eventually proving our main result, Theorem \ref{th:main}.

\subsection*{Statements and Declarations} The authors declare that no artificial intelligence tools were used in the preparation, analysis, or writing of this manuscript.

\section{Notation and functional setting}\label{sec:funct}
We denote by $\langle x\rangle=\sqrt{1+|x|^2}$ the standard Japanese bracket. Moreover, given $M\geq 1$, we set $\langle  x\rangle_{M}=\sqrt{M^2+|x|^2}$. We denote by $C$ a positive constant, which may vary from line to line. We point out that $C$ may depend on fixed parameters, whose dependence will not be tracked explicitly, but is always independent of the parameter $M$ which defines the modified Japanese bracket. Given $A,B>0$, we write $A\lesssim B$ if $A\leqslant CB$. We write $A\sim B$ if $A\lesssim B$ and $B\lesssim A$. We denote by $B_X(r)$ a ball of radius $r>0$ in a Banach space $X$. We write $(\cdot,\cdot)_{L^{2}}$ the standard complex $L^{2}$-scalar product, namely
\begin{equation}\label{scalarproScalare}
	(u,v)_{L^{2}}:=\int_{\mathbb{T}^{2}}u\bar{v}dx\,, 
	\qquad  \mbox{for all}\,\,\, u,v\in L^{2}(\mathbb{T}^{2};\mathbb{C})\,.
\end{equation}
We expand a function $u(x)$, $x\in \mathbb{T}^{2}$, in Fourier series as 
\begin{equation}\label{complex-uU}
	u(x) = \frac{1}{2\pi}
	\sum_{n \in \mathbb{Z}^{2} } \widehat{u}(n)e^{\ii n\cdot x } \, , \qquad 
	\widehat{u}(n) := \frac{1}{2\pi} \int_{\mathbb{T}^{2}} u(x) e^{-\ii n \cdot x } \, dx \, .
\end{equation}
We also use the notation
$u_n := \widehat{u}(n)$ and $
\ov{u_n}:=\ov{\widehat{u}(n)}$. Given a subset of $\Lambda\subseteq \Z^2$, we consider the projections
\[
\Pi_{\Lambda} u=\frac{1}{2\pi}\sum_{n\in\Lambda} u_n \,e^{\mathrm{i} n x}, \qquad \Pi^{\perp}_{\Lambda} u=\frac{1}{2\pi}\sum_{n\notin\Lambda} u_n \,e^{\mathrm{i} n x}.
\]
We also set, for $\beta\geq 0$,
 \[
\Pi_{\le \beta} u=\frac{1}{2\pi}\sum_{|n|\le \beta} u_n \,e^{\mathrm{i} n x}, \qquad \Pi_{>\beta} u=\frac{1}{2\pi}\sum_{|n|> \beta} u_n \,e^{\mathrm{i} n x}.
\]
Given $s\geqslant 0$, we consider the classical Sobolev space $H^{s}(\mathbb{T}^{2};\mathbb{C})$ endowed with the norm 
$$\|u(\cdot)\|_{s}^{2}:=\sum_{n\in \mathbb{Z}^{2}}\langle n\rangle^{2s}|u_{n}|^{2}\,.$$ 
It will be convenient to consider also the equivalent Sobolev norm 
\begin{equation}\label{norma_sM}
	\|u(\cdot)\|_{s,M}^{2}:=\sum_{n\in \mathbb{Z}^{2}}\langle n\rangle_M^{2s}|u_{n}|^{2}\,.
	\end{equation}
We consider the Fourier multipliers
\[(D^su)_n=\langle n\rangle^su_n,\qquad (D_M^su)_n=\langle n\rangle_M^su_n,
\]
so that we have
\[ 
\| D^s u\|_{\ell^2}=\|u\|_{s},\qquad \| D^s_M u\|_{\ell^2}=\|u\|_{s, M}.
\]
 We recall the following interpolation estimate
\begin{equation}\label{interpolo}
	\|u\|_{{s},M}\leq \|u\|^{\theta}_{{s_1,M}}\|u\|^{1-\theta}_{{s_2,M}}\,,
	\qquad \theta\in[0,1]\,,\;\;s_1\leq s_2\,,\quad s=\theta s_1+(1-\theta)s_2\,.
\end{equation}
{We denote by $\ell^1=\ell^1(\Z^2; \C)$ the space of complex-valued sequences $u = \{u_n\}_{n\in\Z^2}$ such that $\|u\|_{\ell^1} := \sum_{n\in\Z^2}|u_n| < \infty$. In what follows, we will often identify a function $u$ defined on $\mathbb{T}^2$ with its sequence of Fourier coefficients $\{u_n\}_{n\in\Z^2}$. Given two sequences $u, v \in \ell^1$, we denote their discrete convolution by $(u * v)_n := \sum_{m \in \Z^2} u_{n-m} v_m$.

For every $s>1$ and $M\ge1$, the Cauchy--Schwarz inequality gives
\begin{equation}\label{smoothMs}
\|u\|_{\ell^1}= \sum_{n\in\mathbb Z^2}
   \langle n\rangle_M^{-s}
   \langle n\rangle_M^s |u_n| \le
\left(\sum_{n\in\mathbb Z^2}
      (M^2+|n|^2)^{-s}\right)^{1/2}
\|u\|_{s,M}
\lesssim_s M^{1-s}\|u\|_{s,M}.
\end{equation}
where in the last inequality we split the sum into $|n|\le M$ and $2^jM<|n|\le 2^{j+1}M$, $j\ge0$.

Finally, given two smooth symbols $a(x,\xi)$ and $b(x,\xi)$ defined on $\mathbb{T}^2 \times \R^2$, we define their Poisson bracket as customary:
\begin{equation}\label{PoissonBra}
	\{a,b\} := \nabla_{\xi}a \cdot \nabla_{x}b - \nabla_{x}a \cdot \nabla_{\xi}b\,.
\end{equation}

\subsection{Hamiltonian structure of \eqref{NLS}}\label{sec:ham} 
Given $u\in H^s$, we set
\begin{equation}\label{fF}
	U := \vect{u}{\overline{u}}\,,
\end{equation}
and consider the space
\begin{equation}\label{RealSobolev}
	\mathcal{H}^s=\mathcal{H}^s(\T^2) = \big\{\, U = \vect{u}{\overline{u}} : u\in H^s(\T^2) \,\big\}\,,
\end{equation}
which is a closed real subspace of $H^s(\T^2)\times H^s(\T^2)$. 
With a slight abuse of notation, we shall denote by $\|\cdot\|_{{s},M}$ the  norm on $\mathcal{H}^{s}$. Note  that  $\|Z\|_{s,M}^2 = 2\|z\|_{s,M}^2$.

Moreover, we can naturally extend to $\mathcal{H}^s$ the scalar product \eqref{scalarproScalare} as
\begin{equation}\label{scalarprod2x2}
	(Z,W)_{L^{2}\times L^{2}} := 2\mathrm{Re}(z,w)_{L^{2}} = \int_{\T^2} (z \overline{w} + \overline{z} w) \, dx\,,
	\qquad Z = \vect{z}{\overline{z}}\,,\; W = \vect{w}{\overline{w}} \in \mathcal{H}^{0}\,.
\end{equation}

Throughout the paper we shall perform some coordinate transformations $\Phi: \mathcal{H}^s \to \mathcal{H}^s$. 
Explicitly, for any $Z = (z, \bar{z}) \in \mathcal{H}^s$, a map preserving the real structure (i.e. it takes values in $\mathcal{H}^s$) necessarily takes the form 
\begin{equation}\label{eq:real_map}
    \Phi(Z)=\big( \Phi^{(1)}(z, \bar{z}), \Phi^{(2)}(z, \bar{z})  \big) = \big(\Phi^{(1)}(z, \bar{z}), \, \overline{\Phi^{(1)}(z, \bar{z})}\big) \, ,
\end{equation}
where $\Phi^{(1)}, \Phi^{(2)}$ denote the first and second component of the transformation. By \eqref{eq:real_map}, the entire map $\Phi$ is uniquely determined by $\Phi^{(1)}$. 
This structural property allows us to introduce a slight abuse of notation: we identify the map $\Phi$ with its first component $\Phi^{(1)}$ and write $\Phi(z)$ to denote the action of this first component on the variable. In other words, $\Phi(z)$ stands for $\Phi^{(1)}(z, \bar{z})$. Consequently, the action of the full transformation on the space $\mathcal{H}^s$ simply reads
$    \Phi(Z) = \big( \Phi(z), \, \overline{\Phi(z)} \big) \, $. Sometimes we shall use some linear coordinate transformations depending on a fixed variable $Z=(z,\bar{z})$, we shall adopt the same notation
   $  \Phi(Z)W = \big( \Phi(z)W, \, \overline{\Phi(z)W} \big) \, ,$ for any $W=(w,\bar{w})$.

%Da uniformare questo sotto con quello sopra\\
%Notice that in Section \ref{sec:energy}, the energy estimates in Proposition \ref{pr:ener_est} is derived for the vector $Z = (z, \overline{z})^T \in \mathcal{H}^s(\mathbb{T}^2)$ in order to account for the $2 \times 2$ matrix of the operators. However, due to the symmetry of the real Hamiltonian flow with respect to its conjugate, the dynamics are determined by the first component. Furthermore, the embedding $z \mapsto (z, \overline{z})^T$ guarantees that the Sobolev norm of the vector state is proportional to the norm of its scalar component, namely $\|Z\|_{s,M}^2 = 2\|z\|_{s,M}^2$. Therefore, the upper bounds obtained in Proposition \ref{pr:ener_est} translate to their scalar counterparts. From this point forward we work with the scalar variables $z, w, v^\lambda \in H^s(\mathbb{T}^2; \mathbb{C})$.       

We also set
\begin{equation}\label{matriciozze}
	E := \sm{1}{0}{0}{-1}\,,\quad
	\id := \sm{1}{0}{0}{1}\,,\quad \diag(b) := b\,\id\,, \quad b\in\mathbb{R}\,.
\end{equation}
We consider the symplectic form 
%{\color{orange} qui va messo il segno - oppure vanno cambiati i segni meno in $dH(u)[h]$ e in $\{F,G\}$. Per ora cambio segno qui}
\[
\Omega(z, w)=-2 \mathrm{Re} (\mathrm{i} z, w)_{L^2}
\]
and we define the vector field of a Hamiltonian $H$ as the unique $L^2$ function $X_H$ such that
%\red{Bisogna capire se mettere le doppie variabili o no} 
%\blu{(Felice) al massimo aggiungerei la versione in doppia variabile dopo
\[
d H(u)[h]=-\Omega(X_H(u), h).
\]
This gives $X_H(u)=- i\nabla_{\overline{u}}H(u,\overline{u})$ and, according to previous notation
\[ X_H(U):= \begin{pmatrix} -i\nabla_{\bar u}H(u,\bar u)\\ i\nabla_uH(u,\bar u) \end{pmatrix}, \qquad U=\begin{pmatrix}u\\ \bar u\end{pmatrix}. \]
This $2$-form also induces the Poisson structure
\[
\{F, G\}=i\sum_n(\partial_{u_n}F \partial_{\overline{u}_n}G- \partial_{\overline{u}_n}F\partial_{u_n}G )=\Omega(X_F, X_G).
\]

The quasilinear Schr\"odinger equation \eqref{NLS} possesses a Hamiltonian structure on the phase space $\mathcal{H}^s$. Indeed, considering the energy Hamiltonian
\begin{equation}\label{Ham:NLS}
	H(u,\overline{u}) = \int_{\T^2} |\nabla u |^2\,dx +\frac{1}{2} \int_{\T^2} |u |^4\,dx + \frac{1}{2}\int_{\T^2} \big|\nabla \big(h(|u|^2)\big) \big|^2 dx\,,
\end{equation}
\eqref{NLS} can be rewritten as
\[
\begin{cases}
	\mathrm{i}\partial_t u = \nabla_{\overline{u}} H(u, \overline{u}),\\[1ex]
	-\mathrm{i}\partial_t \overline{u} = \nabla_{{u}} H(u, \overline{u}),
\end{cases} 
\] 
where $\nabla_{\overline{u}} = \frac{1}{2}(\partial_{\Re(u)} + \mathrm{i}\partial_{\Im(u)})$ and $\nabla_{{u}} = \frac{1}{2}(\partial_{\Re(u)} - \mathrm{i}\partial_{\Im(u)})$
are the standard Wirtinger derivatives. 

In the Fourier expansion \eqref{complex-uU}, the Hamiltonian $H$ reads as 
%\red{H ha $2$ variabili ma le altre hamiltoniane no}
$$H(u,\overline{u})=H^{(2)}(u,\overline{u})+H^{(4)}(u,\overline{u})+H^{(\geq 8)}(u,\overline{u}),$$
where 
{\begin{equation}\label{ham}
\begin{aligned}
H^{(2)}(u,\overline{u})&=\sum_{n\in \Z^2} |n|^2 |u_n|^2, \\
H^{(4)}(u,\overline{u})&=\frac{1}{8\pi^2} \sum_{n_1-n_2+n_3-n_4=0} u_{n_1} \overline{u_{n_2}} u_{n_3} \overline{u_{n_4}},\\
H^{(\geq 8)}(u, \bar{u})&=\frac{1}{2}\int_{\T^2} \big|\nabla \big(h(|u|^2)\big) \big|^2 dx\,.
\end{aligned}
\end{equation}}
 We point out that the restriction $n_1-n_2+n_3-n_4=0$ for the monomials in $H^{(4)}$ arises from the conservation of momentum.

\section{Finite dimensional truncation and unstable trajectories}\label{sec:toy}
In this section we consider a suitable finite dimensional (in Fourier space) approximation of \eqref{NLS} which exhibits unstable trajectories. Our analysis is essentially based on the seminal paper \cite{CKSTT}, with some adaptations that will be discussed below. We start by recalling the main ideas.

%\red{The resonant terms associated to the degree four part of the Hamiltonian are given by the relation $\omega(n_1,n_2,n_3,n_4)=0$, where 
%$$\omega(n_1,n_2,n_3,n_4)=|n_1|^2-|n_2|^2+|n_3|^2-|n_4|^2.$$
%The condition $\omega(n_1,n_2,n_3,n_4)=0$, combined with the restriction $n_1-n_2+n_3-n_4=0$ arising from momentum conservation, readily implies that the $4$-wave resonant modes form (possibly degenerate) rectangles in $\Z^2$.

The effective finite dimensional system will be extracted from the restriction of the resonant part of the quartic Hamiltonian to a subspace of the form
\[
u_n=0 \qquad n\notin \Lambda,
\]
for an appropriately chosen finite set $\Lambda\subset \Z^2$.

We point out that the resonant part of $H^{(4)}$ is given by the sum of monomials $u_{n_1} \overline{u_{n_2}} u_{n_3} \overline{u_{n_4}}$ satisfying
\begin{equation}\label{eq:disp_mom}
\begin{cases}
\omega(n_1,n_2,n_3,n_4):=|n_1|^2-|n_2|^2+|n_3|^2-|n_4|^2=0,\\
n_1-n_2+n_3-n_4=0\,.
\end{cases}
\end{equation}
Any four wave-vectors $n_1, n_2, n_3, n_4$ fulfilling the above conditions form a (possibly degenerate) rectangle in the $\Z^2$-lattice. The non-degenerate four-wave resonances will serve as the building blocks of the set $\Lambda$, described below following \cite{CKSTT} and later refinements from \cite{GHHMP}.

\subsection{Lambda set}
The set $\Lambda$ is the union of $N$ subsets of $\Z^2$
\begin{equation}\label{eq:LLn}
\Lambda=\Lambda_1\cup \Lambda_2\cup\dots\cup\Lambda_N,
\end{equation}
where $N\in\mathbb{N}$ will be set later, and the cardinality of each set $\Lambda_j$ is $2^{N-1}$. In order to state the other properties of $\Lambda$, we need some definitions. First, for a fixed $s>1$, we define
\begin{equation}\label{weight}
	S_{\Lambda_j}:=\sum_{n\in \Lambda_j} |n|^{2s}, \qquad j=1, \dots, N.
\end{equation}
Moreover, we say that $(n_1, n_2, n_3, n_4)\in(\Z^2)^4$ is a \emph{nuclear family} if $n_1, n_3 \in \Lambda_i$ and $n_2, n_4 \in \Lambda_{i+1}$ for some $i=1, \ldots, N-1$ and they form a non-degenerate rectangle in $\mathbb{Z}^2$.

\begin{theorem}\label{thm:set}
	Fix $\eta>0$ small enough, and let $s>1$. For any $N$ large enough, there exists a set $\Lambda\subset\Z^2$, of the form \eqref{eq:LLn}, with the following properties. 
	\begin{itemize}
		\item[$(P1)$] (Closure): If ${n}_1, {n}_2, {n}_3\in\Lambda$ are three vertices of a rectangle then the fourth vertex belongs to $\Lambda$ too.
		\item[$(P2)$] (Existence and uniqueness of spouse and children): For each $1\le i\le N-1$  and every ${n}_1\in\Lambda_i$ there exists a unique spouse ${n}_3\in \Lambda_i$ and unique (up to trivial permutations) children ${n}_2, {n}_4\in\Lambda_{i+1}$ such that $({n}_1, {n}_2, {n}_3, {n}_4)$ is a nuclear family in $\Lambda$.
		\item[$(P3)$] (Existence and uniqueness of parents and sibling): For each $1\le i\le N-1$  and every ${n}_2\in\Lambda_{i+1}$ there exists a unique sibling ${n}_4\in\Lambda_{i+1}$ and unique (up to trivial permutations) parents ${n}_1, {n}_3\in\Lambda_{i}$ such that $({n}_1, {n}_2, {n}_3, {n}_4)$ is a nuclear family in $\Lambda$.
		\item[$(P4)$] (Non-degeneracy): A sibling of any mode $m$ is never equal to its spouse.
		\item[$(P5)$] (Faithfulness): Apart from nuclear families, $\Lambda$ contains no other non-degenerate parallelograms.
        \item[$(P6)$] (Bound on the modes): There exist  $\alpha_0>1$ and $C>0$, independent of $(\eta,N)$, such that for any $\alpha\geq \alpha_0$ and a suitable
    \begin{equation}\label{Rmax}
		\mathbf{R}\in [ \exp(\alpha^N),  \exp(2(1+\eta) \alpha^N)],
	\end{equation}
we have the bound
\begin{equation}\label{size:mode2}
\mathbf{R}_{\mathrm{min}}:= C^{-1} \mathbf{R}  \le |n|\le  C 3^N \mathbf{R}:=\mathbf{R}_{\mathrm{max}} \qquad \forall n\in \Lambda.
	\end{equation}
\item[$(P7)$] (Sobolev norm explosion): We have the estimate
\begin{equation}\label{eq:lower_bound}\frac{S_{\Lambda_{N-2}}}{S_{\Lambda_3}}> 2^{(N-6)(s-1)}\,, \qquad \frac{S_{\Lambda_j}}{S_{\Lambda_i}}\lesssim e^{s N} \qquad 1\le i\le j\le N.
	\end{equation}
    \end{itemize}
\end{theorem}
\begin{proof}
It follows by \cite[Theorem 7.3]{GHHMP}.
\end{proof}

%\begin{remark}
%Property (P5) is slightly more general than the corresponding stated in \cite{CKSTT}, still follows from the same construction, as observed in \cite{GHHMP}.
%\end{remark}
From now on we set
\begin{equation}\label{choose_M}
	M:=\mathbf{R}_{\mathrm{max}}.
\end{equation}

{

\subsection{Toy model} We now introduce the Hamiltonian that allows to construct approximate solutions of \eqref{NLS} exhibiting an arbitrarily large growth of high-order Sobolev norms over a finite time interval $[0,T]$. Its dynamics will be Fourier-supported, for all times, on the set $\Lambda$ provided by Theorem \ref{thm:set}. In the next sections, we will show that this Hamiltonian indeed provides a good approximation of the full NLS dynamics up to the instability time $T$.

To define it, we first split the quartic Hamiltonian of \eqref{NLS} as
\[
H^{(4)}=\sum_{k=0}^4 H^{(4,k)},
\]
where
\[
H^{(4,k)}(u,\bar u)
:=\frac{1}{8\pi^2}
\sum_{(n_1,n_2,n_3,n_4)\in\mathcal A_k}
u_{n_1}\overline{u_{n_2}}u_{n_3}\overline{u_{n_4}},
\]
\begin{equation}\label{Ak}
\begin{split}
\mathcal A_k=
\Bigl\{
(n_1,n_2,n_3,n_4)\in(\mathbb Z^2)^4:&\,
n_1-n_2+n_3-n_4=0,\text{exactly }k\text{ modes belong to }\mathbb Z^2\setminus\Lambda
\Bigr\}.
\end{split}
\end{equation}
We point out that $H^{(4,1)}$ couples the Fourier modes in $\Lambda$ with those outside $\Lambda$. In Section \ref{sec:wb} we will construct a change of coordinates eliminating $H^{(4,1)}$, up to a higher degree reminder. This naturally leads to consider the truncated quartic Hamiltonian
\begin{equation}\label{calH4}
\mathcal H^{(4)}
=
H^{(4,0)}+H^{(4,\geq2)},
\end{equation}
where
\[
H^{(4,\geq2)}
=
H^{(4,2)}+H^{(4,3)}+H^{(4,4)}.
\]
The effective Hamiltonian studied in this section is therefore
\[
H^{(2)}+\mathcal H^{(4)}.
\]

\begin{theorem}\label{thm:toy_model_scaled}
	Let $\lambda > 0$ and fix $\gamma_0 > 0$ large enough. There exist constants $\sigma,\mathbb{K} > 0$ (independent of $\gamma_0$) such that, for any $N$ large enough, the following holds: there exists a time $T_0$ satisfying
	\begin{equation}\label{T0}
		0 < T_0 \leq \mathbb{K} \gamma_0 N^2,
	\end{equation} 
	and a solution 
\begin{equation}\label{eq:scaled_sol}
v^\lambda(t,x)=\frac{1}{2\pi}\sum_{n\in\Lambda}v^\lambda_n(t)e^{\mathrm{i} n\cdot x}   
	\end{equation}
    of the Hamiltonian $H^{(2)}+\mathcal{H}^{(4)}$ defined for $t \in [0, T]$, with 
	\begin{equation}\label{def:T}
		T=\lambda^2 T_0,
	\end{equation}
which satisfies the energy transfer properties
	\begin{align}
		\label{b1}|v_n^\lambda(0)| &> \lambda^{-1}(1 - e^{-\gamma_0 N \sigma})  &&n\in \Lambda_3, \\\label{b1<}
        |v_n^\lambda(0)| &< \lambda^{-1} e^{-\gamma_0 N \sigma}  &&n\in \Lambda_j \,\,\mathrm{ for }\,\, j \neq 3,\\
	\label{b2}	|v_n^\lambda(T)| &> \lambda^{-1}(1 - e^{-\gamma_0 N \sigma})  &&n\in \Lambda_{N-2}, \\\label{b2<}
    |v_n^\lambda(T)| &< \lambda^{-1} e^{-\gamma_0 N \sigma}  &&n\in \Lambda_j\,\, \text{ for }\,\, j \neq N-2.
    \end{align}
	Moreover, we have the following bounds
\begin{equation}\label{eq:l1_bound_scaled}
		\sup_{t \in [0, T]} \|v^\lambda(t)\|_{\ell^1} \le N 2^{N-1} \lambda^{-1}.
	\end{equation}
	\begin{equation}\label{eq:Ms_bound}
		\sup_{t \in [0, T]} \|v^\lambda(t)\|_{s,M} \lesssim  \sqrt{N2^{N-1}} \lambda^{-1} M^s.
	\end{equation}
\end{theorem}

\begin{proof}
We look for solutions which are Fourier supported on a set $\Lambda$ as in Theorem \ref{thm:set}. To this aim, we consider the Hamiltonian system $H^{(2)}+\mathcal{H}^{(4)}$ restricted to the subspace
\begin{equation}
	\mathcal{U}_{\Lambda} = \big\{ u \in \ell^1(\mathbb{Z}^2; \mathbb{C}) : u_n = 0 \text{ if } n \notin \Lambda \big\}.
\end{equation}
Decomposing the quadratic part of the NLS Hamiltonian as
\[
H^{(2)}
:=
H^{(2,0)}+H^{(2,2)},
\]
where
\[
H^{(2,0)}(u,\bar u)
:=
\sum_{n\in\Lambda}|n|^2|u_n|^2,
\qquad
H^{(2,2)}(u,\bar u)
=
\sum_{n\notin\Lambda}|n|^2|u_n|^2,
\]
we get that the restricted dynamics is governed by the Hamiltonian $H^{(2, 0)}+H^{(4, 0)}$, since the vector field of $H^{(2, 2)}+H^{(4, \geq 2)}$ vanishes on $\mathcal{U}_{\Lambda}$.

Observe moreover that $H^{(2, 0)}+H^{(4, 0)}$ Poisson commutes with the mass Hamiltonian. Hence, if $u(t)$ is a solution to $H^{(2, 0)}+H^{(4, 0)}$ then
\[
z(t):=e^{-\mathrm{i}t\left(\Delta- \frac{1}{2\pi^2} \| \Pi_{\Lambda} u(0)\|_{L^2}^2\right)} u(t)
\]
formally solves the Toy Model system associated with the Hamiltonian function
\[
H_{TM}(z, \bar{z})=-\frac{1}{8\pi^2}\sum_{n\in \Lambda} |z_n|^4+\frac{1}{8\pi^2}\sum_{\substack{n_1-n_2+n_3-n_4=0\\ n_k\in \Lambda,\\ n_1\neq n_2, n_4}} z_{n_1} \overline{z_{n_2}} z_{n_3} \overline{z_{n_4}}.
\]
Arguing as in \cite[Corollary 3.2]{GK2015} (see also the discussion in \cite[Section 2.2]{CKSTT}), we can assume that $\Lambda$ satisfies the so-called intragenerational equality, namely that the subspace
\begin{equation}
\widetilde{\mathcal{U}}_{\Lambda} = \big\{ c \in \mathcal{U}_{\Lambda} : c_n = c_{n'} \quad \forall n, n' \in \Lambda_j, \ j=1, \dots, N \big\}
\end{equation}
is left invariant by the flow of $H_{TM}$. Then, setting 
\[
b_j(t)=z_n(t) \qquad \forall n\in \Lambda_j, \quad j=1, \dots, N,
\]
the Toy model system reduces, after an harmless time re-parametrization, to the following nearest-neighbor system of ODEs
\begin{equation}\label{eq:toy_ode}
	\dot{b}_j = -\mathrm{i} |b_j|^2 b_j + 2\mathrm{i} \overline{b}_j (b_{j-1}^2 + b_{j+1}^2), \quad j=1,\dots, N.
\end{equation}
Theorem 3 in \cite{GK2015} guarantees the existence of $\sigma>0$ and an orbit $(b_j(t))_{j=1, \dots, N}$ of \eqref{eq:toy_ode} satisfying
\begin{align}
		\label{b1'}|b_3(0)| > (1 - e^{-\gamma_0 N \sigma}), \quad &|b_j(0)| < e^{-\gamma_0 N \sigma} \text{ for } j \neq 3, \\
	\label{b2'}	|b_{N-2}(T_0)| >(1 - e^{-\gamma_0 N \sigma}), \quad &|b_j(T_0)| <  e^{-\gamma_0 N \sigma} \text{ for } j \neq N-2\\
	\label{b3'}	 \sum_{j=1}^N |b_j(t)|^2=1, \quad &\forall t\in [0, T_0]
	\end{align}
    with $T_0$ as in \eqref{T0}. By exploiting the scaling associated with the cubic equation \eqref{eq:toy_ode}, we can consider the scaled solution
    \begin{equation}\label{blj}
    b_j^{\lambda}(t)=\lambda^{-1} b_j(\lambda^{-2} t), \qquad t\in [0, T], \qquad j=1, \dots, N.
    \end{equation}
 In particular, the function
\begin{equation}\label{eq:scaled_sol-2}
v^\lambda(t,x)=\frac{1}{2\pi}\sum_{n\in\Z^2} v^\lambda_n(t)e^{\mathrm{i} n\cdot x},   
	\end{equation}
    with 
    \begin{equation}\label{vlj}
v_n^\lambda(t) = \begin{cases}
b_j^\lambda(t) e^{-\mathrm{i} t\big(|n|^2 + (2\pi^2)^{-1} \| v^{\lambda}(0)\|_{L^2}^2\big)} & \text{if } n \in \Lambda_j, \\
			0 & \text{if } n \notin \Lambda,
		\end{cases}
\end{equation}
is a solution to $H^{(2)}+\mathcal{H}^{(4)}$. Estimates \eqref{b1}-\eqref{b2<} follows by \eqref{b1'}-\eqref{b2'}. 

In order to prove the remaining bounds, we first observe that
  \begin{equation}\label{b3}	\sup_{t\in[0,T]}|v_n^\lambda(t)|\leq\lambda^{-1},\qquad\, n\in \Lambda,
    \end{equation}
as follows from \eqref{b3'}, \eqref{blj} and \eqref{vlj}.

    Using \eqref{b3} and the fact that each generation $\Lambda_j$ has cardinality $2^{N-1}$, we obtain
\[
\| v^{\lambda}(t)\|_{\ell^1}=\sum_{j=1}^N \sum_{n \in \Lambda_j} |v_n^{\lambda}(t)|\le \lambda^{-1} 2^{N-1} N,
\]
which proves the bound \eqref{eq:l1_bound_scaled}.

Recalling the definition \eqref{norma_sM} of the $\|\cdot\|_{s,M}$-norm and \eqref{size:mode2}, we similarly have
	\begin{equation*}
		\|v^\lambda(t)\|_{s,M}^2 = \sum_{j=1}^N \sum_{n \in \Lambda_j} |v_n^\lambda(t)|^2 (M^2 + |n|^2)^s\le N 2^{N-1} \lambda^{-2}({M}^2 + \mathbf{R}_{\mathrm{max}}^2)^s.
	\end{equation*}
Since we chose $M = \mathbf{R}_{\mathrm{max}}$, the above estimate readily implies \eqref{eq:Ms_bound}, concluding the proof.
\end{proof}

From now on, we set
\begin{equation}\label{choose_lambda}
\lambda=\mathbf{R}^{s_{\ast}},
\end{equation}
where $s_*$ is a parameter that will be chosen later.

As a consequence of the above result, we can show that the Hamiltonian $H^{(2)}+\mathcal{H}^{(4)}$ admits strong unstable solutions in the $H^s$-topology, for $s$ sufficiently large.

\begin{cor}\label{co:unstable_toy}
	Let $s \geq 7$. Fix $\mu > 0$ sufficiently small and $\mathcal{K} > 0$ sufficiently large. There exist $\gamma_0>0$ and $s_*$ of the form
\begin{equation}\label{forma_star}s_{*}=s+\delta,\quad\delta>0,\,\delta=o(1)\mbox { as }N\to\infty,
\end{equation}
    such that the solution $v^{\lambda}$ to $H^{(2)}+\mathcal{H}^{(4)}$ constructed in Theorem \ref{thm:toy_model_scaled} satisfies, for $N$ large enough, 
	\begin{equation}\label{toy:high_growth}
		\| v^{\lambda}(0) \|_{H^s} \sim \mu, \qquad \| v^{\lambda}(T) \|_{H^s} > \mathcal{K},
	\end{equation}
	\begin{equation}\label{toy:low_smallness}
		\sup_{t \in [0, T]} \| v^\lambda(t) \|_{s_0, M}=o(1)\mbox{ as }N\to\infty,\qquad \forall s_0\le s-1.
	\end{equation}
\end{cor}

\begin{proof}
We start by proving the first estimate in \eqref{toy:high_growth}. By \eqref{b1} and \eqref{eq:lower_bound} we have that
\begin{align*}
\| v^{\lambda} (0)\|_{H^s}^2 &\le \lambda^{-2} S_{\Lambda_3}+\lambda^{-2} e^{-2\gamma_0 N \sigma} \sum_{j\neq 3} S_{\Lambda_j}\le \lambda^{-2} S_{\Lambda_3} \left( 1+e^{-2\gamma_0 N \sigma}\sum_{j\neq 3} \frac{S_{\Lambda_j}}{S_{\Lambda_3}}  \right)\\
&\le \lambda^{-2} S_{\Lambda_3} \left( 1+C(N-1)e^{-2\gamma_0 N \sigma} e^{sN}  \right).
\end{align*}
By choosing, for instance,
\[
\gamma_0=\frac{s}{2\sigma}+1,
\]
and taking $N$ large enough, we get
\[
\| v^{\lambda} (0)\|_{H^s}^2\le \frac{3}{2} \lambda^{-2} S_{\Lambda_3}.
\]
Next, we need a lower bound for the $H^s$-norm of the initial datum. Taking $N$ large enough we have
\[
\| v^{\lambda}(0)\|_{H^s}^2\geq  \sum_{n\in \Lambda_3} |v_n^{\lambda}(0)|^2 |n|^{2s}\geq \lambda^{-2} (1-e^{-\gamma_0 N \sigma})^2 S_{\Lambda_3}\geq \frac{1}{2} \lambda^{-2} S_{\Lambda_3}.
\]
Hence, the first estimate in \eqref{toy:high_growth} follows by choosing $s_*$ such that
\[
	\lambda^{-2} S_{\Lambda_3} \sim \mu^2.
\]
By \eqref{size:mode2}
\[
	2^{N-1} C^{-2s} \mathbf{R}^{2s} \le S_{\Lambda_3} \le 2^{N-1} 3^{2Ns} C^{2s} \mathbf{R}^{2s}.
\]
Therefore, we need to choose $\lambda$ such that
\begin{equation}\label{last}
	a\mu^{-2} 2^{N-1} C^{-2s} \mathbf{R}^{2s} \le \lambda^2 \le b\mu^{-2} 2^{N-1} 3^{2Ns} C^{2s} \mathbf{R}^{2s}
\end{equation}
for some universal constants $a, b>0$. We observe that for $N$ large enough
\[
	a C^{-2s} < b\, 3^{2Ns} C^{2s} ,
\]
thus the interval \eqref{last} is non-empty.
We write 
\begin{equation}\label{s+delta}
	s_* = s + \delta
\end{equation}
where $\delta=\delta(N, \mu)>0$. Considering that $\lambda=\mathbf{R}^{s_*}$, the condition \eqref{last} reads
\[
	a\mu^{-2} 2^{N-1} C^{-2s} \le \mathbf{R}^{2\delta} \le b\mu^{-2} 2^{N-1} 3^{2Ns} C^{2s},
\]
which in turn is equivalent to
\begin{equation}\label{delta}
	\frac{1}{2 \log(\mathbf{R})} \log(a\mu^{-2} 2^{N-1} C^{-2s}) \le \delta \le \frac{1}{2 \log(\mathbf{R})} \log(b\mu^{-2} 2^{N-1} 3^{2Ns} C^{2s}).
\end{equation}
Choosing $\delta$ in the above interval we have the desired $s_*$. 
Observe that
\[
	\lim_{N\to +\infty} \delta=0.
\]
Concerning the second bound in \eqref{toy:high_growth} we have, by \eqref{b2} and \eqref{eq:lower_bound},
\[
\frac{\| v^{\lambda}(T)\|_{H^s}^2}{\| v^{\lambda}(0)\|_{H^s}^2}\geq \frac{\lambda^{-2} (1-e^{-\gamma_0 N \sigma})^2 S_{\Lambda_{N-2}}}{\frac{3}{2}\lambda^{-2} S_{\Lambda_3}}\geq \frac{1}{3} \frac{S_{\Lambda_{N-2}}}{S_{\Lambda_3}}>\frac{1}{3} 2^{(N-6)(s-1)}.
\]
Recall moreover that $\|v^{\lambda}(0)\|_{H^s}\sim\mu$. Then, by taking 
\begin{equation}\label{choice:N}
    N\geq \frac{\log(3 C \mathcal{K}^2 \mu^{-2})}{(s-1)\log(2)}+7
\end{equation}
for a suitable universal constant $C>0$,
we get
\[
\| v^{\lambda}(T)\|_{H^s}>\mathcal{K},
\]
as desired. Finally, by \eqref{eq:Ms_bound}, 
\[
\| v^{\lambda} \|_{s_0, M}\le \sqrt{N2^{N-1}} \lambda^{-1} M^{s_0}\lesssim 3^{N s_0} \sqrt{2^{N-1}} \mathbf{R}^{s_0-s_*}.
\]
Using $s_* > s_0+1$ and \eqref{Rmax}, we get $\mathbf{R}^{s_0-s_*} \le \mathbf{R}^{-1} \le \exp(-\alpha^N)$. Thus, the right-hand side of the above inequality is bounded by $C 3^{N s_0} 2^{N/2} \exp(-\alpha^N)$, which converges to zero as $N \to \infty$.
\end{proof}

\begin{remark}
The scaling parameter $\lambda$ is used to ensure the smallness of the solution in $\ell^1$. In our context, the nonlinearity involves derivatives, which implies a potential loss of regularity. Therefore, $\lambda$ cannot be chosen freely but must be related to the frequency cutoff $M$ to balance the growth of the vector field estimates in the $\|\cdot\|_{s,M}$-norm.
\end{remark}

\section{Weak Birkhoff normal form}\label{sec:wb}
In this section, we use a normal form argument to eliminate the first order non-resonant nonlinear
interactions that couple the dynamics of the modes in $\Lambda$ with the remaining ones. This is achieved by constructing a suitable system of coordinates which shows that the
NLS Hamiltonian differs from the Hamiltonian $H^{(2)}+\mathcal{H}^{(4)}$ only by a
remainder term, which will be shown to be perturbative in our analysis. In turn, this will allow us to show the existence of NLS solutions which behave as the unstable trajectory $v^{\lambda}$ constructed in Corollary \ref{co:unstable_toy}.

\smallskip

The key point is to perform a weak normalization of the Hamiltonian, in order to obtain
a change of coordinates that equals the identity up to smoothing operators. Maps of this
type preserve the paradifferential structure of the original equation, and this will be fundamental to manage the effects of the remainder term, which carries derivatives of order two, over long time intervals.

The main result of this section is the following (we recall that $M$ is fixed, and given by \eqref{choose_M}).

\begin{theorem}\label{thm:wbnf}
Fix $s>1$. There exists $r_0>0$ with the following property: for any given $r\in (0, r_0)$, we can find a symplectic, invertible change of coordinates 
$$\Phi\colon B_{\ell^1}(r)\cap H^{s}\to B_{\ell^1}(2r)\cap H^{s},\quad v=\Phi(z),$$
such that 
\begin{equation}\label{ham:bnf}
{H}\circ \Phi = H^{(2)}+\mathcal{H}^{(4)} + R^{(\geqslant 6)} + {H}^{(\geq 8)}\circ \Phi,
\end{equation}
where 
\begin{equation}\label{bound:rem}
	\| X_{R^{(\geqslant 6)}}(z)\|_{s,M}\lesssim  \| z \|_{\ell^1}^4 \| z \|_{s,M} \qquad \forall z\in B_{\ell^1}(r)\cap H^{s}.
\end{equation}
Moreover, for all $z\in B_{\ell^1}(r)\cap H^{s}$, we have 
\begin{align}\label{close1}
\|\Phi(z)-z\|_{\ell^1}&\lesssim  \| z \|^3_{\ell^1},\\\label{close11} \|\Phi(z)-z\|_{s,M}&\lesssim  \| z \|^2_{\ell^1} \| z \|_{s,M},
\\\label{Phi:equiv}
\| \Phi(z) \|_{s, M}&\lesssim \| z \|_{s, M},
\\ \label{close2}
\| d \Phi(z)[h]-h \|_{\ell^1}&\lesssim  \| z \|^2_{\ell^1} \| h\|_{\ell^1},\\ \label{close3}
\| d \Phi(z)[h]-h \|_{s,M}&\lesssim \| z\|_{\ell^1}^2 \| h \|_{s,M}\,.
\end{align}
Finally, 
\begin{itemize}
\item[(i)] the maps $\Phi-\id$ and $d \Phi(z)-\id$ are of finite rank, more precisely
\begin{equation}\label{treM}
\Phi(z)-z=\Pi_{\le 3 M} [\Phi(z)-z], \qquad d\Phi(z)[h]-h=\Pi_{\le 3M} [d\Phi(z)[h]-h]
\end{equation}
for all $z, h\in B_{\ell^1}(r)\cap H^s$.
\item[(ii)] ${H}\circ \Phi$ Poisson commutes with the mass Hamiltonian $\mathcal{M}$.
\end{itemize}
The same properties and estimates  of $\Phi$ hold for the inverse $\Phi^{-1}$.
\end{theorem} 
The rest of this section is devoted to the proof of the above theorem.

\subsection*{Existence of a local flow and estimates}
We consider the generating function
\[
F(z)=\sum_{\substack{n_1-n_2+n_3-n_4=0\\ (n_1, n_2, n_3, n_4)\in \mathcal{A}_1}} F_{n_1 n_2 n_3 n_4}\,z_{n_1} \overline{z_{n_2}} z_{n_3} \overline{z_{n_4}},
\]
where $\mathcal{A}_1$ is defined in \eqref{Ak} and we set 
\[
F_{n_1 n_2 n_3 n_4}= \dfrac{\mathrm{i}}{8\pi^2\omega(n_1,n_2,n_3,n_4)}
\]
We claim that the above expression is well-defined, and moreover
\begin{equation}\label{upper_uno}
\sup | F_{n_1 n_2 n_3 n_4} | \le 1.
\end{equation}
Indeed, if $n_1-n_2+n_3-n_4=0$ and $(n_1,n_2,n_3,n_4)\in \mathcal{A}_1$, then by the closure property $(P1)$ of the $\Lambda$ set, the configuration cannot form a rectangle, meaning $\omega(n_1,n_2,n_3,n_4)\in\Z\setminus\{0\}$, which yields the bound \eqref{upper_uno}. A direct computation shows that the Hamiltonian $F$ satisfies the following homological equation
\begin{equation}\label{homolog}
\{F, H^{(2)} \} + H^{(4, 1)} = 0.
\end{equation}
We will construct $\Phi$ as the time-one flow map of $F$. In the following, we show that this flow is indeed well-defined. First, we provide suitable estimates for the vector field $X_F$ on $\ell^1$ and $H^{s}$. Set
\begin{equation}\label{p}
p_n = |z_n| \qquad n\in \Z^2.
\end{equation}
By Young's inequality,
\begin{equation}\label{XF1}
\| X_F(z) \|_{\ell^1} \lesssim \sum_{n\in \Z^2} \sum_{n_1-n_2+n_3=n} |z_{n_1}| \,|{z_{n_2}}|\, |z_{n_3} | \lesssim \sum_{n\in \Z^2} (p*p*p)_n \lesssim \| p*p*p \|_{\ell^1} \lesssim \| z \|_{\ell^1}^3.
\end{equation}
For the Sobolev norm, we have
\begin{align*}
\| X_F(z) \|^2_{s, M} &\lesssim \sum_{n\in \Z^2} \Big| \sum_{n_1-n_2+n_3=n} z_{n_1} \overline{z_{n_2}} z_{n_3} \Big|^2 \langle n \rangle_{M}^{2s} \lesssim \sum_{n\in \Z^2} \Big( \sum_{n_1-n_2+n_3=n} |z_{n_1} |\,|{z_{n_2}}|\, |z_{n_3}|  \langle n \rangle_{M}^{s} \Big)^2 \\
&\lesssim \sum_{n\in \Z^2} \Big( \sum_{\substack{n_1-n_2+n_3=n\\ |n_1|\geq \max\{ |n_2|, |n_3| \}}} |z_{n_1}| \,|{z_{n_2}}| \,|z_{n_3}| \langle n_1 \rangle_{M}^{s} \\&\quad+ \sum_{\substack{n_1-n_2+n_3=n\\ |n_2|\geq \max\{ |n_1|, |n_3| \}}} |z_{n_1}| \,|{z_{n_2}}| \,|z_{n_3}| \langle n_2 \rangle_{M}^{s}\\
&\quad + \sum_{\substack{n_1-n_2+n_3=n\\ |n_3|\geq \max\{ |n_1|, |n_2| \}}} |z_{n_1}| \,|{z_{n_2}}| \,|z_{n_3}| \langle n_3 \rangle_{M}^{s} \Big)^2,
\end{align*}
and applying again Young's inequality we obtain
\begin{equation}\label{XFs}
\| X_F(z) \|^2_{s, M} \lesssim \| D_M^s p* p*p \|_{\ell^2}^2 \lesssim \|D_M^s p \|_{\ell^2}^2 \| p\|^4_{\ell^1} \lesssim \| z \|^4_{\ell^1}\| z \|^2_{s, M}.
\end{equation}
We fix $z\in B_{\ell^1}(r)$, for some $r>0$ to be chosen, and consider the operator
\begin{equation}\label{opL}
L_z [w](t) = z + \int_0^t X_F(w(\tau))\,d\tau
\end{equation}
defined on 
\[
Y_r = \Big\{ w \in \mathcal{C}([0, 1] ; \ell^1) : \sup_{t\in [0, 1]} \| w(t) \|_{\ell^1} \le 2 \|z\|_{\ell^1} \Big\}.
\]
We have that, for all $w\in Y_r$ and $t\in[0,1]$,
\begin{equation}\label{Ll13}
\| L_z[w](t) - z\|_{\ell^1} \le \int_0^t \| X_F(w(\tau))\|_{\ell^1}\,d\tau \stackrel{\eqref{XF1}}{\lesssim} \int_0^t \| w(\tau) \|^3_{\ell^1}\,d\tau \lesssim  \| z \|_{\ell^1}^3.
\end{equation}
In particular, 
$$\| L_z[w](t) - z\|_{\ell^1}\lesssim r^2\|z\|_{\ell^1},$$ and by choosing $r>0$ small enough we guarantee that $L_z$ maps $Y_r$ into itself. Now let us prove that $L_z$ is a contraction on $Y_r$. Indeed, for all $w_1, w_2\in Y_r$ and $t\in[0,1]$,
\begin{align*}
\| L_z[w_1](t) - L_z[w_2](t) \|_{\ell^1} &\le \int_0^t \| X_F(w_1(\tau)) - X_F(w_2(\tau)) \|_{\ell^1}\,d\tau\\
&\lesssim \int_0^t (\| w_1(\tau) \|_{\ell^1} + \| w_2(\tau)\|_{\ell^1})^2 \|w_1(\tau) - w_2(\tau)\|_{\ell^1}  \,d\tau\\
&\lesssim r^2 t \|w_1(\tau) - w_2(\tau)\|_{\ell^1}.
\end{align*}
Therefore, choosing $r>0$ possibly smaller, we obtain that $L_z$ is a contraction on $Y_r$. By the Banach contraction theorem, there exists a unique fixed point $w_z\in Y_r$ of the map \eqref{opL}. 

Setting $\Phi^t(z):=w_z(t,\cdot)$ for all $z\in B_{\ell^1}(r)$ and $t\in [0, 1]$, we obtain a well-defined local flow for $F$, such that $\Phi^t\colon B_{\ell^1}(r)\to B_{\ell^1}(2r)$. Note that, by construction,
\begin{align}\label{eq:integral_identity_Phi}
&\Phi^t(z) = z + \int_0^t X_F(\Phi^{\tau}(z))\,d\tau  &&\forall t\in [0, 1], \forall z\in B_{\ell^1}(r)\\
\label{claim1}
&\| \Phi^t(z) \|_{\ell^1} \le 2 \| z \|_{\ell^1}  &&\forall t\in [0, 1],\;\forall z\in B_{\ell^1}(r).
\end{align}
We define the symplectic change of coordinates $\Phi$ by
\[
\Phi:=\Phi^{t=1}.
\]
Now we show that the local flow $\Phi^t$ propagates the Sobolev regularity. Given $z\in B_{\ell^1}(r)\cap H^s$ and $t\in [0, 1]$ we have
\begin{equation}\label{sky2}
\begin{aligned}
\| \Phi^t(z)-z\|_{s, M} \le \int_0^t \| X_F(\Phi^{\tau}(z)) \|_{s, M}\,d\tau &\stackrel{\eqref{XFs}}{\lesssim}  \int_0^t \| \Phi^{\tau}(z) \|_{\ell^1}^2 \| \Phi^{\tau}(z) \|_{s, M}\,d\tau\\
&\, \lesssim \|z\|_{\ell^1}^2 \| \Phi^{\tau}(z) \|_{s, M}\,.
\end{aligned}
\end{equation}
Therefore, $\| \Phi^t(z)-z\|_{s, M}\lesssim r^2\| \Phi^{\tau}(z) \|_{s, M}$, and by choosing $r>0$ possibly smaller we deduce
\begin{equation}\label{claim2}
\|\Phi^t(z)\|_{s, M} \le 2  \| z \|_{s, M}\,.
\end{equation}
Finally, estimates \eqref{close1}, \eqref{close11} and \eqref{Phi:equiv} follows respectively from \eqref{Ll13}, \eqref{sky2} and \eqref{claim2}.

\subsection*{Estimates on the differential}
Let $z\in B_{\ell^1}(r)$. Since $\Phi^t(z)\in\mathcal{C}([0,1],\ell^1)$, we deduce from \eqref{eq:integral_identity_Phi} that for every fixed $t\in[0,1]$ the map $z\mapsto\Phi^t(z)$ is Frech\'et differentiable in $\ell^1$, and by setting $\Psi_z^t(h) = d\Phi^t(z) [h]$ we have the identity
\begin{equation}\label{int:diff}
\Psi_z^t(h)  = h+ \int_0^t d X_F(\Phi^{\tau}(z))[\Psi_z^{\tau}(h)]\,d\tau.
\end{equation}
Let us prove estimates \eqref{close2} and \eqref{close3}. To this end, we first show some bounds on the vector field $X_F$. 
\begin{equation}\label{bound:dXF}
\begin{aligned}
\|dX_F(z)[h]\|_{\ell^1} &\lesssim \sum_{n\in \Z^2} \Big| \sum_{n_1-n_2+n_3=n} h_{n_1} \overline{z_{n_2}} z_{n_3} + z_{n_1} \overline{h_{n_2}} z_{n_3} + z_{n_1} \overline{z_{n_2}} h_{n_3} \Big|\\
&\lesssim \sum_{n\in \Z^2} \sum_{n_1-n_2+n_3=n} \big(|z_{n_1}| |z_{n_2}| |h_{n_3}| + |z_{n_1}| |h_{n_2}| |z_{n_3}| + |h_{n_1}| |z_{n_2}| |z_{n_3}|\big)\\
&\lesssim \| p*p*\tilde{h}\|_{\ell^1} \lesssim \| z\|_{\ell^1}^2 \| h \|_{\ell^1},
\end{aligned}
\end{equation}
where we denoted by $\tilde{h}=(|h_n|)_{n\in \Z^2}$ and $p=(p_n)_{n\in \Z^2}$ is defined in \eqref{p}. 
Assuming in addition $z\in H^s$, and reasoning as above with Young's inequality, we also obtain
\begin{equation}\label{blu}
\| d X_F(z)[h] \|_{s, M} \lesssim M^{s} \| z \|_{\ell^1}^2 \| h \|_{L^2} \lesssim \| z \|_{\ell^1}^2 \| h \|_{s, M}.
\end{equation}
Then we have
\begin{equation*}
\begin{split}
\|\Psi_z^{t}(h)-h\|_{\ell^1}&=\Big\|\int_0^t dX_F(\Phi^{\tau}(z))[\Psi_z^{\tau}(h)]\Big\|_{\ell^1}\\
&\lesssim \|\Phi^{\tau}(z)\|_{\ell^1}^2\|\Psi_z^{\tau}(h)\|_{\ell^1}\lesssim \|z\|_{\ell^1}^2\big(\|\Psi_z^{\tau}(h)-h\|_{\ell^1}+\|h\|_{\ell^1}\big)\\
&\lesssim r^2\|\Psi_z^{\tau}(h)-h\|_{\ell^1}+\|z\|_{\ell^1}^2\|h\|_{\ell^1},
\end{split}
\end{equation*}
where we used identity \eqref{int:diff} in the first step, \eqref{bound:dXF} in the second step and \eqref{claim1} in the third step. By choosing $r$ possibly smaller, the bound above yields estimate \eqref{close2}.

Similarly, assuming $z\in H^s$ and using \eqref{int:diff}, \eqref{blu} and \eqref{claim1} we obtain
\begin{equation*}
\begin{split}
\|\Psi_z^{t}(h)-h\|_{s,M}&=\Big\|\int_0^t dX_F(\Phi^{\tau}(z))[\Psi_z^{\tau}(h)]\Big\|_{s,M}\\
&\lesssim \|\Phi^{\tau}(z)\|_{\ell^1}^2\|\Psi_z^{\tau}(h)\|_{s,M}\lesssim \|z\|_{\ell^1}^2\big(\|\Psi_z^{\tau}(h)-h\|_{s,M}+\|h\|_{s,M}\big)\\
&\lesssim r^2\|\Psi_z^{\tau}(h)-h\|_{s,M}+\|z\|_{\ell^1}^2\|h\|_{s,M},
\end{split}
\end{equation*}
and by choosing $r$ possibly smaller we deduce estimate \eqref{close3}.

\subsection*{Lie Series Expansion}
By the Lie series expansion and \eqref{homolog}, the new Hamiltonian is given by
\begin{align*}
\mathcal{H} \circ \Phi &= H^{(2)} + H^{(4)} + \{ F, H^{(2)} \} + R^{(\geq 6)} + H^{(\geq 8)}\circ \Phi\\
&= H^{(2)} + H^{(4, 0)} + H^{(4, \geq 2)} + \underbrace{(\{ F, H^{(2)} \} + H^{(4, 1)})}_{=0} + R^{(\geq 6)} + H^{(\geq 8)}\circ\Phi,
\end{align*}
where 
\[
R^{(\geq 6)}(z) = \int_0^1 \{ F, \mathcal{H}^{(4)} \}\circ \Phi^{\tau}(z)\,d\tau +  \int_0^1 \tau\, \{ F, H^{(4, 1)} \}\circ \Phi^{\tau}(z)\,d\tau\,.
\]
Then the associated vector field is
\begin{equation}\label{versei}
\begin{split}
X_{R^{(\geq 6)}}(z) &= \int_0^1 [d \Phi^{\tau}(z)]^{-1} X_{\{ F, \mathcal{H}^{(4)} \}}\circ \Phi^{\tau}(z)\,d\tau \\
&+ \int_0^1 \tau [d \Phi^{\tau}(z)]^{-1} X_{ \{ F, H^{(4, 1)} \}}\circ \Phi^{\tau}(z)\,d\tau.
\end{split}
\end{equation}
By Young's inequality, for all $z\in H^s$,
\begin{equation}\label{casei}
\|X_{\{ F, \mathcal{H}^{(4)}\}}(z)\|_{s, M} \lesssim \|z\|^4_{\ell^1} \| z\|_{s, M}, \qquad \|X_{\{ F, {H}^{(4, 1)}\}}(z)\|_{s, M} \lesssim \|z\|^4_{\ell^1} \|z\|_{s, M}.
\end{equation}
The bound \eqref{bound:rem} then follows by \eqref{close2}, \eqref{close3}, \eqref{claim1}, \eqref{claim2}, \eqref{versei} and \eqref{casei}.

\subsection*{Finite rank maps}
Given $n_1, n_2, n_3\in \Lambda$, we have
\[
|n_1-n_2+n_3|\le |n_1|+|n_2|+|n_3|\le 3 M,
\]
since $\operatorname{max}_{n\in\Lambda}|n|\leqslant M$ in view of \eqref{size:mode2}. Thus, if $(n_1, n_2, n_3, n_4)\in\mathcal{A}_1$, each $|n_i|\le 3 M$ and therefore
\[
\Pi_{\le 3M} X_F(z) = \Pi_{\le 3 M} X_F( \Pi_{\le 3M} z) = X_F(z) \qquad \forall z\in B_{\ell^1}(r)\cap H^s.
\]
Hence,
\[
\Pi_{\le 3M} (\Phi^t(z)-z) = \int_0^t \Pi_{\le 3 M} X_F(\Phi^{\tau}(z)) d\tau = \int_0^t X_F(\Phi^{\tau}(z)) d\tau = \Phi^t(z)-z,
\]
and
\[
\Pi_{\le 3 M} d X_F(z) [h] = d \Big( \Pi_{\le 3 M} X_F(z) \Big) [h] = d X_F(z) [h].
\]
In view of \eqref{int:diff} we also have that
\[
\Pi_{\le 3 M} (\Psi_z^t(h)-h) = \Psi_z^t(h)-h.
\]
Thus \eqref{treM} is proved, and in particular $\Phi-\id$ and $d\Phi(z)-\id$ are finite rank maps.

\subsection*{Commutation with the mass functional}
We observe that
\[
\frac{d}{dt} \mathcal{M}\circ\Phi^t = \{F, \mathcal{M} \}\circ\Phi^t = 0 \qquad \forall t\in [0, 1].
\]
Then $\mathcal{M}\circ\Phi^t = \mathcal{M}\circ\Phi^0 = \mathcal{M}$ for all $t\in [0, 1]$. Since the Hamiltonian ${H}$ is mass preserving and $\Phi$ is symplectic, we have that
\[
\{ {H}\circ \Phi, \mathcal{M} \} = \{ {H}, \mathcal{M}\circ \Phi^{-1}\} \circ \Phi = \{ {H}, \mathcal{M}\} \circ \Phi = 0.
\]
}
\subsection*{Inverse of $\Phi$}
By taking $r_0$ possibly smaller, $\Phi$ is invertible (onto its image), the inverse $\Phi^{-1}$ being given by the backward flow of $F$ from time $t=0$ to time $t=-1$. In particular, all the properties and estimates satisfied by $\Phi$ hold true also for $\Phi^{-1}$.

\section{Energy estimates}\label{sec:energy}
In this section we show suitable energy estimates for the difference between solutions to the NLS equation, in the normal form coordinates, and trajectories of the truncated Hamiltonian $H^{(2)}+\mathcal{H}^{(4)}$, as the energy cascade orbit provided by Theorem \ref{thm:toy_model_scaled}.

\smallskip

In this section, we consider functional equations on the real Sobolev spaces $\mathcal{H}^s$ defined in \eqref{RealSobolev} and we systematically use the variables $Z=(z, \bar{z})$. We recall the discussion at the beginning of section \ref{sec:ham} concerning the convention used to denote the changes of variables.

With a slight abuse of notation, when we write that $Z=Z(t)$ solves an equation associated to a Hamiltonian $H$ we mean that $z=z(t)$ does it.
We also recall that the constant $M$ has been fixed in \eqref{choose_M}.

The main result of this section is the following.

\begin{theorem}
    [\emph{Energy estimate}]\label{pr:ener_est}
    Fix $s_0>4$, $s\geqslant s_0$ and $T>0$. There exists $\eps>0$ sufficiently small such that the following holds: if {$V^\lambda \in \mathcal{C}([0,T],\mathcal{H}^{s+2}(\T^2))$} is a solution to $H^{(2)}+\mathcal{H}^{(4)}$, and $Z\in \mathcal{C}([0,T],\mathcal{H}^s(\T^2))\cap \mathcal{C}^1([0, T], \mathcal{H}^{s-2}(\T^2))$ is a solution to $H \circ \Phi$ (see \eqref{ham:bnf}) such that $\Phi(Z)(t)$ exists for $t\in[0,T]$ and satisfies the smallness assumption
	\begin{equation}\label{piccolezza-norma-bassa}
		\sup_{t\in[0,T]}\|\Phi(Z)(t)\|_{s_0,M} < \eps,
	\end{equation}
	{then, for every $t\in[0,T]$, the difference $W := Z - V^\lambda$ satisfies the following bound:}
	{
	\begin{equation}\label{energia_finale}
		\begin{aligned}
			\|W(t)\|_{s,M}^2 \lesssim {}& \|W(0)\|_{s,M}^2 + \int_{0}^t M \|W(\tau)\|_{s,M}^2\|Z(\tau)\|_{s_0,M}^6 \\
			&+ \|W(\tau)\|_{s,M} \Big( \|Z(\tau)\|_{s_0-2,M}^6\big[\|(Z-\Phi(Z))(\tau)\|_{s+2,M}+\|V^{\lambda}(\tau)\|_{s+2,M}\big] \\
			&+ \|Z(\tau)\|_{s_0-1,M}^6\big[\|(Z-\Phi(Z))(\tau)\|_{s+1,M}+\|V^{\lambda}(\tau)\|_{s+1,M}\big] \\
			&+ \|X_{\mathcal{H}^{(4)}}(W+V^{\lambda}) - X_{\mathcal{H}^{(4)}}(V^{\lambda})\|_{s,M} +{ \|R(Z)(\tau)\|_{s,M}} + \|R_5(Z)(\tau)\|_{s,M} \\
			&+ \|Z(\tau)\|_{s_0,M}^6\|W(\tau)\|_{s,M} + \|\widetilde{G}(Z)(\tau)\|_{s,M} \Big) d\tau,
		\end{aligned}
	\end{equation}
	}
	where the remainder terms $R$, $R_5$ and $\widetilde{G}$ satisfy 
	\begin{align}
      \label{new_est_r}  \|R(Y)\|_{s,M} &\lesssim \|Y\|_{s_0,M}^{6}\|Y\|_{s,M},\\
	\label{new_est_5}	\|R_5(Y)\|_{s,M} &\lesssim \|Y\|_{\ell^1}^4 \|Y\|_{s,M}, \\
	\label{new_est_G}	\|\widetilde{G}(Y)\|_{s,M} &\lesssim M^2\|Y\|^8_{s_0,M}\|Y\|_{s,M},
	\end{align}
    for any $Y\in \mathcal{H}^s$.
\end{theorem}

%\begin{remark}
 %   Since both $Z(t)$ and $V^\lambda(t)$ are defined on $[0,T]$, their difference $W(t)$ is well defined on the same interval. The smallness condition \eqref{piccolezza-norma-bassa} will be used to symmetrize the equation and to prove the equivalence of the modified energy with the Sobolev norm.
%\end{remark}

In view of the quasilinear nature of \eqref{NLS}, the estimates above cannot be directly derived through standard integration by parts techniques, as they would result in a loss of derivatives. Here we rely instead on tools from paradifferential calculus.

{Our strategy is the following:
\begin{itemize}
\item We show that the equation for $Z$ inherits the paradifferential structure of the original equation \eqref{NLS}, thanks to the fact that the Birkhoff map $\Phi$ is a smooth perturbation of the identity. See Proposition \ref{NLSparapara}.
\item We write the equation for $W$ and we show that the only unbounded term, which does not allow to close immediately the energy estimate, is a paradifferential linear operator. Then, we simmetrize it by means of a microlocal change of coordinates.
\item We construct a modified paradifferential energy (see \eqref{new_modified_energy} below) which is equivalent to the $\|\cdot\|_{s,M}$ norm. Under the smallness assumption \eqref{piccolezza-norma-bassa}, this modified functional allows us to close the energy estimates by absorbing the derivative loss.
\end{itemize}
}

%\begin{remark}\label{rem:lower}
	%In order to perform the paradifferential estimates and close the bootstrap assumption for the quasilinear equation, it is necessary to maintain a priori smallness at a low-regularity level. This is the role of the smallness assumption \eqref{piccolezza-norma-bassa} on the $H^{s_0}$-norm, which ensures that the high-frequency interactions do not destroy the energy bounds during the \blue{(large) instability time interval $[0, T]$.} 
%\end{remark}
We start by recalling from \cite{FIJMPA} the main notions and results of paradifferential calculus that will be used throughout this section, adapting them when necessary to the weighted norm $\|\cdot\|_{s,M}$ defined in \eqref{norma_sM}. Some technical facts and their proofs are collected in Appendix \ref{app:para}.

We shall consider symbols 
$\mathbb{T}^{2}\times \mathbb{R}^{2}\ni (x,\xi)\mapsto a(x,\xi)$
in the spaces 
$\mathcal{N}_{s}^{m}$, $m,s\in \mathbb{R}$, $s\ge0$,
defined by the seminorms
\begin{equation}\label{normaSimbo}
	|a|_{\mathcal{N}_{s,M,n}^{m}} := \sup_{|\beta|\le n}\sup_{ \xi\in \mathbb{R}^{2}}
	{\langle \xi\rangle_M^{-m+|\beta|}}\|\partial_{\xi}^{\beta}a(\cdot,\xi)\|_{s,M}\,.
\end{equation}
The constant $m\in \mathbb{R}$ indicates the \emph{order} of the symbol, while
$s$ denotes its differentiability in the space variable.
\begin{comment}
{\color{blue}
For a pure Fourier multiplier $p=p(\xi)$ we use instead the seminorm
\[
	|p|_{\mathcal{M}_{M,n}^m}
	:=\sup_{|\beta|\leq n}\sup_{\xi\in\mathbb{R}^2}
	\langle\xi\rangle_M^{-m+|\beta|}
	|\partial_\xi^\beta p(\xi)|.
\]
Whenever one factor is a pure Fourier multiplier, the action and composition estimates below remain valid with its $\mathcal{N}$-seminorm replaced by $|\cdot|_{\mathcal{M}_{M,n}^m}$. The proof is the same Fourier-kernel argument, using that the multiplier has only zero spatial frequency; the constants are uniform for $M\geq1$.
}
\end{comment}
Let $0<\eps< 1/4$ and consider
a smooth function $\chi : \mathbb{R}\to[0,1]$ satisfying 
\begin{equation}\label{cutofffunct}
	\chi(y) = \begin{cases}
		1 & \text{if } |y| \le 5/4, \\
		0 & \text{if } |y| \ge 8/5.
	\end{cases}
\end{equation}
We define the cut-off function
\begin{equation}\label{cutofffunctepsilon}
	\chi_{\eps}(\xi) := \chi(|\xi|/\eps)\,.
\end{equation}
For a symbol $a(x,\xi)$ in $\mathcal{N}_{s}^{m}$
we define its Bony-Weyl quantization as 
\begin{equation}\label{quantiWeyl}
	\opbw(a(x,\xi))h := T_{a}h := \frac{1}{(2\pi)^{2}}\sum_{j\in \mathbb{Z}^{2}}e^{\mathrm{i} j\cdot x}
	\sum_{k\in\mathbb{Z}^{2}}
	\chi_{\eps}\Big(\frac{|j-k|}{\langle j+k\rangle}\Big)
	\widehat{a}\Big(j-k,\frac{j+k}{2}\Big)\widehat{h}(k),
\end{equation}
where $\widehat{a}(\eta,\xi)$ denotes the $\eta$-Fourier coefficient 
of $a(x,\xi)$ in the variable $x\in \mathbb{T}^{2}$.

The definition of $T_a$ is independent of the choice of the cut-off function $\chi_{\eps}$,
up to regularizing operators (satisfying estimates of the form \eqref{diffQuanti}), as a consequence of Lemma \ref{azione}.

\subsection{Paralinearization} 
In this section, we provide a paralinearization for the vector fields driving the NLS equation and the difference equation. 

We emphasize that the algebraic structure of the paralinearization is well-established and borrowed directly from the work of Feola-Gr\'ebert-Iandoli \cite{FGI}. However, to close our energy estimates, we need to control the action of these operators using the modified Sobolev norms $\|\cdot\|_{s,M}$. While the algebraic identities are carried from \cite{FGI} without modifications, the quantitative tame estimates and remainder bounds in the $\|\cdot\|_{s,M}$ topology require an adaptation of the standard paradifferential calculus. {The proofs of these adapted estimates are collected in Appendix \ref{app:para}, in particular Subsections \ref{app:para_calculus} and \ref{app:bony}.}

Following \cite{FGI}, we focus on the paralinearization of the vector field generated by the high-order Hamiltonian $H^{(\geq 8)}$. We define the scalar symbols
\begin{equation}\label{simboa2}
	\begin{aligned}
		a_2(V) &:= \big[h'(|v|^2)\big]^2 |v|^2\,, \qquad b_2(V) := \big[h'(|v|^2)\big]^2 v^2,\\
		\vec{a}_1(V) \cdot \xi &:= 2\big[h'(|v|^2)\big]^2 \sum_{j=1}^2 \Im(v \overline{\partial_{x_j} v})\xi_j\,, \quad \xi=(\xi_1, \xi_2)\,.
	\end{aligned}
\end{equation}
Next, we introduce the operator matrices
\begin{gather}\label{matriceA2}
	A_2(V) := \begin{pmatrix} a_2(V) & b_2(V) \\ \overline{b_2(V)} & a_2(V) \end{pmatrix},\\
	\label{matriceA}
	A(V; \xi) := -\mathrm{i} E A_2(V)|\xi|^2 - \mathrm{i} \big(\vec{a}_1(V) \cdot \xi\big) \id.
\end{gather}
By exploiting this exact algebraic structure, we can decompose the quasilinear vector field as
\begin{equation}\label{paralin:8}
X_{H^{(\ge 8)}}(V) = \opbw(A(V; \xi))V + R_{\mathrm{para}}(V),
\end{equation}
where the remainder satisfies the following bound {(whose  justification relies on the paralinearization and paradifferential tools developed in Subsections \ref{app:para_calculus} and \ref{app:bony} of Appendix \ref{app:para})}:
\begin{equation}\label{bound:Rpara}
\|R_{\mathrm{para}}(V)\|_{s,M} \lesssim \|V\|_{s_0,M}^6 \|V\|_{s,M}, \quad s \ge s_0 > 4.
\end{equation}

A key feature of our approach is that we choose to write the paralinearized system \eqref{QLNLS444} using the Bony-Weyl quantization $\opbw$, rather than the classical Bony quantization $\opb$. We decided to use the Weyl quantization, because it better captures the  skew-adjoint character of the Schr\"odinger equation.

{For the complete algebraic derivation of
\eqref{paralin:8}, we refer to Feola--Gr\'ebert--Iandoli
\cite[Section 4]{FGI} (see also \cite{FIJMPA}). The weighted
Bony--Weyl paralinearization and the product, composition, and
remainder estimates used here are proved in Subsections
\ref{app:para_calculus} and \ref{app:bony} of Appendix
\ref{app:para}.}

\begin{proposition}[\emph{Paralinearization in Birkhoff Coordinates}]\label{NLSparapara}
The equation associated to the Hamiltonian $H\circ\Phi$ can be written as the following paradifferential system:
\begin{equation}\label{QLNLS444}
	\dot{Z} = \mathrm{i} E \Delta Z + \opbw\big(A(\Phi(Z);\xi)\big)\Phi(Z) + X_{\mathcal{H}^{(4)}}(Z) + R(Z) + R_{5}(Z) + \widetilde{G}(Z)\,,
\end{equation}
where the remainder terms satisfy
\begin{align}\label{stimaRRR}
	\|R(Z)\|_{s,M} &\lesssim \|Z\|_{s_0,M}^{6}\|Z\|_{s,M} \qquad \forall s\ge s_0>4, \\
	\label{stima5}
	\|R_5(Z)\|_{s,M} &\lesssim \|Z\|_{\ell^1}^4 \|Z\|_{s,M}, \\
	\label{stima_G}
	\|\widetilde{G}(Z)\|_{s,M} &\lesssim M^2\|Z\|^8_{s_0,M}\|Z\|_{s,M}.
\end{align}
{Furthermore, the seminorms of the symbols satisfy the following
bounds: given $1<\widetilde{s}_0\leq s-1$, we have
\begin{equation}\label{realtaAAA2}
	\begin{aligned}
		|a_{2}(\Phi(Z))|_{\mathcal{N}^0_{s',M,0}} + |b_{2}(\Phi(Z))|_{\mathcal{N}^0_{s',M,0}} &\lesssim \|\Phi(Z)\|^{6}_{s',M}\,,
		\qquad \quad s'\in[\widetilde s_0,s], \\
		|\vec{a}_{1}(\Phi(Z))\cdot\xi|_{\mathcal{N}^1_{s',M,0}} &\lesssim \|\Phi(Z)\|^{6}_{s'+1,M}\,,
		\qquad s'\in[\widetilde s_0,s-1].
	\end{aligned}
\end{equation}}
\end{proposition}

\begin{proof}

%%%%%%%%%%%%

By \eqref{ham:bnf} the normalized NLS equation is given by
\begin{equation}\label{eqforZ}
\dot{Z} = \mathrm{i} E \Delta Z + X_{\mathcal{H}^{(4)}}(Z) + X_{R^{(\ge 6)}}(Z) + [d\Phi(Z)]^{-1} X_{H^{(\ge 8)}}(\Phi(Z)).
\end{equation}
By \eqref{paralin:8}
\begin{align*}
[d\Phi(Z)]^{-1} X_{H^{(\ge 8)}}(\Phi(Z)) &= X_{H^{(\ge 8)}}(\Phi(Z)) + \big([d\Phi(Z)]^{-1} - \id\big)X_{H^{(\ge 8)}}(\Phi(Z))\\
&=\opbw(A(\Phi(Z); \xi))\Phi(Z) + R_{\mathrm{para}}(\Phi(Z))+\widetilde{G}(Z),
\end{align*}
where
\[
\widetilde{G}(Z) := \big([d\Phi(Z)]^{-1} - \id\big)X_{H^{(\ge 8)}}(\Phi(Z)).
\]
Substituting this into \eqref{eqforZ} yields \eqref{QLNLS444}, with 
\begin{align*}
    R(Z) := R_{\mathrm{para}}(\Phi(Z)), \qquad
    R_5(Z) := X_{R^{(\ge 6)}}(Z).
\end{align*}
We prove the remainder estimates \eqref{stimaRRR}-\eqref{stima_G}. By the bounds \eqref{bound:Rpara} and \eqref{Phi:equiv} 
$$\|R(Z)\|_{s,M} \lesssim \|\Phi(Z)\|_{s_0,M}^6 \|\Phi(Z)\|_{s,M} \lesssim \|Z\|_{s_0,M}^6 \|Z\|_{s,M}.$$
By the smallness condition \eqref{piccolezza-norma-bassa} and the estimate \eqref{bound:rem},
$$\|R_5(Z)\|_{s,M} = \|X_{R^{(\ge 6)}}(Z)\|_{s,M} \lesssim \|Z\|_{\ell^1}^4 \|Z\|_{s,M}.$$
In order to bound $\|\widetilde{G}(Z)\|_{s, M}$ we first provide an estimate for the non homogeneous vector field $X_{H^{(\geq 8)}}$. By Lemma \ref{lem:Moser:sM} and Corollary \ref{cor:tame:prod} {we have, for $s_1>1$,
\begin{equation}\label{tame:X8}
\begin{aligned}
\| \Delta(h(|u|^2)) h'(|u|^2) u  \|_{s-2, M}\lesssim& \| \Delta(h(|u|^2)) \|_{s-2, M} \| h'(|u|^2) \|_{s_1, M} \| u\|_{s_1, M}\\
&+\| \Delta(h(|u|^2)) \|_{s_1, M} \| h'(|u|^2) \|_{s-2, M} \| u\|_{s_1, M}\\
&+\| \Delta(h(|u|^2)) \|_{s_1, M} \| h'(|u|^2) \|_{s_1, M} \| u\|_{s-2, M}\\
\lesssim& \| h(|u|^2) \|_{s, M}  \| |u|^2\|_{s_1, M} \| u\|_{s_1, M}\\
&+\| h(|u|^2) \|_{s_1+2, M} \| |u|^2\|_{s-2, M} \| u\|_{s_1, M}\\
&+\| h(|u|^2) \|_{s_1+2, M} \| |u|^2\|_{s_1, M} \| u\|_{s-2, M}\\
\lesssim& \| u\|_{L^{\infty}}^2 \| u\|_{s, M} \| u\|_{s_1, M}^4\\
&+\| u\|_{L^{\infty}}^2 \| u\|_{s_1+2, M} \| u\|^3_{s_1, M} \| u\|_{s-2, M}\\
\lesssim& \| u\|_{s_1+2, M}^6 \| u\|_{s, M}\lesssim \| u\|_{s_0, M}^6 \| u\|_{s, M}.
%&+\| u\|_{L^{\infty}}^2 \| u\|_{s_0+2, M} \| u\|^3_{s_0, M} \|u\|_{s-2, M}
\end{aligned}
\end{equation}
}
By Theorem \ref{thm:wbnf}-(i)
\begin{align*}
\|\big([d\Phi(Z)]^{-1} - \id\big) h\|_{s,M} &\lesssim M^2 \|\big([d\Phi(Z)]^{-1} - \id\big) h\|_{s-2,M}\\
&\stackrel{\eqref{close3}}{\lesssim} M^2 \|Z\|_{s_0,M}^2 \|h\|_{s-2,M}.
\end{align*}

Set $h = X_{H^{(\ge 8)}}(\Phi(Z))$. Applying the tame estimate \eqref{tame:X8} evaluated at $\Phi(Z)$,
$$\|X_{H^{(\ge 8)}}(\Phi(Z))\|_{s-2,M} \lesssim \|\Phi(Z)\|_{s_0,M}^6 \|\Phi(Z)\|_{s,M} \lesssim \|Z\|_{s_0,M}^6 \|Z\|_{s,M}.$$
Combining these inequalities yields
$$\|\widetilde{G}(Z)\|_{s,M} \lesssim M^2 \|Z\|_{s_0,M}^2 \big(\|Z\|_{s_0,M}^6 \|Z\|_{s,M}\big) = M^2 \|Z\|_{s_0,M}^8 \|Z\|_{s,M}.$$
{It remains to verify \eqref{realtaAAA2}. The algebraic formulas
for the symbols are those derived in \cite{FGI}. Since the double
zero of $h$ gives $h'(r)=r q(r)$ with $q$ smooth, the coefficients
$a_2$, $b_2$, and $\vec a_1$ in \eqref{simboa2} vanish to order at
least six in $V$ (counting one spatial derivative in $\vec a_1$).
The weighted Moser estimate \eqref{Moser2} and the tame product
estimate in Corollary \ref{cor:tame:prod} therefore give
\eqref{realtaAAA2} at $V=\Phi(Z)$. The constants are uniform in
$M$ on the fixed small ball determined by
\eqref{piccolezza-norma-bassa}.}
\end{proof}

In view of the above result, we can write the paralinearized equation for the difference $W = Z - V^\lambda$:
\begin{equation}\label{eq:differenza}
	\begin{aligned}
		\partial_t W &= \opbw\big(-\mathrm{i} E|\xi|^2 + A(\Phi(Z);\xi)\big)W \\
		&\quad + \big( X_{\mathcal{H}^{(4)}}(W+V^{\lambda}) - X_{\mathcal{H}^{(4)}}(V^{\lambda}) \big) \\
		&\quad + \opbw\big(A(\Phi(Z);\xi)\big)(\Phi(Z)-Z) \\
		&\quad + \opbw\big(A(\Phi(Z);\xi)\big)V^{\lambda} \\
		&\quad + R(Z) + R_5(Z) + \widetilde{G}(Z).
	\end{aligned}
\end{equation}

%\begin{remark}\label{re:canc} {\color{red} Da sviluppare}
%	In principle, evaluating the difference between the two equations yield a paralinearized operator acting on $\Phi(Z)$, which induces a loss of derivatives. To bypass this obstruction, we rewrite the term by adding and subtracting $Z - V^{\lambda}$. In this way, the resulting terms can be managed . The loss of derivatives on $W = Z - V^\lambda$ is bypassed through paradifferential symmetrization, the term involving $\Phi(Z)-Z$ is of lower order (smoothing), and the derivative loss on $V^{\lambda}$ can be absorbed since the $\|\cdot\|_{s_0,M}$ norms of the Toy model solutions remain small in time.
%\end{remark}

\begin{remark}\label{re:canc}
A direct subtraction of the equations for $Z$ and $V^\lambda$
produces the order-two term
\[
    \opbw\big(A(\Phi(Z);\xi)\big)\Phi(Z),
\]
which cannot be treated directly as a forcing term in $H^s$, since
this would require two additional derivatives of $\Phi(Z)$. We
instead use the identity
\[
    \Phi(Z)
    = W+V^\lambda+\big(\Phi(Z)-Z\big),
    \qquad W:=Z-V^\lambda,
\]
and decompose
\[
\begin{aligned}
    \opbw\big(A(\Phi(Z);\xi)\big)\Phi(Z)
    ={}& \opbw\big(A(\Phi(Z);\xi)\big)W \\
    &+\opbw\big(A(\Phi(Z);\xi)\big)\big(\Phi(Z)-Z\big) \\
    &+\opbw\big(A(\Phi(Z);\xi)\big)V^\lambda .
\end{aligned}
\]
These three contributions are treated differently. The first one is
included in the principal operator acting on $W$ and is controlled by
the paradifferential symmetrization introduced below. The second one
is perturbative because $\Phi-\id$ is smoothing. Finally, the last
term acts on the prescribed Toy-model solution rather than on the
unknown $W$: the required higher Sobolev norms of $V^\lambda$ are
estimated explicitly, while the coefficients of
$A(\Phi(Z);\xi)$ remain small in the low norm. Thus, no loss of
derivatives is imposed on the unknown $W$.
\end{remark}

\subsection{Proof of energy estimates}
To perform the energy estimates, we symmetrize the principal part of the paradifferential operator in \eqref{eq:differenza}. From the definition \eqref{matriceA}, the principal symbol (of order $2$ in $\xi$) driving the linear and quasilinear parts is given by $-\mathrm{i}E(\id + A_2(\Phi(Z)))|\xi|^2$. Therefore, we focus on the matrix $E(\id+A_2(\Phi(Z)))$. 

We define the scalar functions
\begin{equation}\label{autovalori}
	\mu=\mu(\Phi(Z)) := \sqrt{1+2|\Phi(Z)|^2[h'(|\Phi(Z)|^2)]^2}\,, 
	\qquad a_2^{(1)}=a_2^{(1)}(\Phi(Z)) := \mu(\Phi(Z))-1\,,
\end{equation}
and we note that $\pm\mu$ are the eigenvalues 
of the matrix $E(\id+{A}_2(\Phi(Z)))$. 
We denote by $S$ the matrix of the eigenvectors of $E(\id+A_{2}(\Phi(Z)))$, more explicitly,
\begin{equation}\label{matriceS}
	\begin{aligned}
		S &= \begin{pmatrix}
			s_1 & s_2 \\
			\overline{s}_2 & s_1
		\end{pmatrix}\,, 
		\qquad S^{-1} = \begin{pmatrix}
			s_1 & -s_2 \\
			-\overline{s}_2 & s_1
		\end{pmatrix}\,, \\
		s_1 &:= \frac{1+|\Phi(Z)|^2[h'(|\Phi(Z)|^2)]^2+\mu}
		{\sqrt{2\mu (1+[h'(|\Phi(Z)|^2)]^2|\Phi(Z)|^2+\mu)}}\,, \\
		s_2 &:= \frac{-\Phi(Z)^2[h'(|\Phi(Z)|^2)]^2}{\sqrt{2\mu (1+[h'(|\Phi(Z)|^2)]^2|\Phi(Z)|^2+\mu)}}\,.
	\end{aligned}
\end{equation}
Since $S$ is the matrix of eigenvectors 
of $E(\id+A_{2}(\Phi(Z)))$, we have that
\begin{equation}\label{diago_algebra}
	S^{-1}E(\id+A_{2}(\Phi(Z)))S = E\diag(\mu)\,,\qquad s_1^2-|s_2|^2 = 1
\end{equation}
where we have used the notation \eqref{matriciozze}.

\begin{lemma}\label{stime_diago_max}
	Let $\widetilde{s}_0> 1$. 
	The symbols $a_2^{(1)}$ defined in \eqref{autovalori}, 
	and $s_1-1, s_2$ defined in \eqref{matriceS}, 
	satisfy the following estimates
	\begin{equation*}
		|a_2^{(1)}|_{\mathcal{N}^0_{s',M,0}}+|s_1-1|_{\mathcal{N}^0_{s',M,0}}
		+|s_2|_{\mathcal{N}^0_{s',M,0}} \lesssim \|\Phi(Z)\|^{6}_{s',M}\,, 
		\qquad \forall \, s' \in [\widetilde{s}_0, s] \,.
	\end{equation*}
\end{lemma}

\begin{proof}
{The proof follows by using the explicit expressions \eqref{autovalori}, \eqref{matriceS} and the estimates \eqref{realtaAAA2} combined with the fact that $h$ has a zero at the origin of order at least two and $\Phi(Z)$ satisfies the smallness condition \eqref{piccolezza-norma-bassa}.}
\end{proof}
We now study how the system \eqref{QLNLS444} transforms under the action of the diagonalizing maps 
\begin{equation}\label{diago_para}
	\begin{aligned}
		\mathcal{F}(\Phi(Z)) &:= \opbw(S^{-1}(\Phi(Z);\xi))\,, \qquad 
		\mathcal{G}(\Phi(Z)) &:= \opbw(S(\Phi(Z);\xi))\,.
	\end{aligned}
\end{equation}

\begin{lemma}\label{propMAPPA}
For any $s \ge s_0 > 4$,
	we have the following.
	
	\noindent
	$(i)$ The following bounds hold 
	\begin{equation}\label{stime-descentTOTA}
		\begin{aligned}
			\|\mathcal{F}(\Phi(Z))W\|_{s,M} + \|\mathcal{G}(\Phi(Z))W\|_{s,M} &\lesssim 
			\|W\|_{s,M}\big(1 + \|\Phi(Z)\|^{6}_{s_0,M}\big)\,,\\
			\|(\mathcal{F}(\Phi(Z))-\id)W\|_{s,M} + \|(\mathcal{G}(\Phi(Z))-\id)W\|_{s,M} &\lesssim
			\|W\|_{s,M}\|\Phi(Z)\|^{6}_{s_0,M}\,,
		\end{aligned}
	\end{equation} 
	for any $W\in \mathcal{H}^{s}(\T^2;\mathbb{C})$.
	
	\noindent
	$(ii)$ We have $\mathcal{G}(\Phi(Z))\circ\mathcal{F}(\Phi(Z)) = \id+R(\Phi(Z))$
	where the remainder $R$ satisfies 
	\begin{equation}\label{achille10}
		\|R(\Phi(Z))W\|_{s+2,M} \lesssim \|W\|_{s,M}\|\Phi(Z)\|^{6}_{s_0,M}\,,
	\end{equation}
		for any $W\in \mathcal{H}^{s}(\T^2;\mathbb{C})$.
        
	\noindent
	$(iii)$ For any $t\in [0,T]$ and $\widetilde{s}_0>1$,
	we have 
	$\partial_{t}\mathcal{F}(\Phi(Z))[\cdot] = \opbw(\partial_{t}S^{-1}(\Phi(Z)))$. Furthermore, for all $s' \in [\widetilde{s}_0, s-2]$,
	\begin{equation}\label{achille11}
		|\partial_{t}S^{-1}(\Phi(Z))|_{\mathcal{N}^{0}_{s',M,0}} \lesssim
		\|\Phi(Z)\|^{6}_{s'+2,M}\,,
		\quad
		\|\partial_{t}\mathcal{F}(\Phi(Z))W\|_{s,M} \lesssim \|W\|_{s,M}\|\Phi(Z)\|^{6}_{s_0,M}\,.
	\end{equation}
\end{lemma}

\begin{proof}
	$(i)$ The bounds \eqref{stime-descentTOTA} follow by the action estimates in Lemma \ref{azione} and the bounds in Lemma \ref{stime_diago_max}.
	
	{
	\noindent
	$(ii)$ We apply Proposition \ref{prop:compo} 
	for the composition of the operators in \eqref{diago_para} and we use the bounds in Lemma \ref{stime_diago_max}. Before applying the remainder estimate, we separate the constant identity parts of $S$ and $S^{-1}$, whose quantizations compose exactly; thus only $S-\id$ and $S^{-1}-\id$ enter in the bound.
	
	\noindent
	$(iii)$ We note that the time derivative of the symbol $S^{-1}(\Phi(Z))$ involves the time derivative of $\Phi(Z)$. By the chain rule, $\partial_t s_1(\xi) = d_v s_1 [\partial_t \Phi(Z)] + d_{\overline{v}} s_1 [\partial_t \overline{\Phi(Z)}]$. Since $\Phi(Z)$ solves the equation \eqref{NLS}, we have 
    \begin{equation}\label{time:deriv}
    \|\partial_t \Phi(Z)\|_{s',M} \lesssim \|\Phi(Z)\|_{s'+2, M}.
    \end{equation}
    Direct inspection using the explicit structure of the symbols $s_1,s_2$ in \eqref{matriceS} gives the first estimate in \eqref{achille11} for every admissible $s'$. Taking then $s'=s_0-2>2$ and applying \eqref{actionSob} gives the second estimate.
	}
\end{proof}

We are now in position to state the following proposition.

\begin{lemma}\label{diagovera}
	Recalling \eqref{matriceA}, \eqref{matriceS}, and \eqref{autovalori}, we have
	\begin{multline*}
	    \mathcal{F}(\Phi(Z))\opbw(-\mathrm{i} E |\xi|^2+A(\Phi(Z);\xi))\mathcal{G}(\Phi(Z))[\cdot] = \\
			\quad -\mathrm{i} E\opbw\big(\mu(\Phi(Z))|\xi|^{2}\big)[\cdot]
			- \mathrm{i} \opbw\big(\vec{a}_1^{(1)}(\Phi(Z))\cdot\xi\big)[\cdot] \id + Q(\Phi(Z))[\cdot],
	\end{multline*}
	where $\vec{a}_1^{(1)}(\Phi(Z))$ has real components and $Q(\Phi(Z))$ is a linear operator satisfying
	\begin{equation}\label{stima-smooth}
		\|Q(\Phi(Z))Y\|_{s,M} \lesssim \|\Phi(Z)\|_{s_0,M}^6 \|Y\|_{s,M}
	\end{equation}
    for any $Y\in \mathcal{H}^s$.
\end{lemma}
{
\begin{proof}
	Set
	\[
		B:=E(\id+A_2(\Phi(Z))),\qquad
		D:=S^{-1}BS=\mu E,\qquad K_j:=S^{-1}\partial_{x_j}S.
	\]
	We apply Proposition \ref{prop:compo} with $\rho=2$ to the order-two block, and with $\rho=1$ to the order-one block. Besides the algebraic identity \eqref{diago_algebra}, the first-order contribution generated by the conjugation of the order-two block is
	\[
		\sum_{j=1}^2\Big((\partial_{x_j}S^{-1})BS
		-S^{-1}B(\partial_{x_j}S)\Big)\xi_j
		=-\sum_{j=1}^2(K_jD+DK_j)\xi_j.
	\]
	A direct computation from \eqref{matriceS} gives
	\[
		-\sum_{j=1}^2(K_jD+DK_j)\xi_j
		=-2\mathrm{i}\mu\sum_{j=1}^2
		\operatorname{Im}\big(\overline{s_2}\,\partial_{x_j}s_2\big)\xi_j\,\id.
	\]
	Consequently,
	\[
		\vec a_1^{(1)}
		=\vec a_1+2\mu\operatorname{Im}(\overline{s_2}\nabla_xs_2)
	\]
	has real components. All the terms left after this order-one contribution are absorbed into the remainder. Indeed,  they satisfy \eqref{stima-smooth} by Proposition \ref{prop:compo} with  index $s_0-2>2$,  and by Lemma \ref{stime_diago_max}.
\end{proof}
}

\begin{proposition}[{\bf Diagonalization at order $2$}]\label{diago2max}
	Let $s\geq s_0>4$ and $W\in \mathcal{C}([0,T],\mathcal{H}^s(\T^2))\cap \mathcal{C}^1([0, T], \mathcal{H}^{s-2}(\T^2))$ be a solution of the equation \eqref{eq:differenza} and set 
	\begin{equation}\label{cambio1}
		W_1 = \mathcal{F}(\Phi(Z))W\,,
	\end{equation}
	with $\mathcal{F}$ defined in \eqref{diago_para}. 
	Then $W_1\in \mathcal{C}([0,T],\mathcal{H}^s(\T^2))\cap \mathcal{C}^1([0, T], \mathcal{H}^{s-2}(\T^2))$ solves the equation
	\begin{equation}\label{QLNLS444KK}
		\begin{aligned}
			\dot{W}_1 ={}& \opbw(A^{(1)}(\Phi(Z);\xi))W_{1} \\
			&+ \mathcal{F}(\Phi(Z))\big[X_{\mathcal{H}^{(4)}}(W+V^{\lambda}) - X_{\mathcal{H}^{(4)}}(V^{\lambda})\big] \\
			&+ \mathcal{F}(\Phi(Z))\opbw(A(\Phi(Z);\xi))[\Phi(Z)-Z] \\
			&+ \mathcal{F}(\Phi(Z))\opbw(A(\Phi(Z);\xi))V^{\lambda} \\
			&+ \mathcal{F}(\Phi(Z))[R(Z)+R_5(Z)+\widetilde{G}(Z)] \\
			&+ Q(\Phi(Z))W_1 + \widetilde{Q}(\Phi(Z))W,
		\end{aligned}
	\end{equation}
	where the new principal symbol is 
    \[
    A^{(1)}(\Phi(Z);\xi) := -\mathrm{i} E\mu(\Phi(Z))|\xi|^{2} - \mathrm{i}\big(\vec{a}_1^{(1)}(\Phi(Z))\cdot\xi\big)\id,
    \]
    while $Q$ and $\widetilde{Q}$ are {bounded remainders} satisfying 
	\begin{equation}\label{stimaRRRKK}
		\|Q(\Phi(Z))Y\|_{s,M} + \|\widetilde{Q}(\Phi(Z))Y\|_{s,M} \lesssim \|Z\|_{s_0,M}^{6}\|Y\|_{s,M}\,
	\end{equation}
for any $Y \in \mathcal{H}^s(\mathbb{T}^2; \mathbb{C})$.

    The scalar symbol $a_2^{(1)}$ and the vector symbol $\vec{a}_1^{(1)}$ are real valued 
	and satisfy the following estimates: given $\widetilde{s}_0>1$,
	\begin{equation}\label{realtaAAA2KK}
		\begin{aligned}
		|a_{2}^{(1)}|_{\mathcal{N}^{0}_{s',M,0}} &\lesssim \|\Phi(Z)\|^{6}_{s',M}\,, 
			\quad &\forall \, s' \in [\widetilde{s}_0, s]\,,\\
			|\vec{a}_{1}^{(1)}\cdot\xi|_{\mathcal{N}^1_{s',M,0}} &\lesssim \|\Phi(Z)\|^{6}_{s'+1,M}\,,
			\quad &\forall \, s' \in [\widetilde{s}_0, s-1]\,.
		\end{aligned}
	\end{equation} 

\end{proposition}

\begin{proof}
By a direct computation
	\begin{align}
		\dot{W}_1 ={}& [\partial_t\mathcal{F}(\Phi(Z))]W + \mathcal{F}(\Phi(Z))\dot{W} \nonumber\\ 
		={}& [\partial_t\mathcal{F}(\Phi(Z))]W + \mathcal{F}(\Phi(Z))\opbw(-\mathrm{i} E|\xi|^2+A(\Phi(Z);\xi))W \label{mare1}\\
		&+ \mathcal{F}(\Phi(Z))\big[X_{\mathcal{H}^{(4)}}(W+V^{\lambda}) - X_{\mathcal{H}^{(4)}}(V^{\lambda})\big] \nonumber \\
		&+ \mathcal{F}(\Phi(Z))\opbw(A(\Phi(Z);\xi))(\Phi(Z)-Z) \nonumber \\
		&+ \mathcal{F}(\Phi(Z))\opbw(A(\Phi(Z);\xi))V^{\lambda} \nonumber \\
		&+ \mathcal{F}(\Phi(Z))[R(Z)+R_5(Z)+\widetilde{G}(Z)]. \nonumber
	\end{align}
	The first term in \eqref{mare1} is a contribution to the remainder $\widetilde{Q}$ because of item $(iii)$ in Lemma \ref{propMAPPA}. Concerning the second term in \eqref{mare1}, we use item $(ii)$ in Lemma \ref{propMAPPA} to rewrite it as
    \begin{multline}\label{principale}
        \mathcal{F}(\Phi(Z))\opbw(-\mathrm{i} E |\xi|^2+A(\Phi(Z);\xi))W = \\\mathcal{F}(\Phi(Z))\opbw(-\mathrm{i} E|\xi|^2+A(\Phi(Z);\xi))\mathcal{G}(\Phi(Z))\mathcal{F}(\Phi(Z))W,
    \end{multline}
	modulo a term which is absorbed into $\widetilde{Q}(\Phi(Z))W$. By applying Lemma \ref{diagovera}, the principal block in \eqref{principale} equals $\opbw(A^{(1)}(\Phi(Z);\xi))W_{1}$, modulo a final contribution which is absorbed into the remainder $Q(\Phi(Z))W_1$.
\end{proof}
We define the following modified energy functional
\begin{equation}\label{new_modified_energy}
	\|W\|_{Z,s,M}^2 := \left(\opbw(\mu(\Phi(Z))^s)\langle D\rangle_M^{s} W_1, \, \langle D\rangle_M^{s} W_1\right)_{L^2},
\end{equation}
where $W_1 = \mathcal{F}(\Phi(Z))W$ as defined in \eqref{cambio1}, $\langle D \rangle_M$ is the Fourier multiplier with symbol $\langle \xi \rangle_M$, and the scalar product is the standard one in $L^2(\mathbb{T}^2;\mathbb{C}^2)$.

In the following we prove that, under the smallness assumption \eqref{piccolezza-norma-bassa}, the norm $\| \cdot\|_{Z, s, M}$ is equivalent to the Sobolev norm $\|\cdot\|_{s,M}$.

\begin{proposition}\label{modified_energy}
	Assume \eqref{piccolezza-norma-bassa}. Then there are universal constants $c_1,c_3>0$ depending only on $s$ such that 
	\begin{equation}\label{equivalenze_piccole}
		c_1\|W\|_{s,M}^2 \leq \|W\|_{Z,s,M}^2 \leq  c_3\|W\|_{s,M}^2
	\end{equation}
	for any $W \in \mathcal{H}^{s}(\mathbb{T}^2)$.
\end{proposition}
{
\begin{proof}
Recall \eqref{autovalori}.
By Lemma \ref{stime_diago_max} and the  Moser
estimate \eqref{Moser2:base},
\[
|\mu(\Phi(Z))^s-1|_{\mathcal{N}^0_{s_0,M,0}}
\lesssim \|\Phi(Z)\|_{s_0,M}^6.
\]
Since $\mu^s$ is real, $\opbw(\mu^s)$ is self-adjoint. Moreover,
$\opbw(\mu^s)=\id+\opbw(\mu^s-1)$, and Lemma \ref{azione} gives
\[
\big|\big(\opbw(\mu^s-1)Y,Y\big)_{L^2}\big|
\lesssim \|\Phi(Z)\|_{s_0,M}^6\|Y\|_{L^2}^2.
\]
Choosing $\eps$ in \eqref{piccolezza-norma-bassa} so that the last
constant is at most $1/2$, we obtain
	\begin{multline*}
		\frac12 \| \langle D \rangle_M^s \mathcal{F}(\Phi(Z)) W \|_{L^2}^2
		\le \left( \opbw(\mu^s) \langle D \rangle_M^s \mathcal{F}(\Phi(Z)) W, \, \langle D \rangle_M^s \mathcal{F}(\Phi(Z)) W \right)_{L^2}
        \\\le \frac32 \| \langle D \rangle_M^s \mathcal{F}(\Phi(Z)) W \|_{L^2}^2\,.
	\end{multline*}
Finally, the second bound in \eqref{stime-descentTOTA} directly yields
\[
(1-C\|\Phi(Z)\|_{s_0,M}^6)\|W\|_{s,M}
\leq \|\mathcal{F}(\Phi(Z))W\|_{s,M}
\leq(1+C\|\Phi(Z)\|_{s_0,M}^6)\|W\|_{s,M}.
\]
Possibly choosing $\eps$ smaller enough, this proves \eqref{equivalenze_piccole}.
\end{proof}
}

We are now in position to prove the main energy estimate.

\begin{proof}[Proof of Proposition \ref{pr:ener_est}]
We start with the following computation, which is  justified by first considering smooth solutions and then appealing to a standard density argument:
\begin{equation}\label{comincio}
	\begin{aligned}
		\frac{d}{dt}\|W\|_{Z,s,M}^2 &= \big( \opbw(\partial_t(\mu(\Phi(Z))^s)) \langle D \rangle_M^s W_1, \, \langle D \rangle_M^s W_1 \big)_{L^2} \\
		&\quad + 2\Re \big( \opbw(\mu(\Phi(Z))^s) \langle D \rangle_M^s \dot{W}_1, \, \langle D \rangle_M^s W_1 \big)_{L^2}.
	\end{aligned}
\end{equation}
Along the proof we will systematically use the inequality 
\[
\| W_1 \|_{s, M}\lesssim \| W\|_{s, M}
\]
provided by Lemma \ref{propMAPPA}-(i) in combination with the smallness condition \eqref{piccolezza-norma-bassa}.

The first term on the right-hand side of \eqref{comincio} involves the time derivative of the symbol $\mu^s$. By the chain rule
\[
\begin{aligned}
\partial_t\bigl(\mu(\Phi(Z))^s\bigr)
={}&
s\,\mu(\Phi(Z))^{s-2}
h'(|\Phi(Z)|^2)
\left(
h'(|\Phi(Z)|^2)
+2|\Phi(Z)|^2h''(|\Phi(Z)|^2)
\right)
\\
&
\left(
\overline{\Phi(Z)}\,\partial_t\Phi(Z)
+
\Phi(Z)\,\partial_t\overline{\Phi(Z)}
\right).
\end{aligned}
\]
{
Therefore, by Corollary \ref{cor:tame:prod}, the weighted
Moser estimate \eqref{Moser2}, the bound
\eqref{time:deriv}, Lemma \ref{stime_diago_max} and recalling
\eqref{autovalori}, since $s_0-2>2$ we have
\[
\big|\partial_t \mu^s(\Phi(Z))\big|_{\mathcal{N}^0_{s_0-2,M,0}}
=\|\partial_t \mu^s(\Phi(Z))\|_{s_0-2,M}
\lesssim \| \Phi(Z) \|_{s_0, M}^6.
\]
Hence, by Lemma \ref{azione} and Cauchy--Schwarz inequality, the first term in \eqref{comincio} can be estimated by $\|W\|_{s,M}^2\|Z\|_{s_0,M}^6$, up to constant factors.
}

We now focus on the second line in \eqref{comincio}. By using \eqref{QLNLS444KK} we infer that the second line in \eqref{comincio} is equal to 

\begin{align}
	&\quad 2\Re\big( \opbw(\mu(\Phi(Z))^s)\langle D\rangle_M^{s}\opbw(A^{(1)}(\Phi(Z);\xi)){W}_1, \, \langle D\rangle_M^{s} W_1\big)_{L^2}\label{a1}\\
	&+2\Re\big( \opbw(\mu(\Phi(Z))^s)\langle D\rangle_M^{s}\mathcal{F}(\Phi(Z))[X_{\mathcal{H}^{(4)}}(W+V^{\lambda})-X_{\mathcal{H}^{(4)}}(V^{\lambda})], \, \langle D\rangle_M^{s} W_1\big)_{L^2}\label{a2}\\
	&+2\Re\big( \opbw(\mu(\Phi(Z))^s)\langle D\rangle_M^{s}\mathcal{F}(\Phi(Z))\opbw(A(\Phi(Z);\xi))(\Phi(Z)-Z), \, \langle D\rangle_M^{s} W_1\big)_{L^2}\label{a3}\\
	&+2\Re\big( \opbw(\mu(\Phi(Z))^s)\langle D\rangle_M^{s}\mathcal{F}(\Phi(Z))\opbw(A(\Phi(Z);\xi))V^{\lambda}, \, \langle D\rangle_M^{s} W_1\big)_{L^2}\label{a7}\\
    &+2\Re\big( \opbw(\mu(\Phi(Z))^s)\langle D\rangle_M^{s} \mathcal{F}(\Phi(Z))  R(Z), \, \langle D\rangle_M^{s} W_1\big)_{L^2}\label{a9}\\
	&+2\Re\big( \opbw(\mu(\Phi(Z))^s)\langle D\rangle_M^{s} \mathcal{F}(\Phi(Z)) R_{5}(Z), \, \langle D\rangle_M^{s} W_1\big)_{L^2}\label{a4}\\
	&+2\Re\big( \opbw(\mu(\Phi(Z))^s)\langle D\rangle_M^{s}Q(Z)W_1, \, \langle D\rangle_M^{s} W_1\big)_{L^2}\label{a5}\\
	&+2\Re\big( \opbw(\mu(\Phi(Z))^s)\langle D\rangle_M^{s}\widetilde{Q}(Z)W, \, \langle D\rangle_M^{s} W_1\big)_{L^2}\label{a6}\\
	&+2\Re\big( \opbw(\mu(\Phi(Z))^s)\langle D\rangle_M^{s} \mathcal{F}(\Phi(Z)) \widetilde{G}(Z), \, \langle D\rangle_M^{s} W_1\big)_{L^2}\label{a8}.
\end{align}
By the Cauchy--Schwarz inequality, Lemma \ref{azione} and the bounds \eqref{stime-descentTOTA}, \eqref{realtaAAA2}, we estimate the order-two and order-one parts of $A$ separately. The term \eqref{a7} is bounded from above by
\[
\|W\|_{s,M}\Big(
\|Z\|_{s_0-2,M}^6\|V^{\lambda}\|_{s+2,M}
+\|Z\|_{s_0-1,M}^6\|V^{\lambda}\|_{s+1,M}\Big).
\]
We emphasize that the operator $\opbw(A)$ is of order 2, this reflects on the fact that $V^\lambda$ is estimated in the norm $\| \cdot \|_{s+2, M}$

Reasoning exactly in the same way, we infer that 
\begin{equation*}
	\begin{aligned}
	\eqref{a2}+\eqref{a3} \lesssim \|W\|_{s,M}\Big(&\|Z\|^6_{s_0-2,M}\|\Phi(Z)-Z\|_{s+2,M}
	+\|Z\|^6_{s_0-1,M}\|\Phi(Z)-Z\|_{s+1,M}\\
	&+\|X_{\mathcal{H}^{(4)}}(W+V^{\lambda})-X_{\mathcal{H}^{(4)}}(V^{\lambda})\|_{s,M}\Big).
	\end{aligned}
\end{equation*}

Analogously, by also using \eqref{stimaRRRKK}, we obtain 
	\begin{align*}{\eqref{a9}}&+\eqref{a4}+\eqref{a5}+\eqref{a6}+\eqref{a8}\\
    &\lesssim\|W_1\|_{s,M}\Big({\|R(Z)\|_{s,M}}+\|R_5(Z)\|_{s,M}+\|Z\|^6_{s_0,M}(\|W_1\|_{s,M}+\|W\|_{s,M})+\|\widetilde{G}(Z)\|_{s,M}\Big)\\
    &\lesssim\|W\|_{s,M}\Big({\|R(Z)\|_{s,M}}+\|R_5(Z)\|_{s,M}+\|Z\|^6_{s_0,M}\|W\|_{s,M}+\|\widetilde{G}(Z)\|_{s,M}\Big).
\end{align*}
%{where we also used \eqref{cambio1} and \eqref{stime-descentTOTA}. }
We are left with the estimate of the term \eqref{a1} which involves the leading order operators. Here, we need to see a cancellation, at the microlocal level, in order to overcome a loss of derivatives. Recall the principal symbol 
\[
A^{(1)}(\Phi(Z);\xi) = -\mathrm{i} E\mu(\Phi(Z))|\xi|^2 - \mathrm{i} \big(\vec{a}_1^{(1)}(\Phi(Z))\cdot\xi\big)\id
\]
defined in Proposition \ref{diago2max}. Therefore, we have 
\begin{equation*}
	\begin{split}
		\eqref{a1} =& -2\Re\big(\opbw(\mu(\Phi(Z))^s)\langle D\rangle_M^{s}\opbw(\mathrm{i} E\mu(\Phi(Z))|\xi|^2\\
    &+\mathrm{i}\big(\vec{a}_1^{(1)}(\Phi(Z))\cdot\xi\big)\id){W}_1, \, \langle D\rangle_M^{s} W_1\big)_{L^2}.
	\end{split}
\end{equation*}
The term of order $1$ is bounded from above by $\|Z\|_{s_0,M}^6\|W\|_{s,M}^2$ thanks to Proposition \ref{prop:compo}, used with $\rho=1$. For the principal term of order $2$, we use Proposition \ref{prop:compo} with $\rho=2$, and we explicitly evaluate the principal symbol of the commutator arising from the Poisson bracket:
\begin{align*}
	\{ \mu^s (M^2 + |\xi|^2)^s, \, \mu |\xi|^2 \} 
	&= \nabla_\xi \big( \mu^s \langle\xi\rangle_M^{2s} \big) \cdot \nabla_x \big( \mu |\xi|^2 \big) 
	\,-\, \nabla_x \big( \mu^s \langle\xi\rangle_M^{2s} \big) \cdot \nabla_\xi \big( \mu |\xi|^2 \big) \\
	&= \Big( \mu^s \cdot 2s \langle\xi\rangle_M^{2s-2} \xi \Big) \cdot \Big( |\xi|^2 \nabla_x \mu \Big) 
	\,-\, \Big( s \mu^{s-1} \langle\xi\rangle_M^{2s} \nabla_x \mu \Big) \cdot \Big( 2\mu \xi \Big) \\
	&= 2s \mu^s \langle\xi\rangle_M^{2s-2} (\xi \cdot \nabla_x \mu) \Big[ |\xi|^2 - \langle\xi\rangle_M^2 \Big] \\
	&= -2s M^2 \mu^s \langle\xi\rangle_M^{2s-2} (\xi \cdot \nabla_x \mu)\,.
\end{align*}
To bound the action of this operator on the modified energy norm, we evaluate the supremum of its symbol relative to the Sobolev weight $\langle\xi\rangle_M^{2s}$: 
\[
\sup_{\xi \in \mathbb{R}^2} \frac{\big| M^2 \langle\xi\rangle_M^{2s-2} \xi \big|}{\langle\xi\rangle_M^{2s}} = \sup_{r \ge 0}  \frac{M^2\,r}{M^2 + r^2} = \frac{M}{2}\,.
\]
Let
\[
a(x,\xi)
=
\mu(x)^s\nabla_x\mu(x)\cdot
\frac{M^2\xi}{\langle\xi\rangle_M^2}.
\]
Since $s_0-1>2$, we have
\[
\begin{aligned}
|a|_{\mathcal N^0_{s_0-1,M,0}}
\lesssim
M\|\mu^s\nabla_x\mu\|_{s_0-1,M}\lesssim
M\|\mu-1\|_{s_0,M}
\lesssim
M\|Z\|_{s_0,M}^6.
\end{aligned}
\]
Therefore, by item $(ii)$ of Lemma~\ref{azione}, applied with
$m=0$, the commutator is bounded by 
\[
C M\|Z\|_{s_0,M}^6\|W\|_{s,M}^2.
\]
Altogether, we proved that 
\begin{equation*}
	\begin{aligned}
		\frac{d}{dt}\|W\|_{Z,s,M}^2 &\lesssim M \|W\|_{s,M}^2\|Z\|_{s_0,M}^6+\|W\|_{s,M}\|Z\|^6_{s_0-2,M}\big(\|\Phi(Z)-Z\|_{s+2,M}+\|V^{\lambda}\|_{s+2,M}\big)\\
		&\quad +\|W\|_{s,M}\|Z\|^6_{s_0-1,M}\big(\|\Phi(Z)-Z\|_{s+1,M}+\|V^{\lambda}\|_{s+1,M}\big)\\
		&\quad + \|W\|_{s,M}\Big(\|X_{\mathcal{H}^{(4)}}(W+V^{\lambda})-X_{\mathcal{H}^{(4)}}(V^{\lambda})\|_{s,M}+{\|R(Z)\|_{s,M}}+\|R_5(Z)\|_{s,M}\\
		&\quad + \|Z\|^6_{s_0,M}\|W\|_{s,M}+\|\widetilde{G}(Z)\|_{s,M}\Big).
	\end{aligned}
\end{equation*}
Integrating both sides on $[0,t)$ and using the $L^2$ equivalence \eqref{equivalenze_piccole}, we conclude the proof.
\end{proof}

\section{Approximation argument and conclusion of the proof}\label{sec:approximation}
In this final section we conclude the proof of Theorem \ref{th:main} on the long time strong instability for \eqref{NLS}. The key step is a suitable approximation argument between orbits of $H\circ \Phi$ and the Toy model, which we describe below.

Let $v^{\lambda}\in\mathcal{C}([0,T],H^s)$ be a solution of $H^{(2)}+\mathcal{H}^{(4)}$ provided by Theorem \ref{thm:toy_model_scaled}, where $T$ is given by \eqref{def:T}. We recall that $\lambda=\mathbf{R}^{s_{\ast}}$, as given in \eqref{choose_lambda}, and $M\sim 3^N\mathbf{R}$ by \eqref{choose_M} and \eqref{size:mode2}. We will repeatedly use that $\mathbf{R}$ is super-exponential in $N$, as follows by \eqref{Rmax}. In particular, positive powers of $\mathbf{R}$ absorb any exponential (or slower) growth in $N$. Throughout this section, we write $A\ll B$ whenever $A=o(B)$ as $N\to\infty$.

\begin{theorem}[Approximation argument]\label{thm:approx}
	Fix $s\geq 7$, $s\geq s_0> 4$, $\sigma >0$ sufficiently small, and $N>0$ sufficiently large. If  \begin{equation}\label{cond:lambdaq}
		s_* > \frac{6 s_0+1}{4}\,
	\end{equation}
	then the following holds. Let $z_0 \in H^s$ such that
	\begin{equation}\label{initial}
		\| z_0-v^{\lambda}(0) \|_{s, M} <  \lambda^{-1-2\sigma} M^{s-1}. 
	\end{equation}
	Then there exists a unique solution $z \in \mathcal{C}([0,T];H^s)$ to $H\circ \Phi$ with $z(0)=z_0$, which satisfies
	\begin{align}\label{final_a}
		\sup_{t\in [0, T]} \| z(t)-v^{\lambda}(t) \|_{s, M} &\leq \lambda^{-1-\sigma} M^{s-1},\\
		\label{l1a}
		\sup_{t\in [0, T]} \|z(t) \|_{\ell^1}&\ll 1.
	\end{align}
	Moreover, $u(t):=\Phi(z(t))$ is well-defined for $t\in[0,T]$, and solves \eqref{NLS}.
\end{theorem}

\subsection{Proof of Theorem \ref{thm:approx}}
Let us set $w = z-v^{\lambda}$, and define
\[
\begin{aligned}
	T_* := \sup\Big\{\,\tau\in[0,T] :\;&\exists\,z\in\mathcal{C}([0,\tau],H^s)
	\text{ solution to }H\circ\Phi\text{, with }z(0)=z_0,\\
	&\|w(t)\|_{s,M}<\lambda^{-1-\sigma}M^{s-1}\;
	\forall\,t\in[0,\tau]\,\Big\}.
\end{aligned}
\]
Since $\lambda>1$, \eqref{initial} gives $\|w(0)\|_{s,M}<\lambda^{-1-\sigma}M^{s-1}$. Observe that, by \eqref{initial} and \eqref{smoothMs},
\[
\| w(0) \|_{\ell^1} \lesssim  \lambda^{-1-\sigma}. 
\]
The above inequality and \eqref{eq:l1_bound_scaled} give
\begin{equation}\label{eq:zl1}
	\| z(0) \|_{\ell^1} = \| (w+v^{\lambda})(0)\|_{\ell^1} \lesssim \lambda^{-1-\sigma} + \theta \lambda^{-1} \lesssim \theta \lambda^{-1}\ll 1,
\end{equation}
where we set $\theta = N 2^{N-1}$. In particular, taking $N$ large enough,
\begin{equation}\label{radius_0}
	\| z(0) \|_{\ell^1} <{\textstyle{\frac{r_0}{2}}},
\end{equation}
where $r_0 > 0$ is the radius given in Theorem \ref{thm:wbnf}. As a consequence, $u_0 = \Phi(z(0))$ is a well-defined function in $B_{\ell^1}(r_0)\cap H^s$. In view of the local well-posedness for quasilinear Schr\"odinger equations \cite[Theorem 1.1]{iandoli}, there exists a time $T_1>0$ and a unique solution $u \in \mathcal{C}([0,T_1); H^s)$ to \eqref{NLS}. The continuous embedding $H^s\hookrightarrow\ell^1$ then yields $\|u(t)\|_{\ell^1}<r_0$ for any $t\in[0,T_2]$, for some $T_2\in(0,T_1]$. Applying again Theorem \ref{thm:wbnf}, we deduce that $z(t):=\Phi^{-1}(u(t))$ is well-defined for all $t\in[0,T_2]$. Since $\Phi$ is symplectic, $z$ is the unique solution to $[0,T_2]$ of $H\circ\Phi$ with initial datum $z_0$. Finally, in view of \eqref{initial}, we obtain that
\begin{equation*}
    \| w(t) \|_{s, M} < \lambda^{-1-\sigma} M^{s-1} \quad\forall\, t\in[0, T_3],
\end{equation*}
for some $0<T_3\le\min\{T_2,T\}$. Hence $T_*>0$. Iterating the above argument, and using the blow-up alternative for \eqref{NLS}, we deduce that there exists a unique solution $z \in \mathcal{C}([0,T_*];H^s)$ to $H\circ \Phi$, with $z(0)=z_0$, satisfying
\begin{align}	
	\label{final}
	\sup_{t\in [0, T_*]} \|w(t) \|_{s, M} &\leq \lambda^{-1-\sigma} M^{s-1},\\\label{small_l1}
	\sup_{t\in [0, T_*]} \|z(t) \|_{\ell^1}&\ll 1,
\end{align}
and such that $u(t):=\Phi(z(t))$ is well-defined for $t\in[0,T_*]$ and solves \eqref{NLS}. Observe moreover that, for $t\in[0,T_*]$, we have the bound
\begin{equation}\label{zMs-s}
	\| z(t) \|_{s, M} = \| (w+v^{\lambda})(t)\|_{s, M} \lesssim \lambda^{-1-\sigma} M^{s-1} + \sqrt{\theta} \lambda^{-1} M^{s} \lesssim \theta \lambda^{-1} M^{s},
\end{equation}
which follows by \eqref{final} and \eqref{eq:Ms_bound}.

It remains to show that $T_*=T$. Assume, by contradiction, $T_*<T$. We claim that, by taking $N$ possibly larger,
\begin{equation}\label{improved}
	\| w(t)\|_{s, M}^2 \le {\textstyle{\frac{1}{2}}} \lambda^{-2(1+\sigma)} M^{2(s-1)}, \qquad t\in [0, T_*].
\end{equation}
This contradicts the definition of $T_*$, thus yielding $T_*=T$.

In order to prove \eqref{improved}, we exploit the energy estimates provided by Theorem \ref{pr:ener_est} over the interval $[0,T_*]$. Preliminary, let us check that the smallness assumption \eqref{piccolezza-norma-bassa} is satisfied. Indeed, for all $t\in [0, T_*]$ we have
\begin{align*}
	\| \Phi(z(t)) \|_{s_0, M} &\le \| z(t)\|_{s_0, M} + \| \Phi(z(t))-z(t)\|_{s_0, M} \le \| z(t)\|_{s_0, M}(1 + \|z(t)\|_{\ell^1})\\
	&\lesssim \| z(t)\|_{s_0, M} \lesssim \theta \lambda^{-1}M^{s_0}\sim \theta 3^{Ns_0}\mathbf{R}^{s_0-s_*}\ll 1,
\end{align*}
where we used \eqref{close11}, \eqref{final} and the fact that $s_0< s_*$, which follows by \eqref{cond:lambdaq}.

The energy estimate \eqref{energia_finale} then yields, for all $t\in [0, T_*]$,
\begin{equation}\label{energy_bound_clean}
	\begin{aligned}
		\| w(t) \|_{s, M}^2 &\lesssim \| w(0)\|_{s, M}^2 + \int_0^t \Big( M\| w(\tau) \|_{s, M}^2 \| z(\tau) \|_{s_0, M}^6 \\
		&\quad + \| w(\tau) \|_{s, M} \Big[ \| z(\tau) \|_{s_0-2, M}^6 \big( \| \Phi(z(\tau))-z(\tau)\|_{s+2, M} + \| v^{\lambda}(\tau) \|_{s+2, M}\big)\\
		&\quad +\| z(\tau) \|_{s_0-1, M}^6 \big( \| \Phi(z(\tau))-z(\tau)\|_{s+1, M} + \| v^{\lambda}(\tau) \|_{s+1, M}\big)\\
		&\quad + \| X_{\mathcal{H}^{(4)}}(z(\tau)) - X_{\mathcal{H}^{(4)}}(v^{\lambda}(\tau))\|_{s, M} {+ \|R(z)\|_{s,M}} + \| R_5(z(\tau))\|_{s, M} \\
		&\quad + \| z(\tau) \|_{s_0, M}^6 \| w(\tau) \|_{s, M} + \|\widetilde{G}(z(\tau))\|_{s, M} \Big]\Big)\,d\tau\,.
	\end{aligned}
\end{equation}
We are going to estimate the various terms in the right hand side of \eqref{energy_bound_clean}, omitting for ease of notation the explicit dependence on $\tau\in[0,T_*]$. The most delicate is the one involving the vector field $X_{\mathcal{H}^{(4)}}$. To deal with this term, we need the following technical result, which exploits in a precise way the structure of the normalized Hamiltonian $\mathcal{H}^{(4)}$ given by \eqref{calH4}.
\begin{lemma}\label{lem:vector_field_diff}
	We have that 
	\begin{equation}\label{eg:diff_X}
		\| X_{\mathcal{H}^{(4)}}(w+v^{\lambda})- X_{\mathcal{H}^{(4)}}(v^{\lambda})\|_{s,M} \lesssim \lambda^{-2} \theta^2 \| w\|_{s,M} \,.
	\end{equation}
\end{lemma}
\begin{proof}
	We start with the bound
	\begin{align*}
		\| X_{\mathcal{H}^{(4)}}(w+v^{\lambda})- X_{\mathcal{H}^{(4)}}(v^{\lambda})\|_{s,M} &\le  \| X_{H^{(4,0)}}(w+v^{\lambda})- X_{H^{(4,0)}}(v^{\lambda})\|_{s,M}\\
		&\quad + \| X_{H^{(4, \geq 2)}}(w+v^{\lambda})- X_{H^{(4, \geq 2)}}(v^{\lambda})\|_{s,M}\\
		&= \underbrace{\| X_{H^{(4,0)}}(w+v^{\lambda})- X_{H^{(4,0)}}(v^{\lambda})\|_{s,M}}_{(I)} \\
		&\quad + \underbrace{\| X_{H^{(4, \geq 2)}}(w+v^{\lambda})\|_{s,M}}_{(II)}\,,
	\end{align*}
	where we used the fact that $X_{{H}^{(4, \geq 2)}}(v^{\lambda})=0$. 
	We note that, since $\Pi_{\Lambda} X_{H^{(4,0)}} = X_{H^{(4,0)}}$ and so $\Pi_{\le M} X_{H^{(4,0)}} = X_{H^{(4,0)}}$,
	\begin{equation}\label{bound:I}
		(I)^2\lesssim M^{2s} \| X_{H^{(4,0)}}(w+v^{\lambda})- X_{H^{(4,0)}}(v^{\lambda})\|_{L^2}^2,
	\end{equation}
	hence we just need to estimate the $L^2$ norm of this difference. By Minkowski and Young's inequality for convolutions, we have
	\begin{align*}
		\| X_{H^{(4,0)}}(w+&v^{\lambda})- X_{H^{(4,0)}}(v^{\lambda})\|_{L^2}^2 \\
		&= \sum_{n\in \Lambda} \Big| \sum_{\substack{n_1-n_2+n_3=n\\ n_1, n_2, n_3\in \Lambda}}  (w+v^{\lambda})_{n_1} \overline{(w+v^{\lambda})_{n_2}} (w+v^{\lambda})_{n_3} - v^{\lambda}_{n_1} \overline{v^{\lambda}_{n_2}} v^{\lambda}_{n_3} \Big|^2\\
		&= \sum_{n\in \Lambda} \Big| \sum_{\substack{n_1-n_2+n_3=n\\ n_1, n_2, n_3\in \Lambda}}  w_{n_1} \overline{(w+v^{\lambda})_{n_2}} (w+v^{\lambda})_{n_3} + v^{\lambda}_{n_1} \overline{w_{n_2}} (w+v^{\lambda})_{n_3} + v^{\lambda}_{n_1} \overline{v^{\lambda}_{n_2}} w_{n_3} \Big|^2\\
		&\le \| w*\overline{(w+v^{\lambda})}*(w+v^{\lambda}) + v^{\lambda}*\overline{w}*(w+v^{\lambda}) + v^{\lambda}*\overline{v^{\lambda}}*w \|_{\ell^2}^2\\
		&\le \big( \| w*\overline{(w+v^{\lambda})}*(w+v^{\lambda})\|_{\ell^2} + \|v^{\lambda}*\overline{w}*(w+v^{\lambda})\|_{\ell^2} + \|v^{\lambda}*\overline{v^{\lambda}}*w \|_{\ell^2} \big)^2\\
		&\le \| w\|_{\ell^2}^2 \big( \|w+v^{\lambda} \|_{\ell^1}^2 + \| v^{\lambda}\|_{\ell^1}\|w+v^{\lambda} \|_{\ell^1} + \|v^{\lambda}\|_{\ell^1}^2 \big)^2 \\
		&\lesssim \| w\|_{\ell^2}^2 (\| w\|_{\ell^1}+\| v^{\lambda}\|_{\ell^1})^4 \lesssim \| w\|_{\ell^2}^2 \| v^{\lambda}\|_{\ell^1}^4.
	\end{align*}
	Combining \eqref{bound:I}, \eqref{eq:l1_bound_scaled} and \eqref{final} we get
	\[
	(I) \lesssim M^{s} \| w\|_{L^2} \| v^{\lambda}\|_{\ell^1}^2 \lesssim \lambda^{-2} \theta^2 \| w\|_{s,M}\,.
	\]
	Now we provide a bound for $(II)$. We first observe that
	\begin{align}
		X_{H^{(4, \geq 2)}}(w+v^{\lambda}) &= \sum_{n\in \Lambda} e^{\mathrm{i} n x} \sum_{\substack{n_1-n_2+n_3=n\\ n_j, n_k\notin \Lambda, j\neq k, j, k\in \{ 1, 2, 3 \}}} (w+v^{\lambda})_{n_1} \overline{(w+v^{\lambda})_{n_2}} (w+v^{\lambda})_{n_3} \label{first}\\ \label{second}
		&\quad + \sum_{n\notin \Lambda} e^{\mathrm{i} n x} \sum_{\substack{n_1-n_2+n_3=n\\ n_j\notin \Lambda,  j\in \{ 1, 2, 3 \}}} (w+v^{\lambda})_{n_1} \overline{(w+v^{\lambda})_{n_2}} (w+v^{\lambda})_{n_3}\,.
	\end{align}
	Concerning the first term \eqref{first}, we note that if $n_j, n_k\notin \Lambda$ then
	\[
	(w+v^{\lambda})_{n_j}=w_{n_j}, \qquad (w+v^{\lambda})_{n_k}=w_{n_k}.  
	\]
	Then, since the sum runs over $n\in \Lambda$, we can reason exactly as for the bound on $(I)$ and obtain an estimate of order $M^s\,\| w\|_{L^2} \| w \|_{\ell^1} \| v^{\lambda}\|_{\ell^1}$. Let us now consider \eqref{second}. It is enough to bound the term
	\begin{equation}\label{term}
		\sum_{n\notin \Lambda} e^{\mathrm{i} n x} \sum_{\substack{n_1-n_2+n_3=n\\ n_1\notin \Lambda}} w_{n_1} \overline{(w+v^{\lambda})_{n_2}} (w+v^{\lambda})_{n_3},
	\end{equation}
	since the other terms can be estimated symmetrically. Let us consider the following splitting
	\begin{align*}
		\eqref{term} &= \sum_{n\notin \Lambda} e^{\mathrm{i} n x} \sum_{\substack{n_1-n_2+n_3=n\\ n_1\notin \Lambda\\ n_2\notin \Lambda, n_3\in \Lambda}} w_{n_1} \overline{w_{n_2}} (w+v^{\lambda})_{n_3} \\
		&\quad + \sum_{n\notin \Lambda} e^{\mathrm{i} n x} \sum_{\substack{n_1-n_2+n_3=n\\ n_1\notin \Lambda\\ n_2\in \Lambda, n_3\notin \Lambda}} w_{n_1} \overline{(w+v^{\lambda})_{n_2}} w_{n_3} \\
		&\quad + \sum_{n\notin \Lambda} e^{\mathrm{i} n x} \sum_{\substack{n_1-n_2+n_3=n\\ n_1\notin \Lambda\\ n_2\notin \Lambda, n_3\notin \Lambda}} w_{n_1} \overline{w_{n_2}} w_{n_3} \\
		&\quad + \sum_{n\notin \Lambda} e^{\mathrm{i} n x} \sum_{\substack{n_1-n_2+n_3=n\\ n_1\notin \Lambda\\ n_2\in \Lambda, n_3\in \Lambda}} w_{n_1} \overline{(w+v^{\lambda})_{n_2}} (w+v^{\lambda})_{n_3}\,.
	\end{align*}
	The worst term is the last one, because $w+v^{\lambda}$ has weaker smallness properties with respect to $w$ alone. Hence let us study its $\| \cdot \|_{s,M}$ norm. We recall that if $n_2, n_3\in \Lambda$ then $|n_2|, |n_3|\le M$. We distinguish two cases: $|n_1|\le M$ or $|n_1|>M$. We have
	\begin{align*}
		\Big\| \sum_{n\notin \Lambda} e^{\mathrm{i} n x} \sum_{\substack{n_1-n_2+n_3=n\\ n_1\notin \Lambda\\ n_2\in \Lambda, n_3\in \Lambda}} w_{n_1}& \overline{(w+v^{\lambda})_{n_2}} (w+v^{\lambda})_{n_3} \Big\|_{s,M}^2  \\
		&\lesssim \underbrace{\sum_{n\notin \Lambda} \Big|\sum_{\substack{n_1-n_2+n_3=n\\ n_1\notin \Lambda, |n_1|\le M\\ n_2\in \Lambda, n_3\in \Lambda}} w_{n_1} \overline{(w+v^{\lambda})_{n_2}} (w+v^{\lambda})_{n_3}\Big|^2 \langle n \rangle_M^{2s}}_{(A)}\\
		&\quad + \underbrace{\sum_{n\notin \Lambda} \Big|\sum_{\substack{n_1-n_2+n_3=n\\ n_1\notin \Lambda, |n_1|>M\\ n_2\in \Lambda, n_3\in \Lambda}} w_{n_1} \overline{(w+v^{\lambda})_{n_2}} (w+v^{\lambda})_{n_3}\Big|^2 \langle n \rangle_M^{2s}}_{(B)}
	\end{align*}
	Let us focus on the term $(A)$. By momentum conservation $|n|\le 3 M$, hence (reasoning as for $(I)$)
	\[
	(A)\lesssim (3M)^{2s} \sum_{n\notin \Lambda} \Big|\sum_{\substack{n_1-n_2+n_3=n\\ n_1\notin \Lambda, |n_1|\le M\\ n_2\in \Lambda, n_3\in \Lambda}} w_{n_1} \overline{(w+v^{\lambda})_{n_2}} (w+v^{\lambda})_{n_3}\Big|^2 \lesssim M^{2s} \| w \|_{L^2}^2 \| v^{\lambda} \|_{\ell^1}^4.
	\]
	Hence, this gives an estimate equivalent to $(I)$.
	Concerning $(B)$, by momentum conservation we have that $\langle n \rangle_M \lesssim |n_1|$, thus, applying Young's inequality again,
	\begin{align*}
		(B)&\lesssim \sum_{n\notin \Lambda} \Big|\sum_{\substack{n_1-n_2+n_3=n\\ n_1\notin \Lambda, |n_1|>M\\ n_2\in \Lambda, n_3\in \Lambda}} |n_1|^{s}\,w_{n_1} \overline{(w+v^{\lambda})_{n_2}} (w+v^{\lambda})_{n_3}\Big|^2 \\
		&\lesssim \| w\|_{s,M}^2 \| v^{\lambda}\|_{\ell^1}^4 \lesssim \theta^4 \lambda^{-4} \| w\|_{s,M}^2.
	\end{align*}
	This concludes the proof.
\end{proof}
Next, using the bounds \eqref{new_est_r}-\eqref{new_est_G}, \eqref{eq:Ms_bound} and \eqref{zMs-s}, we get
\begin{align*}
	\|R(z)\|_{s,M}&
	\lesssim
	\|z\|_{s_0,M}^{6}\|z\|_{s,M}\lesssim
	\big(\theta\lambda^{-1}M^{s_0}\big)^6
	\big(\theta\lambda^{-1}M^s\big)\lesssim
	\theta^7\lambda^{-7}M^{s+6s_0},
	\\ 
	\|R_5(z)\|_{s,M} &\lesssim \|z\|_{\ell^1}^4 \|z\|_{s,M} \lesssim \theta^5 \lambda^{-5} M^s\,,
	\\
	\|\widetilde{G}(z)\|_{s,M} &\lesssim M^2 \|z\|^8_{s_0,M}\|z\|_{s,M} \lesssim \theta^9 \lambda^{-9} M^{s+8s_0+2}.
\end{align*}
Owing to \eqref{close1}-\eqref{close11} and \eqref{zMs-s}, we also obtain
\begin{align*}
	\| \Phi(z)-z\|_{s+2, M}& \lesssim M^2\| \Phi(z)-z\|_{s, M}\lesssim M^2\|z\|^2_{\ell^1}\|z\|_{s,M}\\
	& \lesssim M^{s+2}\theta^3\lambda^{-3}.
\end{align*}
Furthermore,
\begin{align*}
	\| w \|_{s, M} \| z \|_{s_0-2, M}^6 \| \Phi(z)-z \|_{s+2, M} &\lesssim \Big( \lambda^{-1-\sigma} M^{s-1} \Big) \Big( \theta \lambda^{-1} M^{s_0-2} \Big)^6 \Big( M^{s+2}\theta^3\lambda^{-3} \Big) \\
	&\lesssim \theta^9 \lambda^{-10-\sigma} M^{2s-11+6s_0}\,,
\end{align*}
\begin{align*}
	\| w \|_{s, M} \| z \|_{s_0-2, M}^6 \| v^\lambda \|_{s+2, M} &\lesssim \Big( \lambda^{-1-\sigma} M^{s-1} \Big) \Big( \theta \lambda^{-1} M^{s_0-2} \Big)^6 \Big( \theta \lambda^{-1} M^{s+2} \Big) \\
	&\lesssim \Big( \lambda^{-1-\sigma} M^{s-1} \Big) \Big( \theta^6 \lambda^{-6} M^{6s_0-12} \Big) \Big( \theta \lambda^{-1} M^{s+2} \Big) \\
	&\lesssim \theta^7 \lambda^{-8-\sigma} M^{2s-11+6s_0}\,.
\end{align*}
Concerning the third line of \eqref{energy_bound_clean}, we have
\[
\|\Phi(z)-z\|_{s+1,M}\lesssim M^{s+1}\theta^3\lambda^{-3},
\quad
\|z\|_{s_0-1,M}\lesssim\theta\lambda^{-1}M^{s_0-1},
\quad
\|v^\lambda\|_{s+1,M}\lesssim\theta\lambda^{-1}M^{s+1}.
\]
Consequently,
\begin{align*}
	\| w \|_{s, M} \| z \|_{s_0-1, M}^6
	\big(\|\Phi(z)-z\|_{s+1,M}+\|v^\lambda\|_{s+1,M}\big)
	\lesssim{}&
	\theta^9\lambda^{-10-\sigma}M^{2s+6s_0-6}\\
	&+\theta^7\lambda^{-8-\sigma}M^{2s+6s_0-6}.
\end{align*}
Finally, we have
\begin{align*}
	\| w \|_{s, M} \|R (z)\|_{s, M} &\lesssim (\lambda^{-1-\sigma} M^{s-1}) ( \theta^7\lambda^{-7}M^{s+6s_0} ) \lesssim \theta^7 \lambda^{-8-\sigma} M^{2s-1+6s_0}\,,\\
	\| w \|_{s, M} \|R_5(z)\|_{s, M} &\lesssim (\lambda^{-1-\sigma} M^{s-1}) ( \theta^5 \lambda^{-5} M^s ) \lesssim \theta^5 \lambda^{-6-\sigma} M^{2s-1}\,,\\
	\| w \|_{s, M} \|\widetilde{G}(z)\|_{s, M} & \lesssim (\lambda^{-1-\sigma} M^{s-1}) ( \theta^9 \lambda^{-9} M^{s+8s_0+2} ) \lesssim \theta^9 \lambda^{-10-\sigma} M^{2s+1+8s_0}\,.
\end{align*}
Combining \eqref{eg:diff_X} with the above estimates we deduce that, for every $t\in [0, T_*]$,
\begin{align*}
	\| w(t) \|_{s, M}^2 &\lesssim \lambda^{-2-4\sigma} M^{2(s-1 )} + \int_0^t \| w\|_{s, M}^2 \big(\theta^6 \lambda^{-6} M^{6 s_0 + 1} + \theta^2 \lambda^{-2}\big)\,d\tau\\
	&\quad + t \Big( \theta^9 \lambda^{-10-\sigma} M^{2s+1 +8s_0} + \theta^9 \lambda^{-10-\sigma} M^{2s-11 +6s_0} \\
	&\quad + \theta^9 \lambda^{-10-\sigma} M^{2s+6s_0-6 }
	+ \theta^7 \lambda^{-8-\sigma} M^{2s+6s_0-6 }\\
	&\quad + \theta^7 \lambda^{-8-\sigma} M^{2s-11 +6s_0} + \theta^5 \lambda^{-6-\sigma} M^{2s-1 } + {\theta^7 \lambda^{-8-\sigma} M^{2s-1 +6s_0}} \Big) \,.
\end{align*}
Observe moreover that \eqref{cond:lambdaq} implies
$$
\theta^6 \lambda^{-6} M^{6 s_0 + 1} + \theta^2 \lambda^{-2} \le 2 \theta^2 \lambda^{-2},
$$ 
provided $N$ is large enough. Then by Gr\"onwall's lemma, for all $t\in [0, T_*]$,
\begin{align}
    \| w(t) \|_{s, M}^2 \lesssim \exp(2 \theta^2 \lambda^{-2} t) \Big[ \lambda^{-2-4\sigma} M^{2(s-1 )} &+ t\Big( {\theta^7 \lambda^{-8-\sigma} M^{2s-1 +6s_0}} \notag\\
		&\quad+ \theta^7\lambda^{-8-\sigma}M^{2s+6s_0-6 }\label{Gronwall}\\
		&\quad+ \theta^9 \lambda^{-10-\sigma} M^{2s+1 +8s_0} + \theta^5 \lambda^{-6-\sigma} M^{2s-1 } \Big)\Big]\,.\notag
\end{align}
Now consider $c>0$ such that $c \lambda^{2} \log(\lambda) > T$. Recalling \eqref{Rmax}, \eqref{T0} and \eqref{def:T}, we can choose
\begin{equation}\label{lowb:c}
	c = \frac{2\mathbb{K}\gamma_0N^2}{s_*\, \alpha^N},
\end{equation}
with $\alpha\geqslant \alpha_0$ as in property (P6) of the $\Lambda$ set given by Theorem \ref{thm:set}. We first show a bound on the exponential factor in \eqref{Gronwall}. Since we are assuming $T_*<T$, for all $t\in [0, T_*]$ we have
\[
\exp(2 \theta^2 \lambda^{-2} t)\le \exp(2 \theta^2 \lambda^{-2} T) \le \exp\Big(2 \theta^2 \lambda^{-2} \big(c \lambda^2 \log \lambda\big)\Big) = \exp\big(2 c \theta^2 \log \lambda\big) = \lambda^{\gamma},
\]
where we set
\[
\gamma:=2 c \theta^2.
\]
Let us choose $\alpha>4$. Recalling \eqref{lowb:c} and that $\theta=N2^{N-1}$, we have
$$\gamma\sim \frac{N^44^N}{s_*\alpha^N}\ll 1.$$
In view of \eqref{Gronwall} and the bound $T_*<T<\lambda^2\log(\lambda^c)$, estimate \eqref{improved} is satisfied by imposing the following conditions
\begin{align}
	& \lambda^{-2-4\sigma+\gamma} M^{2(s-1 )} \ll \lambda^{-2(1+\sigma)} M^{2(s-1 )} \nonumber \\ 
	&{\Longleftrightarrow \quad \lambda^{\gamma-2\sigma} \ll 1}, \label{11}\\ 
	\nonumber \\
	&{\lambda^{-6-\sigma+\gamma}\log(\lambda^c)\theta^7 M^{2s+6s_0-1}\ll\lambda^{-2(1+\sigma)}M^{2(s-1)}}\nonumber\\
	&{\Longleftrightarrow\quad \lambda^{4-\sigma-\gamma}\gg\theta^7\log(\lambda^c)M^{6s_0+1}},\label{13}\\
	\nonumber \\
	& {\lambda^{-6-\sigma+\gamma}\log(\lambda^c)\theta^7M^{2s+6s_0-6 }
		\ll\lambda^{-2(1+\sigma)}M^{2(s-1 )}}\nonumber\\
	&{\Longleftrightarrow\quad
		\lambda^{4-\sigma-\gamma}
		\gg\theta^7\log(\lambda^c)M^{6s_0-4}},
	\label{13a}\\
	\nonumber\\
	& \lambda^{-8-\sigma+\gamma} \log(\lambda^c) \theta^{9} M^{2s+1 +8 s_0} \ll \lambda^{-2(1+\sigma)} M^{2(s-1 )} \nonumber \\ 
	&\Longleftrightarrow \quad \lambda^{6-\sigma-\gamma} \gg \theta^{9} \log(\lambda^c) M^{3+8 s_0}, \label{12}\\ 
	\nonumber \\
	& \lambda^{-4-\sigma+\gamma} \log(\lambda^c) \theta^5 M^{2s-1 } \ll \lambda^{-2(1+\sigma)} M^{2(s-1 )} \nonumber \\ 
	&\Longleftrightarrow \quad \lambda^{2-\sigma-\gamma} \gg \log(\lambda^c) \theta^5 M^{1} \label{16}.
\end{align}
Notice that the exponent $s$ cancels out of all the conditions. Condition \eqref{11} is satisfied as soon as $\gamma<\frac{\sigma}{2}$. Since $c\ll 1$, estimates \eqref{13}-\eqref{16} reduce to the following lower bounds on $s_*$:
\begin{itemize}
\item Condition \eqref{13} is satisfied if $s_* > \frac{1+6s_0}{4-\sigma-\gamma}$.
	\item Condition \eqref{13a} is satisfied if $s_* > \frac{6s_0-4}{4-\sigma-\gamma}$.
	\item Condition \eqref{12} is satisfied if $s_* > \frac{3+8s_0}{6-\sigma-\gamma}$.
	\item Condition \eqref{16} is satisfied if $s_* > \frac{1}{2-\sigma-\gamma}$.
\end{itemize}
Since $s_0>4$, for $\sigma,\gamma>0$ sufficiently small condition \eqref{13} is the most restrictive one. Thus we require 
\[
s_* > \frac{1+6s_0}{4-\sigma-\gamma}\,,
\]
and $N$ large enough. Recalling \eqref{cond:lambdaq}, the above condition is satisfied. Therefore, we proved estimate \eqref{improved} and this concludes the proof.

\subsection{Proof of Theorem \ref{th:main}}
Let us consider a solution $v^{\lambda}$ to $H^{(2)}+\mathcal{H}^{(4)}$ with
\begin{equation}\label{sm_bi}
	\|v^\lambda(0)\|_s<\frac{\mu}{2},\qquad \|v^\lambda(T)\|_s> 2\mathcal{K},
\end{equation}
as provided by Corollary \ref{co:unstable_toy}. Let also $s_*=s+\delta$, with $0<\delta\ll 1$ as in \eqref{forma_star}. Observe that, since $s_*>s\geq 7$, condition \eqref{cond:lambdaq} is satisfied for $s_0=4+\nu$, with $\nu>0$ small enough.

We take $u_0\in H^s$ such that
\begin{equation}\label{eq:open}
	\|u_0-v^\lambda(0)\|_s
	<\min\left\{\frac{\mu}{2},\lambda^{-3/2}\right\}.
\end{equation}
In particular,
\[
\|u_0\|_s
\leq \|v^\lambda(0)\|_s
+\|u_0-v^\lambda(0)\|_s
<\frac{\mu}{2}+\frac{\mu}{2}<\mu.
\]
Next, using \eqref{eq:l1_bound_scaled} we get
\begin{equation}\label{u0l1}
	\|u_0\|_{\ell^1}\leq \|u_0-v^\lambda(0)\|_s+\|v^\lambda(0)\|_{\ell^1}\leq \lambda^{-3/2}+N2^{N-1}\lambda^{-1},
\end{equation}
so that $\|u_0\|_{\ell^1}\ll 1$. Hence we can apply Theorem \ref{thm:wbnf}, which guarantees that $z_0:=\Phi^{-1}(u_0)\in H^s$ is well-defined, and $\|z_0\|_{\ell^1}\ll 1$. Moreover, using estimates  \eqref{close11}, \eqref{Phi:equiv} and \eqref{u0l1} we obtain that for $N$ large enough
\begin{align*}
	\|z_0-v^\lambda(0)\|_{s,M}
	&\leq \|z_0-u_0\|_{s,M}
	+\|u_0-v^\lambda(0)\|_{s,M}\\
	&\leq M^s\|z_0\|_{\ell^1}^2\|z_0\|_s
	+\lambda^{-3/2}M^s\\
	&\lesssim M^s\|u_0\|_{\ell^1}^2\|u_0\|_s
	+\lambda^{-3/2}M^s\\
	&\lesssim
	\big(\lambda^{-3}+N^2 4^{N-1}\lambda^{-2}\mu+\lambda^{-3/2}\big)M^s
	\lesssim\lambda^{-3/2}M^s\lesssim \lambda^{-1-2\sigma}M^{s-1},
\end{align*}
for $\sigma>0$ sufficiently small. In view of the above inequality, we can apply Theorem \ref{thm:approx}. Then there exist $z$ and $u=\Phi(z)$ solutions over $[0,T]$ to $H\circ \Phi$ and \eqref{NLS}, respectively, with 
\begin{gather}\label{small_boot}
	\sup_{t\in[0,T]}
	\|z(t)-v^\lambda(t)\|_{s,M}
	\leq\lambda^{-1-\sigma}M^{s-1}\lesssim 3^{(s-1)N}\mathbf{R}^{-s_*+s-1}\ll 1,\\\label{small_r1}
	\sup_{t\in[0,T]}\|z(t)\|_{\ell^1}\ll 1.
\end{gather}
Owing to \eqref{sm_bi} and \eqref{small_boot} we obtain 
\begin{equation}\label{zTbig}
	\|z(T)\|_s \geq \|v^\lambda(T)\|_s
	-\|z(T)-v^\lambda(T)\|_s > 2\mathcal{K}.
\end{equation}
Moreover, applying \eqref{close11} with $M=1$\footnote{Even though we fixed $M$ as in \eqref{choose_M} in Theorem \ref{thm:wbnf}, large part of the statement, and in particular estimate \eqref{close11} actually holds for arbitrary $M>1$.}, we get
\begin{equation}\label{closesmall}
	\|u(T)-z(T)\|_s
	\lesssim\|z(T)\|^2_{\ell^1}\|z(T)\|_s.
\end{equation}
In view of \eqref{closesmall}, \eqref{small_r1} and \eqref{zTbig}, taking $N$ sufficiently large we eventually deduce
\[
\|u(T)\|_s\geq {\textstyle{\frac{1}{2}}}\|z(T)\|_s>\mathcal{K}.
\]
It remains to prove the estimate \eqref{time_bound} on $T$. To this aim, we fix 
\begin{equation}\label{rightN}
N=\left\lceil{\textstyle{\frac{4}{(s-1)\log 2}}}\log(\mathcal{K}\mu^{-1})\right\rceil
\end{equation}
The choice above guarantees, for $\mathcal{K}\mu^{-1}$ large enough, the validity of the lower bound \eqref{choice:N} used in the proof of Corollary \ref{co:unstable_toy}, as well as of  all the others lower bounds on $N$ imposed throughout this section, since they are independent of $\mathcal{K}$ and $\mu$.

Using \eqref{rightN} and \eqref{Rmax} we get
\[
\log\mathbf{R}
\leq2(1+\eta)\alpha^N
\lesssim (\mathcal{K}\mu^{-1})^{\frac{\gamma_{\mathrm R}}{s-1}},
\]
where we set $\gamma_R=\frac{4\log\alpha}{\log 2}$. Owing to \eqref{T0}, \eqref{def:T}, and \eqref{choose_lambda} we
deduce, for $\mathcal{K}\mu^{-1}$ large enough,
\begin{align*}
	\log T
	&\leq \log\big(\mathbb{K}\gamma_0N^2\big)
	+2\log\lambda\\
	&=\log\big(\mathbb{K}\gamma_0N^2\big)
	+2s_*\log\mathbf{R}\\
	&\leq C_s(\mathcal{K}\mu^{-1})^{\frac{\gamma_{\mathrm R}}{s-1}}\leq C(\mathcal{K}\mu^{-1})^{\frac{2\gamma_{\mathrm R}}{s-1}},
\end{align*}
where $C_s$ denotes a constant depending only on $s$ (recall that $s_*-s\ll 1$), while $C>0$ is an absolute constant. Thus, setting $\gamma:=2\gamma_R=\frac{8\log\alpha}{\log 2}$, we obtain
\[
T\leq
\exp\left(
C\left(\frac{\mathcal{K}}{\mu}\right)^{
	\frac{\gamma}{s-1}}
\right).
\]
This proves \eqref{time_bound} and concludes the proof.

%%%%%%%%%%%%%%%%%%%%%%%%%%%%%%

%%%%%%%%%%%%%%%%%%%%%%%%%%%

\appendix

\section{Paradifferential calculus}\label{app:para}

\subsection{Bony--Weyl paradifferential calculus}\label{app:para_calculus}

In this subsection we collect some useful results on paradifferential calculus, with reference to the definitions given in Section \ref{sec:energy}.

\begin{lemma}\label{azione}
	The following holds.
	
	{
	\noindent
	$(i)$ Let $m_1,m_2\in \mathbb{R}$ and $s>1$. For
	$a\in\mathcal{N}^{m_1}_{s,M,n}$ and
	$b\in \mathcal{N}^{m_2}_{s,M,n}$ one has
	\begin{equation}\label{prodSimboli}
		|ab|_{\mathcal{N}^{m_1+m_2}_{s,M,n}}
		\lesssim |a|_{\mathcal{N}_{s,M,n}^{m_1}}|b|_{\mathcal{N}_{s,M,n}^{m_2}}.
	\end{equation}
	If moreover $a\in\mathcal{N}^{m_1}_{s,M,n+1}$ and
	$b\in\mathcal{N}^{m_2}_{s,M,n+1}$, then
	\[
		|\{a,b\}|_{\mathcal{N}_{s-1,M,n}^{m_1+m_2-1}}
		\lesssim
		|a|_{\mathcal{N}_{s,M,n+1}^{m_1}}
		|b|_{\mathcal{N}_{s,M,n+1}^{m_2}}.
	\]
	}
	
	\noindent
	$(ii)$ Let $s_0>2$, $m\in \mathbb{R}$ 
	and $a\in\mathcal{N}_{s_0,M,0}^{m}$. 
	Then, for any $s\in \mathbb{R}$, one has
	\begin{equation}\label{actionSob}
		\|T_{a}h\|_{s-m,M} \lesssim |a|_{\mathcal{N}^{m}_{s_0,M,0}}\|h\|_{s,M}\,,
		\qquad \forall h\in H^{s}(\mathbb{T}^{2};\mathbb{C})\,.
	\end{equation}

	{
	\noindent
	$(iii)$ Let $\vartheta>1$ and let $a=a(x)\in H^\vartheta(\mathbb T^2)$
	be independent of $\xi$. Then, for every $r\in\mathbb R$ and every
	$M\geq1$,
	\[
		\|T_a h\|_{r,M}\lesssim_{r,\vartheta}
		\|a\|_{H^\vartheta}\|h\|_{r,M},
		\qquad \forall h\in H^r(\mathbb T^2).
	\]
	The constant is uniform in $M$.
	}
	
	\noindent
	${(iv)}$ Let $s_0>2$, $m\in \mathbb{R}$, {$\rho\geq0$}, and 
	$a\in\mathcal{N}_{s_0+\rho,M,0}^{m}$.
	For $0<\eps_2\leq \eps_1<1/2$ and any 
	$h\in H^{s}(\mathbb{T}^{2};\mathbb{C})$, 
	we define the remainder operator
	\begin{equation}\label{natale}
		R_{a}h := \frac{1}{(2\pi)^{2}}\sum_{j\in \mathbb{Z}^{2}}e^{\mathrm{i} j\cdot x}
		\sum_{k\in\mathbb{Z}^{2}}
		\big(\chi_{\eps_1}-\chi_{\eps_2}\big)\Big(\frac{|j-k|}{\langle j+k\rangle}\Big)
		\widehat{a}\Big(j-k,\frac{j+k}{2}\Big)\widehat{h}(k)\,,
	\end{equation}
	where $\chi_{\eps_1}, \chi_{\eps_2}$ are defined as in \eqref{cutofffunctepsilon}. 
	Then one has the regularizing estimate
	\begin{equation}\label{diffQuanti}
		\|R_a h\|_{s+\rho-m,M} \lesssim
		\|h\|_{s,M}|a|_{\mathcal{N}^{m}_{\rho+s_0,M,0}}\,,
		\qquad \forall h\in H^{s}(\mathbb{T}^{2};\mathbb{C})\,.
	\end{equation}
\end{lemma}

\begin{proof}
	$(i)$ This is a straightforward adaptation from \cite{FIJMPA}.
	
	\smallskip
	\noindent $(ii)$ 
	First of all, notice that
	{$\|a(\cdot,\xi)\|_{s_0,M} \lesssim \langle\xi\rangle_M^m|a|_{\mathcal{N}^m_{s_0,M,0}},$}
	from which we infer that
	\begin{equation}\label{virus5}
		|\widehat{a}(j,\xi)| \lesssim \langle\xi\rangle_M^{m}|a|_{\mathcal{N}_{s_0,M,0}^{m}}\langle j\rangle_M^{-s_0}\,.
	\end{equation}
	Moreover, since $0<\eps<1/4$, we note that for $\xi,\eta\in \mathbb{Z}^2$,
	\begin{equation}\label{equixieta}
		\chi_{\eps}\left(\frac{|\xi-\eta|}{\langle \xi+\eta\rangle}\right) \neq 0 \quad \Rightarrow \quad
		\begin{cases}
			(1-\tilde{\eps})|\xi| \leq (1+\tilde{\eps})|\eta| \\
			(1-\tilde{\eps})|\eta| \leq (1+\tilde{\eps})|\xi| \,,
		\end{cases}
	\end{equation}
	where $0<\tilde{\eps}<4/5$. Indeed, recalling \eqref{cutofffunct}-\eqref{cutofffunctepsilon},
	we have $|\xi| \leq (1+\tfrac{8}{5}\eps)|\eta| + \tfrac{8}{5}\eps|\xi| + \tfrac{8}{5}\eps$
	which, for $|\eta|\neq 0$, implies
	$(1-\tfrac{8}{5}\eps)|\xi| \leq (1+\tfrac{16}{5}\eps)|\eta|$.
	This implies the first condition in \eqref{equixieta} in the case $|\eta|\neq 0$.
	The case $|\eta|=0$ is trivial since the definition of the cut-off function $\chi_{\eps}$ and $\eps<1/4$
	implies that $|\xi|=0$ as well. The second condition in \eqref{equixieta} is similar.
	
	As a consequence, we have the equivalence
	$\langle\xi+\eta\rangle \sim \langle\xi\rangle$: on one hand,
	$|\xi+\eta| \leq |\xi|+|\eta| \leq (1+C)|\xi|$ for some $C=C(\tilde{\eps})>0$; 
	on the other hand, $|\xi| = \tfrac{1}{2}|\xi-\eta+\xi+\eta| \leq \tfrac{1}{2}|\xi-\eta| + \tfrac{1}{2}|\xi+\eta| \leq \tfrac{1}{2}\tfrac{8\eps}{5}\langle\xi+\eta\rangle + \tfrac{1}{2}\langle\xi+\eta\rangle$.
	Therefore, taking $s_0>2$, we can estimate
	\begin{equation}\label{natale2}
		\begin{aligned}
			\|T_{a}h\|^{2}_{s-m,M}
			&\lesssim
			\sum_{\xi\in\mathbb{Z}^{2}}
			\langle\xi\rangle^{2(s-m)}_M\Big|\sum_{\eta\in \mathbb{Z}^{2}}
			\chi_{\eps}\left(\frac{|\xi-\eta|}{\langle \xi+\eta\rangle}\right)\widehat{a}\Big(\xi-\eta, \frac{\xi+\eta}{2}\Big)\widehat{h}(\eta)\Big|^{2}\\
			&\stackrel{\mathclap{\eqref{virus5}, \eqref{equixieta}}}{\lesssim}\,\,\,\,\,\,
			\sum_{\xi\in \mathbb{Z}^{2}}\langle\xi\rangle^{-2m}_M\Big(\sum_{\eta\in \mathbb{Z}^{2}}
			\frac{\langle\xi\rangle^{m}_M}{\langle\xi-\eta\rangle^{s_0}_M}
			|\widehat{h}(\eta)|\langle \eta\rangle^{s}_M
			\Big)^{2} |a|_{\mathcal{N}_{s_0,M,0}^{m}}^{2}\\
			&\lesssim|a|^{2}_{\mathcal{N}^{m}_{s_0,M,0}} \sum_{\xi\in\mathbb{Z}^2}\Big(\sum_{\eta\in\mathbb{Z}^2}|\widehat{h}(\eta)|\langle\eta\rangle^s_M\frac{1}{\langle\xi-\eta\rangle_M^{s_0}}\Big)^2\\
			&\lesssim|a|^{2}_{\mathcal{N}^{m}_{s_0,M,0}}\big\|\big(|\widehat{h}(\xi)|\langle\xi\rangle_M^s\big) * \langle\xi\rangle_M^{-s_0}\big\|_{\ell^2(\mathbb{Z}^2)}^2 \\
			&\leq |a|^{2}_{\mathcal{N}^{m}_{s_0,M,0}}\|\widehat{h}(\xi)\langle\xi\rangle_M^s\|_{\ell^2(\mathbb{Z}^2)}^2\|\langle{\xi}\rangle_M^{-s_0}\|_{\ell^1(\mathbb{Z}^2)}^2\\
			&\lesssim \|h\|_{s,M}^{2}|a|^{2}_{\mathcal{N}^{m}_{s_0,M,0}}\,,
		\end{aligned}
	\end{equation}
	where in the penultimate passage we used Young's inequality for sequences, and in the last one that $\langle\xi\rangle_M^{-s_0} \leq \langle\xi\rangle^{-s_0}$ is in $\ell^1(\mathbb{Z}^2)$ since $s_0>2$.

	{
	\smallskip
	\noindent $(iii)$ Since $a$ is independent of $\xi$, the Fourier
	coefficient of $T_ah$ is bounded, up to the harmless normalization
	factor, by
	\[
		\sum_{\eta\in\mathbb Z^2}
		|\widehat a(\xi-\eta)|\,|\widehat h(\eta)|.
	\]
	On the support of the paradifferential cut-off,
	$\langle\xi\rangle_M\sim\langle\eta\rangle_M$, uniformly for
	$M\geq1$. Hence Young's inequality gives
	\[
		\|T_ah\|_{r,M}
		\lesssim_r \|\widehat a\|_{\ell^1}\|h\|_{r,M}.
	\]
	Finally, Cauchy--Schwarz yields
	\[
		\|\widehat a\|_{\ell^1}
		\leq
		\left(\sum_{p\in\mathbb Z^2}\langle p\rangle^{-2\vartheta}\right)^{1/2}
		\|a\|_{H^\vartheta}
		\lesssim_\vartheta\|a\|_{H^\vartheta},
	\]
	where the  sum is finite  for $\vartheta>1$ in
	dimension two. This proves the claim.
	}
	
	\smallskip
	\noindent
	${(iv)}$ Notice that the set of $\xi,\eta$ such that
	$(\chi_{\eps_1}-\chi_{\eps_2})\big(|\xi-\eta|/\langle\xi+\eta\rangle\big) = 0$
	contains the set where both cut-off functions equal $1$ or both equal $0$, namely
	\[
	|\xi-\eta| \geq \frac{8}{5}\eps_1 \langle \xi+\eta\rangle\quad \text{or}\quad
	|\xi-\eta| \leq \frac{5}{4}\eps_2 \langle\xi+\eta\rangle\,.
	\]
	Therefore, $(\chi_{\eps_1}-\chi_{\eps_2})\big(|\xi-\eta|/\langle\xi+\eta\rangle\big) \neq 0$
	implies
	\begin{equation}\label{condizio}
		\frac{5}{4}\eps_2 \langle\xi+\eta\rangle \leq |\xi-\eta| \leq \frac{8}{5}\eps_1 \langle\xi+\eta\rangle\,.
	\end{equation}
	For $\xi\in \mathbb{Z}^{2}$ 
	we denote by $\mathcal{A}(\xi)$ the set of 
	$\eta\in \mathbb{Z}^{2}$ such that \eqref{condizio} holds. Moreover (reasoning as in \eqref{virus5}), since
	$a\in \mathcal{N}_{s_0+\rho,M,0}^{m}$, we have that
	\begin{equation}\label{virus6}
		|\widehat{a}(j,\xi)| \lesssim \langle\xi\rangle_M^{m}|a|_{\mathcal{N}_{s_0+\rho,M,0}^{m}} \langle j\rangle_M^{-(s_0+\rho)}\,.
	\end{equation}
	To estimate the remainder in \eqref{natale}, we reason as in \eqref{natale2}.
	By relying on \eqref{condizio} and the equivalence of weights on the support $\mathcal{A}(\xi)$, we extract the required decay:
	\begin{equation}\label{stimarestoresto}
		\begin{aligned}
			\|R_ah\|_{s+\rho-m,M}^{2}
			&\lesssim \sum_{\xi\in\mathbb{Z}^{2}}
			\langle\xi\rangle^{2(s+\rho-m)}_M\Big| \sum_{\eta\in \mathcal{A}(\xi)}(\chi_{\eps_1}-\chi_{\eps_2})\left(\frac{|\xi-\eta|}{\langle\xi+\eta\rangle}\right)\widehat{a}\Big(\xi-\eta, \frac{\xi+\eta}{2}\Big)\widehat{h}(\eta)\Big|^{2}\\
			&\stackrel{\mathclap{\eqref{virus6}}}{\lesssim}
			\sum_{\xi\in \mathbb{Z}^{2}} \langle \xi\rangle^{-2m}_M
			\Big(\sum_{\eta\in \mathcal{A}(\xi)} \frac{\langle\xi-\eta\rangle^{\rho}_M\langle\xi+\eta\rangle_M^{m}}{\langle \xi-\eta\rangle_M^{\rho+s_0}} |\widehat{h}(\eta)|\langle \eta\rangle^{s}_M \Big)^{2}|a|_{\mathcal{N}^{m}_{s_0+\rho,M,0}}^{2}\\
			&\lesssim
			\sum_{\xi\in \mathbb{Z}^{2}}
			\Big(\sum_{\eta\in \mathcal{A}(\xi)}
			\frac{|\widehat{h}(\eta)|\langle \eta\rangle^{s}_M}{\langle\xi-\eta\rangle_M^{s_0}}
			\Big)^{2}|a|_{\mathcal{N}^{m}_{s_0+\rho,M,0}}^{2}\\
			&\lesssim \big\|\big(|\widehat{h}(\xi)|\langle\xi\rangle^s_M\big) * \langle\xi\rangle_M^{-s_0}\big\|^2_{\ell^2(\mathbb{Z}^2)}|a|^{2}_{\mathcal{N}^{m}_{\rho+s_0,M,0}}\\
			&\leq \|\widehat{h}(\xi)\langle\xi\rangle^s_M\|^2_{\ell^2(\mathbb{Z}^2)} \|\langle\xi\rangle_M^{-s_0}\|^2_{\ell^1(\mathbb{Z}^2)}|a|^{2}_{\mathcal{N}^{m}_{\rho+s_0,M,0}}\\
			&\lesssim
			\|h\|_{s,M}^{2}|a|^{2}_{\mathcal{N}^{m}_{\rho+s_0,M,0}}\,,
		\end{aligned}
	\end{equation}
	where in the penultimate step we used Young's inequality for sequences, and in the last one we used that $\langle\xi\rangle^{-s_0}_M \leq \langle\xi\rangle^{-s_0}$ is in $\ell^1(\mathbb{Z}^2)$ since $s_0>2$.
\end{proof}

{
\begin{definition}
	Let $\rho\in (0, 2]$ and set
	\[
		n_\rho:=\begin{cases}
		1,&0<\rho\leq1,\\
		2,&1<\rho\leq2.
		\end{cases}
	\]
	Given two symbols $a\in \mathcal{N}^{m_1}_{s_0+\rho,M,n_\rho}$ and $b\in \mathcal{N}^{m_2}_{s_0+\rho,M,n_\rho}$, we define their Bony-Weyl composition as
	\begin{equation}\label{cancelletto}
		a\#_{\rho} b = \begin{cases}
			ab & \text{if }\rho\in(0,1], \\
			ab + \frac{1}{2\mathrm{i}}\{a,b\} & \text{if }\rho\in (1,2],
		\end{cases}
	\end{equation}
	where we denoted by $\{a,b\} := \nabla_{\xi}a\cdot\nabla_xb - \nabla_xa\cdot\nabla_{\xi}b$ the standard Poisson bracket between symbols.
\end{definition}

\begin{proposition}[\emph{Composition}]\label{prop:compo}
	Fix $s_0>2$, $\rho\in (0,2]$, and $m_1,m_2\in \mathbb{R}$. Then the following holds.
	
	\noindent
	For $a\in \mathcal{N}_{s_0+\rho,M,n_\rho}^{m_1}$ and $b\in\mathcal{N}_{s_0+\rho,M,n_\rho}^{m_2}$, we have 
	\begin{equation}\label{composit}
		T_{a}\circ T_{b} = T_{a\#_{\rho}b} + R_{\rho}(a,b)\,,
	\end{equation}
	where $R_\rho(a,b)$ is a remainder operator satisfying, for any $s\in \mathbb{R}$,
	\begin{equation}\label{composit2}
		\|R_\rho(a,b)h\|_{s-m_1-m_2+\rho,M} \lesssim \|h\|_{s,M}\big(|a|_{\mathcal{N}^{m_1}_{s_0+\rho,M,n_\rho}}|b|_{\mathcal{N}^{m_2}_{s_0,M,n_\rho}} + |a|_{\mathcal{N}^{m_1}_{s_0,M,n_\rho}}|b|_{\mathcal{N}^{m_2}_{s_0+\rho,M,n_\rho}}\big)\,.
	\end{equation}
\end{proposition}
	
\begin{proof}
	Let us check the estimate \eqref{composit2} in the  case $\rho\leq 1$. 
	First of all, we note that 
	\begin{align}
		\widehat{(T_{a}T_{b}h)}(\xi) &= \frac{1}{(2\pi)^{2}}\sum_{\eta,\theta\in\mathbb{Z}^{2}} r_1(\xi,\theta,\eta) \widehat{a}\Big(\xi-\theta,\frac{\xi+\theta}{2}\Big)\widehat{b}\Big(\theta-\eta,\frac{\theta+\eta}{2}\Big)\widehat{h}(\eta)\,,\label{def:prodotto1}\\
		\widehat{(T_{ab}h)}(\xi) &= \frac{1}{(2\pi)^{2}}\sum_{\eta,\theta\in\mathbb{Z}^{2}} r_2(\xi,\eta) \widehat{a}\Big(\xi-\theta,\frac{\xi+\eta}{2}\Big)\widehat{b}\Big(\theta-\eta,\frac{\xi+\eta}{2}\Big)\widehat{h}(\eta)\,,\label{def:prodotto2}
	\end{align}
	where $r_1$ and $r_2$ represent the cut-off multipliers derived from \eqref{quantiWeyl}.
	We remark that we can substitute the cut-off function $r_2$ in \eqref{def:prodotto2} with $r_1$ up to smoothing remainders. This follows because one can treat the cut-off difference $r_1(\xi,\theta,\eta)-r_2(\xi,\eta)$ as done in the proof of \eqref{diffQuanti}. {Precisely, define
	\[
	\begin{aligned}
	\widehat{\mathcal E_{\mathrm{cut}}(a,b)h}(\xi)
	:=\frac{1}{(2\pi)^2}\sum_{\eta,\theta\in\mathbb Z^2}
	&(r_1(\xi,\theta,\eta)-r_2(\xi,\eta))\\
	&\times\widehat a\Big(\xi-\theta,\frac{\xi+\eta}{2}\Big)
	\widehat b\Big(\theta-\eta,\frac{\xi+\eta}{2}\Big)
	\widehat h(\eta).
	\end{aligned}
	\]
	The same proof gives $\mathcal E_{\mathrm{cut}}(a,b)$ the bound \eqref{composit2}, so it remains to estimate the part with the common cut-off $r_1$.}
	
	Write $\xi+\theta = \xi+\eta+(\theta-\eta)$.
	By Taylor expanding the symbols at $\frac{\xi+\eta}{2}$, we have
	\begin{equation}\label{expsimbo1}
		\widehat{a}\Big(\xi-\theta,\frac{\xi+\theta}{2}\Big) = \widehat{a}\Big(\xi-\theta,\frac{\xi+\eta}{2}\Big) + \int_0^1\widehat{(\nabla_{\xi}a)}\Big(\xi-\theta, \frac{\xi+\eta}{2}+\sigma \frac{\theta-\eta}{2}\Big)\cdot\frac{\theta-\eta}{2}\,d\sigma\,.
	\end{equation}
	Similarly, one obtains 
	\begin{equation}\label{expsimbo2}
		\widehat{b}\Big(\theta-\eta,\frac{\theta+\eta}{2}\Big) = \widehat{b}\Big(\theta-\eta,\frac{\xi+\eta}{2}\Big) + \int_0^1\widehat{(\nabla_{\xi}b)}\Big(\theta-\eta,\frac{\xi+\eta}{2}+\sigma \frac{\theta-\xi}{2}\Big)\cdot\frac{\theta-\xi}{2}\,d\sigma\,.
	\end{equation}
	By \eqref{expsimbo1} and \eqref{expsimbo2}, we deduce that
	\begin{equation}\label{virus}
		\begin{aligned}
			\widehat{T_{a}T_bh}(\xi) &- \widehat{T_{ab}h}(\xi) = {\widehat{\mathcal E_{\mathrm{cut}}(a,b)h}(\xi)} + \frac{1}{(2\pi)^2}\sum_{\eta,\theta\in\mathbb{Z}^2} r_1(\xi,\theta,\eta)\Big[ \\
			&\qquad \widehat{a}\Big(\xi-\theta,\frac{\xi+\eta}{2}\Big)\int_0^1\widehat{(\nabla_{\xi}b)}\Big(\theta-\eta,\frac{\xi+\eta}{2}+\sigma \frac{\theta-\xi}{2}\Big)\cdot\frac{\theta-\xi}{2}\,d\sigma \\
			&\qquad + \widehat{b}\Big(\theta-\eta,\frac{\xi+\eta}{2}\Big)\int_0^1\widehat{(\nabla_{\xi}a)}\Big(\xi-\theta, \frac{\xi+\eta}{2}+\sigma \frac{\theta-\eta}{2}\Big)\cdot\frac{\theta-\eta}{2}\,d\sigma\\
			&\qquad + \int_0^1\widehat{(\nabla_{\xi}a)}\Big(\xi-\theta, \frac{\xi+\eta}{2}+\sigma \frac{\theta-\eta}{2}\Big)\cdot\frac{\theta-\eta}{2}\,d\sigma \\
			&\qquad \cdot\int_0^1\widehat{(\nabla_{\xi}b)}\Big(\theta-\eta,\frac{\xi+\eta}{2}+\sigma \frac{\theta-\xi}{2}\Big)\cdot\frac{\theta-\xi}{2}\,d\sigma
			\Big]\widehat{h}(\eta)\,.
		\end{aligned}
	\end{equation}
	Consider the first summand in \eqref{virus}. First of all, we note that $r_1(\xi,\theta,\eta)\neq 0$ implies that
	\begin{equation}\label{virus2} 
		(\theta,\eta) \in \Big\{ \frac{|\xi-\theta|}{\langle\xi+\theta\rangle}\leq \frac{8}{5}\eps \Big\} \bigcap \Big\{ \frac{|\theta-\eta|}{\langle\theta+\eta\rangle}\leq \frac{8}{5}\eps \Big\} =: \mathcal{B}(\xi)\,, \quad \xi\in \mathbb{Z}^{2}\,.
	\end{equation} 
	Moreover, we note that 
	\begin{equation}\label{virus3}
		(\theta,\eta)\in \mathcal{B}(\xi) \implies |\xi|\lesssim|\theta|\,,\quad |\theta|\lesssim|\eta|\,,\quad |\eta|\lesssim|\xi|\,.
	\end{equation}
	We now study the first term in \eqref{virus}. We need to bound its
	Sobolev norm $\|\cdot\|_{s-m_1-m_2+\rho,M}$. Recalling
	\eqref{virus2}, its $\xi$-th Fourier coefficient is
	\begin{equation}\label{virus11}
		\begin{aligned}
			&\frac{1}{(2\pi)^2}\sum_{\eta,\theta\in\mathcal{B}(\xi)} r_1(\xi,\theta,\eta)\Big[
			\widehat{a}\Big(\xi-\theta,\frac{\xi+\eta}{2}\Big)\int_0^1\widehat{(\nabla_{\xi}b)}\Big(\theta-\eta,\frac{\xi+\eta}{2}+\sigma \frac{\theta-\xi}{2}\Big)\cdot\frac{\theta-\xi}{2}\,d\sigma\Big]\widehat{h}(\eta) \\
			&\quad =: \frac{1}{2\pi}\sum_{\eta\in\mathbb{Z}^{2} } \widehat{c}\Big(\xi-\eta,\frac{\xi+\eta}{2}\Big)\widehat{h}(\eta)\,,
		\end{aligned}
	\end{equation}
	where, writing explicitly the cut-off inherited from $r_1$, we have defined the symbol
	\[
	\begin{aligned}
		\widetilde r_1(p,\ell,\zeta)
		&:=\chi_\eps\left(\frac{|p-\ell|}{\langle2\zeta+\ell\rangle}\right)
		\chi_\eps\left(\frac{|\ell|}{\langle2\zeta-p+\ell\rangle}\right),\\
		\widehat{c}(p,\zeta) &:= \frac{1}{2\pi}\sum_{\ell\in \mathbb{Z}^{2}} \widehat{a}(p-\ell,\zeta)
		\int_0^1\widehat{(\nabla_{\xi}b)}
		\Big(\ell,\zeta+\sigma\frac{\ell-p}{2}\Big)
		\cdot\frac{\ell-p}{2}\,
		\widetilde r_1(p,\ell,\zeta)\,d\sigma\,,
		\\[-2pt]
		&\hspace{20mm}p\in\mathbb{Z}^{2},\qquad\zeta\in\mathbb{R}^{2}.
	\end{aligned}
	\]
	Reasoning as in \eqref{virus3}, we can deduce that on the support of $\widetilde r_1$ one has
	\begin{equation}\label{virus12}
		\langle2\zeta+p\rangle\sim\langle2\zeta-p\rangle
		\sim\langle2\zeta\rangle,\qquad
		\langle p-\ell\rangle\lesssim\langle\zeta\rangle.
	\end{equation}
	Indeed, the support condition implies $(\theta,\eta)\in \mathcal{B}(\xi)$ by setting 
	\begin{equation}\label{virus13}
		2\xi = 2\zeta+p\,,\quad 2\theta = 2\ell+2\zeta-p\,,\quad 2\eta = 2\zeta-p\,.
	\end{equation}
	Hence \eqref{virus12} follows by \eqref{virus3} and the definition of the cut-offs. Since $0<\rho\leq1$, the second relation in \eqref{virus12} gives
	\[
		|p-\ell|\lesssim
		\langle p-\ell\rangle_M^\rho
		\langle\zeta\rangle_M^{1-\rho}.
	\]
	Therefore we deduce
	\begin{equation*}
		\begin{aligned}
			|\widehat{c}(p,\zeta)| &\lesssim \sum_{\ell\in \mathbb{Z}^2}\langle p-\ell\rangle_M^{-s_0-\rho}\langle\zeta\rangle_M^{m_1}|a|_{\mathcal{N}^{m_1}_{s_0+\rho,M,0}} |p-\ell|\langle\ell\rangle_M^{-s_0}\langle\zeta\rangle_M^{m_2-1}|b|_{\mathcal{N}^{m_2}_{s_0,M,1}}\\
			&\lesssim \langle \zeta\rangle_M^{m_1+m_2-\rho}|a|_{\mathcal{N}^{m_1}_{s_0+\rho,M,0}}|b|_{\mathcal{N}^{m_2}_{s_0,M,1}}
			\sum_{\ell\in \mathbb{Z}^2}\langle p-\ell\rangle_M^{-s_0}\langle\ell\rangle_M^{-s_0}\\
			&\lesssim \langle \zeta\rangle_M^{m_1+m_2-\rho}|a|_{\mathcal{N}^{m_1}_{s_0+\rho,M,0}}|b|_{\mathcal{N}^{m_2}_{s_0,M,1}}\langle p \rangle_M^{-s_0},
		\end{aligned}
	\end{equation*}
	where we have used that there exists a constant $C=C(s_0)>0$ such that for every $M\gg 1$ and every $p\in\mathbb{Z}^2$ one has 
	\[ 
	\sum_{\ell\in\mathbb{Z}^2}\langle p-\ell \rangle_M^{-s_0}\,\langle \ell \rangle_M^{-s_0} \le C\,\langle p \rangle_M^{-s_0}. 
	\] 
	To verify this, set 
    \[
    S_M(p) := \sum_{\ell\in\mathbb{Z}^2}\langle p-\ell \rangle_M^{-s_0}\,\langle \ell \rangle_M^{-s_0}.
    \] 
    The triangle inequality implies 
	\[ \langle p \rangle_M \le \langle \ell \rangle_M + \langle p-\ell \rangle_M \le 2\max\{\langle \ell \rangle_M, \langle p-\ell \rangle_M\}, \] 
	hence for each $\ell$, at least one of $\langle \ell \rangle_M$ or $\langle p-\ell \rangle_M$ is $\ge \tfrac{1}{2}\langle p \rangle_M$. Splitting the sum accordingly yields
	\[ S_M(p) {\leq} \sum_{\langle \ell \rangle_M \ge \tfrac{1}{2}\langle p \rangle_M} \langle p-\ell \rangle_M^{-s_0}\langle \ell \rangle_M^{-s_0} + \sum_{\langle p-\ell \rangle_M \ge \tfrac{1}{2}\langle p \rangle_M} \langle p-\ell \rangle_M^{-s_0}\langle \ell \rangle_M^{-s_0}. \]
	In the first sum, $\langle \ell \rangle_M^{-s_0} \le 2^{s_0}\langle p \rangle_M^{-s_0}$, so 
	\[ \sum_{\langle \ell \rangle_M \ge \tfrac{1}{2}\langle p \rangle_M} \langle p-\ell \rangle_M^{-s_0}\langle \ell \rangle_M^{-s_0} \le 2^{s_0}\langle p \rangle_M^{-s_0}\sum_{\ell\in\mathbb{Z}^2}\langle p-\ell \rangle_M^{-s_0}. \]
	By relabeling, the inner sum equals $\sum_{m\in\mathbb{Z}^2}\langle m \rangle_M^{-s_0}$. For $M\ge 1$, the map $M\mapsto \langle m \rangle_M$ is non-decreasing, hence
	\[ \sum_{m\in\mathbb{Z}^2}\langle m \rangle_M^{-s_0} \le \sum_{m\in\mathbb{Z}^2}\langle m \rangle^{-s_0} < \infty.\]
	The second sum is handled symmetrically.
	
	At this point, one concludes as done in \eqref{natale2}. The other terms in \eqref{virus} are dealt with in a similar way, concluding the proof for $\rho\in(0,1]$. 

	We finally consider $1<\rho\leq2$.  Set
	\[
		p:=\xi-\eta,\qquad \ell:=\theta-\eta,
		\qquad \zeta:=\frac{\xi+\eta}{2}.
	\]
	Thus the two frequency arguments in \eqref{def:prodotto1} are
	$\zeta+\ell/2$ and $\zeta+(\ell-p)/2$. Taylor's formula gives
	\begin{align*}
	\widehat a&\left(p-\ell,\zeta+\frac\ell2\right)
	=A_0+A_1+E_a,\\
	A_0&:=\widehat a(p-\ell,\zeta),
	\qquad
	A_1:=\nabla_\xi\widehat a(p-\ell,\zeta)\cdot\frac\ell2,\\
	E_a&:=\sum_{|\alpha|=2}\frac{2}{\alpha!}
	\left(\frac\ell2\right)^\alpha
	\int_0^1(1-t)\partial_\xi^\alpha\widehat a
	\left(p-\ell,\zeta+t\frac\ell2\right)\,dt,
	\end{align*}
	and, similarly,
	\begin{align*}
	\widehat b&\left(\ell,\zeta+\frac{\ell-p}{2}\right)
	=B_0+B_1+E_b,\\
	B_0&:=\widehat b(\ell,\zeta),
	\qquad
	B_1:=\nabla_\xi\widehat b(\ell,\zeta)
	\cdot\frac{\ell-p}{2},\\
	E_b&:=\sum_{|\beta|=2}\frac{2}{\beta!}
	\left(\frac{\ell-p}{2}\right)^\beta
	\int_0^1(1-t)\partial_\xi^\beta\widehat b
	\left(\ell,\zeta+t\frac{\ell-p}{2}\right)\,dt.
	\end{align*}
	{With the Fourier normalization \eqref{complex-uU},
	the zeroth-order identity is
	\[
	 \frac{1}{2\pi}\sum_{\ell\in\mathbb Z^2}A_0B_0
	 =\widehat{ab}(p,\zeta).
	\]
	Using $q\widehat f(q)=-\mathrm{i}\widehat{\nabla_xf}(q)$, the sum of
	the two first-order terms satisfies}
	\[
	{\frac{1}{2\pi}}\sum_{\ell\in\mathbb Z^2}(A_1B_0+A_0B_1)
	=\frac{1}{2\mathrm{i}}\widehat{\{a,b\}}(p,\zeta).
	\]
	{
	The cut-off remainder also contains
	\[
	 \widehat{\mathcal E_{\mathrm{cut}}^{(1)}(a,b)h}(\xi)
	 :=\frac{1}{(2\pi)^2}\sum_{\eta,\theta\in\mathbb Z^2}
	 \big(r_1(\xi,\theta,\eta)-r_2(\xi,\eta)\big)
	 (A_1B_0+A_0B_1)\widehat h(\eta).
	\]
	Here $p,\ell,\zeta$ have the values fixed above in terms of
	$\xi,\theta,\eta$. The support of $r_1-r_2$ and the proof of
	\eqref{diffQuanti} give both $\mathcal E_{\mathrm{cut}}$ and
	$\mathcal E_{\mathrm{cut}}^{(1)}$ the bound \eqref{composit2}.}

	It remains to bound the terms containing $E_a$, $E_b$, or the
	product $A_1B_1$. On the support of the cut-offs one has
	\[
	\langle\ell\rangle+\langle p-\ell\rangle
	\lesssim\langle\zeta\rangle.
	\]
	For $1<\rho\leq2$, the factors generated by the Taylor expansion
	satisfy
	\begin{align*}
	|\ell|^2
	&\lesssim \langle\ell\rangle_M^\rho
	\langle\zeta\rangle_M^{2-\rho},\\
	|p-\ell|^2
	&\lesssim \langle p-\ell\rangle_M^\rho
	\langle\zeta\rangle_M^{2-\rho},\\
	|\ell|\,|p-\ell|
	&\lesssim \langle\zeta\rangle_M^{2-\rho}
	\big(\langle\ell\rangle_M^\rho
	+\langle p-\ell\rangle_M^\rho\big).
	\end{align*}
	Any additional factor occurring in $A_1E_b$, $E_aB_1$, or
	$E_aE_b$ is absorbed by the corresponding additional
	$\xi$-derivative, using again the support relation above.
	{Set
	\[
	 \mathscr K_2:=A_1B_1+A_0E_b+E_aB_0+A_1E_b+E_aB_1+E_aE_b,
	 \qquad
	 \widehat c_2(p,\zeta):=\frac{1}{2\pi}
	 \sum_{\ell\in\mathbb Z^2}\widetilde r_1(p,\ell,\zeta)\,
	 \mathscr K_2.
	\]}
	Hence the Fourier coefficient $\widehat c_2(p,\zeta)$ of the sum
	of all second-order remainders satisfies
	\begin{align*}
	|\widehat c_2(p,\zeta)|
	\lesssim{}&\langle\zeta\rangle_M^{m_1+m_2-\rho}
	\langle p\rangle_M^{-s_0}\times\Big(
	|a|_{\mathcal N^{m_1}_{s_0+\rho,M,2}}
	|b|_{\mathcal N^{m_2}_{s_0,M,2}}
	+|a|_{\mathcal N^{m_1}_{s_0,M,2}}
	|b|_{\mathcal N^{m_2}_{s_0+\rho,M,2}}
	\Big).
	\end{align*}
	Here we used the convolution estimate proved after
	\eqref{virus12}. The discrete Young inequality, exactly as in
	\eqref{natale2}, gives \eqref{composit2}. The remainders produced
	by changing the cut-offs obey the same estimate, since their
	support makes one of the spatial frequencies comparable with
	$\langle\zeta\rangle$. This concludes the proof for
	$1<\rho\leq2$.
\end{proof}
}

\begin{lemma}[\emph{Paraproduct}]\label{lem:paraproduct}
	Fix $s_0>1$ and 
	let $f,g\in H^{s}(\mathbb{T}^2;\mathbb{C})$ for $s\geq s_0$. Then
	\begin{equation}\label{eq:paraproduct}
		fg = T_{f}g + T_{g}f + \mathcal{R}(f,g)\,,
	\end{equation}
	where for any $0\leq \rho\leq s-s_0$, the remainder satisfies the bound
	\begin{equation}\label{eq:paraproduct22}
		\|\mathcal{R}(f,g)\|_{s+\rho,M} \lesssim \|f\|_{s,M}\|g\|_{s_0+\rho,M} + \|g\|_{s,M}\|f\|_{s_0+\rho,M}\,.
	\end{equation}
\end{lemma}

\begin{proof}
	Notice that the standard product in frequency reads
	\begin{equation}\label{prodFG}
		\widehat{(fg)}(\xi) = {\frac{1}{2\pi}}\sum_{\eta\in\mathbb{Z}^{2}}\widehat{f}(\xi-\eta)\widehat{g}(\eta)\,.
	\end{equation}
	Consider the cut-off function $\chi_{\eps}$ defined in \eqref{cutofffunctepsilon}
	and define a new cut-off {$\Theta : \mathbb{R}^2 \times \mathbb{R}^2 \to [-1,1]$} via the identity
	\begin{equation}\label{cutoffTHETA}
		1 = \chi_{\eps}\left(\frac{|\xi-\eta|}{\langle\xi+\eta\rangle}\right) + \chi_{\eps}\left(\frac{|\eta|}{\langle 2\xi-\eta\rangle}\right) + \Theta(\xi,\eta)\,.
	\end{equation}
	Recalling \eqref{prodFG} and \eqref{quantiWeyl}, we note that 
	\begin{equation}\label{paraprodFG}
		\begin{aligned}
			\widehat{(T_fg)}(\xi) &= {\frac{1}{2\pi}}\sum_{\eta\in\mathbb{Z}^{2}} \chi_{\eps}\left(\frac{|\xi-\eta|}{\langle\xi+\eta\rangle}\right) \widehat{f}(\xi-\eta)\widehat{g}(\eta)\,, \\
			\widehat{(T_g f)}(\xi) &= {\frac{1}{2\pi}}\sum_{\eta\in\mathbb{Z}^{2}} \chi_{\eps}\left(\frac{|\eta|}{\langle 2\xi-\eta\rangle}\right) \widehat{f}(\xi-\eta)\widehat{g}(\eta)\,. 
		\end{aligned}
	\end{equation}
	Setting $\mathcal{R} := \mathcal{R}(f,g)$, its Fourier transform is given by 
	\begin{equation}\label{pararestoFG}
		\widehat{\mathcal{R}}(\xi) = {\frac{1}{2\pi}}\sum_{\eta\in\mathbb{Z}^{2}} \Theta(\xi,\eta) \widehat{f}(\xi-\eta)\widehat{g}(\eta)\,.
	\end{equation}
	To obtain the second line in \eqref{paraprodFG}, one uses \eqref{quantiWeyl} and performs the change of variables $\xi-\eta\rightsquigarrow \eta$.
	{
	By the definition of the cut-off function $\Theta(\xi,\eta)$, there exists a constant $C_\eps\geq1$ such that
	\begin{equation}\label{condTheta}
		\Theta(\xi,\eta)\neq0
		\quad\Longrightarrow\quad
		C_\eps^{-1}\langle\xi-\eta\rangle
		\leq\langle\eta\rangle
		\leq C_\eps\langle\xi-\eta\rangle.
	\end{equation}
	Indeed, if one of the two frequencies is sufficiently smaller than the other, the corresponding paraproduct cut-off equals $1$ and the other one equals $0$, hence $\Theta=0$. Therefore $\Theta\neq0$ implies the comparability in \eqref{condTheta}.
	}
	Because the two frequencies are equivalent, we have $\langle\xi\rangle_M \leq \langle\xi-\eta\rangle_M + \langle\eta\rangle_M \lesssim \langle\xi-\eta\rangle_M \sim \langle\eta\rangle_M$. This allows us to symmetrically split the Sobolev weight:
	\begin{equation}\label{weight-split}
		\langle\xi\rangle_M^{s+\rho} \lesssim \langle\xi-\eta\rangle_M^{s+\rho} + \langle\eta\rangle_M^{s+\rho} \lesssim \langle\xi-\eta\rangle_M^s\langle\eta\rangle_M^\rho + \langle\xi-\eta\rangle_M^\rho\langle\eta\rangle_M^s\,.
	\end{equation}
	Applying \eqref{weight-split} to the Fourier transform of the remainder, we derive
	\begin{equation*}
		\begin{aligned}
			\|\mathcal{R}(f,g)\|_{s+\rho,M} &= \big\| \langle\xi\rangle_M^{s+\rho} \widehat{\mathcal{R}}(\xi) \big\|_{\ell^2} \\
			&\lesssim \Big\| \sum_{\eta\in\mathbb{Z}^{2}} |\widehat{f}(\xi-\eta)||\widehat{g}(\eta)| \big( \langle\xi-\eta\rangle_M^s\langle\eta\rangle_M^\rho + \langle\xi-\eta\rangle_M^\rho\langle\eta\rangle_M^s \big) \Big\|_{\ell^2} \\
			&\leq \big\| \big( \langle\xi\rangle_M^{s}|\widehat{f}| \big) * \big( \langle\xi\rangle_M^{\rho}|\widehat{g}| \big) \big\|_{\ell^2} + \big\| \big( \langle\xi\rangle_M^{\rho}|\widehat{f}| \big) * \big( \langle\xi\rangle_M^{s}|\widehat{g}| \big) \big\|_{\ell^2} \,.
		\end{aligned}
	\end{equation*}
	By Young's inequality ($\|A*B\|_{\ell^2} \le \|A\|_{\ell^2}\|B\|_{\ell^1}$), we obtain
	\begin{equation*}
		\|\mathcal{R}(f,g)\|_{s+\rho,M} \lesssim \|\langle\xi\rangle_M^{s}\widehat{f}\|_{\ell^2} \|\langle\xi\rangle_M^{\rho}\widehat{g}\|_{\ell^1} + \|\langle\xi\rangle_M^{\rho}\widehat{f}\|_{\ell^1} \|\langle\xi\rangle_M^{s}\widehat{g}\|_{\ell^2} \,.
	\end{equation*}
	Finally, by the Cauchy-Schwarz inequality, since $s_0 > 1$, we can embed the $\ell^1$ norm into a weighted $\ell^2$ norm
	\[
		\|\langle\xi\rangle_M^{\rho}\widehat{g}\|_{\ell^1} \leq \sum_{\xi\in\mathbb{Z}^2} \langle\xi\rangle_M^{s_0+\rho}|\widehat{g}(\xi)| \langle\xi\rangle_M^{-s_0} \leq \|\langle\xi\rangle_M^{s_0+\rho}\widehat{g}\|_{\ell^2} \Big\| \langle\xi\rangle_M^{-s_0} \Big\|_{\ell^2} \lesssim \|g\|_{s_0+\rho,M}\,,
	\]
	and similarly $\|\langle\xi\rangle_M^{\rho}\widehat{f}\|_{\ell^1} \lesssim \|f\|_{s_0+\rho,M}$. Plugging these bounds into the previous estimate yields
	\[
		\|\mathcal{R}(f,g)\|_{s+\rho,M} \lesssim \|f\|_{s,M} \|g\|_{s_0+\rho,M} + \|f\|_{s_0+\rho,M} \|g\|_{s,M}\,,
	\]
	which concludes the proof.
\end{proof}

As a corollary, we deduce the tame estimate in $\| \cdot \|_{s, M}$ norm for the product of two functions.

{
\begin{cor}\label{cor:tame:prod}
Let $s\geq\widetilde s>1$ and $u,v\in H^s(\T^2)$. Then, there
exists a constant $C=C(s,\widetilde s)>0$, independent of $M$, such
that
\[
\|uv\|_{s,M}\leq C(s,\widetilde s)
\big(\|u\|_{s,M}\|v\|_{\widetilde s,M}
+\|u\|_{\widetilde s,M}\|v\|_{s,M}\big).
\]
\end{cor}

\begin{proof}
By the decomposition \eqref{eq:paraproduct},
\[
 uv=T_uv+T_vu+\mathcal R(u,v).
\]
Item $(iii)$ of Lemma \ref{azione}, applied with
$\vartheta=\widetilde s$ and $r=s$, gives
\[
 \|T_uv\|_{s,M}
 \lesssim \|u\|_{H^{\widetilde s}}\|v\|_{s,M}
 \leq \|u\|_{\widetilde s,M}\|v\|_{s,M},
\]
and, symmetrically,
\[
 \|T_vu\|_{s,M}
 \lesssim \|v\|_{\widetilde s,M}\|u\|_{s,M}.
\]
Finally, \eqref{eq:paraproduct22}, with $\rho=0$ and
$s_0=\widetilde s$, yields
\[
 \|\mathcal R(u,v)\|_{s,M}
 \lesssim
 \|u\|_{s,M}\|v\|_{\widetilde s,M}
 +\|v\|_{s,M}\|u\|_{\widetilde s,M}.
\]
Adding the last three estimates proves the claim.
\end{proof}
}

\subsection{Bony paralinearization and change of quantization}\label{app:bony}

{We introduce the Littlewood--Paley decomposition. We consider the cut-off function $\chi$ defined in \eqref{cutofffunct} and}
\[
{\chi_k(\xi)=\chi(2^{-k}|\xi|), \qquad
\varphi_0:=\chi_0,\qquad \varphi_k:=\chi_k-\chi_{k-1}\quad (k\geq1).}
\]
{We introduce the operators
\[
S_k u=\mathcal{F}^{-1}(\chi(2^{-k}|\xi|)\,\hat{u}(\xi)),
\quad k\in\mathbb Z,
\qquad \mathrm{and} \qquad \Delta_k=S_k-S_{k-1}.
\]}

\begin{remark}\label{rem:weighted-equivalence}
We observe that
\begin{equation}\label{equiv1}
\| u \|^2_{s, M}\le 2^s( M^{2s} \| u\|_{L^2}^2+\| u\|_{H^s}^2) \le 2^{s+1} \| u\|_{s, M}^2.
\end{equation}
\end{remark}
Given a symbol $a\in\mathcal{N}_s^m$, we consider the Bony \textbf{standard} quantization
\begin{equation}\label{quantiBony}
{
\mathcal{T}_a u=\frac{1}{(2\pi)^{2}}\sum_{j, q\in \mathbb{Z}^{2}} e^{\mathrm{i} j\cdot x} \Psi(j-q, q)\,
\widehat{a}(j-q,q) \,\widehat{u}(q).}
\end{equation}
for an admissible cut-off function $\Psi$. {The  relation between this quantization and the Bony--Weyl quantization $T_a$ used in the paper is established below in Lemma \ref{lem:cambiamo}.}\\
By choosing
\begin{equation}\label{cutoff:standard}
\Psi(j, q)= \sum_{k=0}^{+\infty} \chi_{k-3}(j) \varphi_k(q)
\end{equation}
we have that
\[
\mathcal{T}_a u=(S_{-3} a) (S_0 u)+\sum_{k=1}^{+\infty} (S_{k-3} a) (\Delta_k u).
\]
{The proof of item $(iii)$ of Lemma \ref{azione} applies
verbatim to this standard quantization. Thus, if $a=a(x)$ and
$\vartheta>1$, then, for every $r\in\mathbb R$,
\[
 \|\mathcal T_a h\|_{r,M}
 \lesssim_{r,\vartheta}\|a\|_{H^\vartheta}\|h\|_{r,M},
\]
uniformly for $M\geq1$.}

{
The standard Moser composition estimate, combined with
Remark \ref{rem:weighted-equivalence}, gives the corresponding
estimate in the weighted norm.

\begin{lemma}[Moser estimates]\label{lem:Moser:sM}
Let $r\geq0$, and let
$F\in C^\infty(\C;\C)$ (in the real sense) satisfy $F(0)=0$. If
$u\in H^r(\T^2;\C)\cap L^\infty(\T^2;\C)$ and
$\|u\|_{L^\infty}\leq 1$, then, for every $M\geq1$,
\begin{equation}\label{Moser2:base}
 \|F(u)\|_{r,M}\leq C(F,r,\|u\|_{L^{\infty}})\|u\|_{r,M}.
\end{equation}
If moreover, for some $N\in\N$, $N\geq1$,
\[
 D^jF(0)=0,\qquad j=1,\ldots,N,
\]
where $D^jF$ denotes the $j$-th real Fr\'echet derivative, then
\begin{equation}\label{Moser2}
 \|F(u)\|_{r,M}\leq C(F,r,N)\|u\|_{L^{\infty}}^N\|u\|_{r,M}.
\end{equation}
Both constants are uniform for $M\geq1$.
\end{lemma}

\begin{proof}
The standard Moser estimate, for $r\geq 0$ on $\T^2$,  gives
\[
 \|F(u)\|_{H^r}\leq C(F,r,\|u\|_{L^{\infty}})\|u\|_{H^r};
\]
see, for instance, \cite[Theorem~2.87]{BCD}. Applying
Remark \ref{rem:weighted-equivalence} first to $F(u)$ and then to
$u$ proves \eqref{Moser2:base}.

For the second estimate, set $\lambda:=\|u\|_{L^\infty}$. The case
$\lambda=0$ is immediate. Otherwise set
\[
 v:=\lambda^{-1}u,
 \qquad
 F_\lambda(z):=\lambda^{-(N+1)}F(\lambda z).
\]
The vanishing assumptions and Taylor's formula show that the family
$F_\lambda$, $0<\lambda\leq 1$, is uniformly bounded in the finitely
many $C^K$ seminorms entering the standard Moser estimate
on the unit ball. Hence \eqref{Moser2:base}, applied to
$(F_\lambda,v)$, gives
\[
 \|F(u)\|_{r,M}
 =\lambda^{N+1}\|F_\lambda(v)\|_{r,M}
 \leq C(F,r)\lambda^N\|u\|_{r,M}
 \leq C(F,r,N)\|u\|_{L^{\infty}}^N\|u\|_{r,M}.
\]
This proves \eqref{Moser2}.
\end{proof}
}

The main result of this section is the following.

{
\begin{theorem}\label{thm:formulaBony}
Let $s\geq\sigma>1$, let $0\leq\rho<\sigma-1$, and let
$F\in C^\infty(\R;\R)$ satisfy $F(0)=0$. Then, for every
real-valued $u\in H^s(\T^2)$,
\[
 \|F(u)-\mathcal T_{F'(u)}u\|_{s+\rho,M}
 \leq C(F,s,\sigma,\rho,\|u\|_{H^\sigma})
 \|u\|_{\sigma,M}\|u\|_{s,M}.
\]
The function $C$ is independent of $M\geq1$.
\end{theorem}

\begin{proof}
Set $ \widetilde F(z):=F(z)-F'(0)z.$ 
Since 
$\mathcal T_{F'(0)}u=F'(0)u$, we get
\[
 F(u)-\mathcal T_{F'(u)}u
 =\widetilde F(u)-\mathcal T_{\widetilde F'(u)}u.
\]
Thus, after replacing $F$ by $\widetilde F$, we may assume
$F(0)=F'(0)=0$.

We use the classical unweighted Bony paralinearization theorem in
its classical Sobolev form. For the admissible quantization
\eqref{quantiBony}, it gives
\begin{equation}\label{eq:Bony-unweighted}
 \|F(u)-\mathcal T_{F'(u)}u\|_{H^{s+\rho}}
 \leq C(F,s,\sigma,\rho,\|u\|_{H^\sigma})
 \|u\|_{H^\sigma}\|u\|_{H^s}.
\end{equation}
This is the standard Bony paralinearization estimate; see, for
instance, \cite[Theorem~5.2.4]{MetivierPara}. For fixed
$s,\sigma,\rho$, the constant depends on $F$ only through finitely
many $C^K$ seminorms on a bounded interval containing the range of
$u$.

It remains only to pass to the weighted norm. Since
$F(0)=F'(0)=0$, Taylor's formula gives
\[
 F(z)=z^2\int_0^1(1-t)F''(tz)\,dt,
 \qquad
 F'(z)=z\int_0^1F''(tz)\,dt.
\]
Choose the intermediate Sobolev index
\[
 \widetilde\sigma:=\frac{1+\sigma-\rho}{2}.
\]
Then $1<\widetilde\sigma<\sigma-\rho$  because
$\rho<\sigma-1$. The Sobolev
embedding $H^{\widetilde\sigma}(\mathbb T^2)\hookrightarrow
L^\infty(\mathbb T^2)$ and the first Taylor formula imply
\[
 \|F(u)\|_{L^2}
 \leq C(F,\|u\|_{H^\sigma})
 \|u\|_{H^{\widetilde\sigma}}\|u\|_{L^2}.
\]
Moreover, since the scalar map $F'$ vanishes at the origin, the
standard Moser estimate gives
\[
 \|F'(u)\|_{H^{\widetilde\sigma}}
 \leq C(F,\widetilde\sigma,\|u\|_{H^\sigma})
 \|u\|_{H^{\widetilde\sigma}}.
\]
The spatial-symbol estimate in item $(iii)$ of Lemma \ref{azione},
applied to the standard quantization as observed after
\eqref{cutoff:standard}, therefore yields
\begin{equation}\label{eq:Bony-L2}
 \|F(u)-\mathcal T_{F'(u)}u\|_{L^2}
 \leq C(F,\widetilde\sigma,\|u\|_{H^\sigma})
 \|u\|_{H^{\widetilde\sigma}}\|u\|_{L^2}.
\end{equation}
Let $R_F(u):=F(u)-\mathcal T_{F'(u)}u$. By
Remark \ref{rem:weighted-equivalence},
\begin{equation}\label{eq:Bony-weight-split}
 \|R_F(u)\|_{s+\rho,M}
 \lesssim M^{s+\rho}\|R_F(u)\|_{L^2}
 +\|R_F(u)\|_{H^{s+\rho}}.
\end{equation}
Moreover, comparison of the Fourier weights gives
\[
 \|u\|_{H^{\widetilde\sigma}}
 \leq M^{\widetilde\sigma-\sigma}\|u\|_{\sigma,M},
 \qquad
 \|u\|_{L^2}\leq M^{-s}\|u\|_{s,M},
\]
because
$\langle n\rangle^{\widetilde\sigma}
\leq M^{\widetilde\sigma-\sigma}\langle n\rangle_M^\sigma$
for every $n\in\mathbb Z^2$ and $M\geq1$. We also have
\[
 \|u\|_{H^\sigma}\leq\|u\|_{\sigma,M},
 \qquad
 \|u\|_{H^s}\leq\|u\|_{s,M}.
\]
Consequently, using first \eqref{eq:Bony-L2} and then
\eqref{eq:Bony-unweighted},
\[
\begin{aligned}
 M^{s+\rho}\|R_F(u)\|_{L^2}
 &\leq C(F,\widetilde\sigma,\|u\|_{H^\sigma})
 M^{s+\rho}M^{\widetilde\sigma-\sigma}M^{-s}
 \|u\|_{\sigma,M}\|u\|_{s,M}\\
 &=C(F,\widetilde\sigma,\|u\|_{H^\sigma})
 M^{\rho+\widetilde\sigma-\sigma}
 \|u\|_{\sigma,M}\|u\|_{s,M},
\end{aligned}
\]
whereas
\[
 \|R_F(u)\|_{H^{s+\rho}}
 \leq C(F,s,\sigma,\rho,\|u\|_{H^\sigma})
 \|u\|_{\sigma,M}\|u\|_{s,M}.
\]
Since $\rho+\widetilde\sigma-\sigma<0$, the last power of $M$ is
bounded by one.
Combining the last two bounds with \eqref{eq:Bony-weight-split}
proves the claim.
\end{proof}
}

{
\begin{lemma}[\emph{Change of quantization}]\label{lem:cambiamo}
	Let $s\in\mathbb{R}$, $s_0>2$, and let $m \in \mathbb{R}$. Let $\mathcal{T}_a$ denote the classical Bony paradifferential quantization in \eqref{quantiBony} and $T_a = \opbw(a)$ denote the Bony-Weyl quantization in \eqref{quantiWeyl}. We have the following.
	\begin{itemize}
		\item[(i)] If $a \in \mathcal{N}^{m}_{s_0+2,M,2}$, we have
				\begin{equation}\label{eq:transition_bound_gen}
			\mathcal{T}_a h = T_a h + \frac{\mathrm{i}}{2} T_{\nabla_x \cdot \nabla_\xi a} h + R^{(1)}(a) h\,,
		\end{equation}
		where the  remainder $R^{(1)}$ gains exactly two derivatives (order $m-2$):
		\begin{equation}\label{eq:Rtrans1_bound}
			\|R^{(1)}(a) h\|_{s-m+2,M} \lesssim |a|_{\mathcal{N}^{m}_{s_0+2,M,2}} \|h\|_{s,M}\,, \qquad \forall h \in H^s(\mathbb{T}^2)\,.
		\end{equation}
		\item[(ii)] Let $\sigma>1$ and $0\leq\rho<\sigma-1$. If $a=a(x)\in H^\sigma(\mathbb{T}^2)$ is independent of $\xi$, then, for every constant $c$,
		\begin{equation}\label{eq:transition_bound_spatial}
			\|\mathcal{T}_a h-T_a h\|_{s+\rho,M}
			\lesssim \|a-c\|_{\sigma,M}\|h\|_{s,M}\,,
			\qquad \forall h \in H^s(\mathbb{T}^2)\,.
		\end{equation}
	\end{itemize}
	All constants are uniform for $M\geq1$.
\end{lemma}

\begin{proof}
By recalling the definition of the cutoff function $\chi_{\eps}$ in \eqref{cutofffunctepsilon},
	we introduce the  operator $\widetilde{\mathcal{T}}_a h$	\begin{equation}\label{eq:inter_op}
		\widehat{(\widetilde{\mathcal{T}}_a h)}(\xi) := \frac{1}{2\pi}\sum_{\eta\in\mathbb{Z}^{2}} \chi_{\eps}\left(\frac{|\xi-\eta|}{\langle\xi+\eta\rangle}\right) \widehat{a}(\xi-\eta, \eta) \widehat{h}(\eta)\,.
	\end{equation}
We consider the following splitting
	\begin{equation}\label{eq:trans_split}
		\mathcal{T}_a  - T_a  = R_{\mathrm{cutoff}}+R_{\mathrm{Taylor}}, \qquad R_{\mathrm{cutoff}}:=\mathcal{T}_a  - \widetilde{\mathcal{T}}_a , \qquad  R_{\mathrm{Taylor}}:=\widetilde{\mathcal{T}}_a  - T_a \,.
	\end{equation}
We begin by estimating the \emph{cutoff} remainder.
	The multiplier of this difference is  $\Psi(\xi-\eta, \eta) - \chi_{\eps}\left(\frac{|\xi-\eta|}{\langle\xi+\eta\rangle}\right)$ (recall \eqref{quantiBony}, \eqref{cutoff:standard}). Since both cut-offs are identically $1$ in the  region $|\xi-\eta| \ll \langle\eta\rangle$, their difference  is supported exclusively in the region where the frequencies are comparable, i.e., $|\xi-\eta| \gtrsim \langle \eta \rangle \sim \langle \xi \rangle$. 
	More precisely, if
	\[
	\omega(p,\eta):=\Psi(p,\eta)-
	\chi_\eps\left(\frac{|p|}{\langle p+2\eta\rangle}\right),
	\qquad p=\xi-\eta,
	\]
	the admissibility properties of the two cut-offs give constants
	$c,C>0$ such that, outside a fixed bounded set,
	\begin{equation}\label{eq:cutoff-comparison}
	\omega(\xi-\eta,\eta)\neq0
	\quad\Longrightarrow\quad
	c\langle\eta\rangle\leq|\xi-\eta|\leq C\langle\eta\rangle,
	\qquad
	C^{-1}\langle\eta\rangle\leq\langle\xi\rangle\leq C\langle\eta\rangle.
	\end{equation}
	The finitely many frequencies in the bounded set satisfy the same
	estimates below after changing the constant.

	On this support, we can extract up to $2$ derivatives from the spatial regularity of the symbol. By following the exact same $\ell^2$-convolution arguments used to bound the difference of two admissible cut-offs (see item (iv) of Lemma \ref{azione}), we obtain, for all $h\in H^s$, 
	\[
	\langle\xi\rangle_M^{s-m+2}
	\big|\widehat a(\xi-\eta,\eta)\widehat h(\eta)\big|
	\lesssim
	|a|_{\mathcal N^m_{s_0+2,M,0}}
	\langle\xi-\eta\rangle_M^{-s_0}
	\langle\eta\rangle_M^s|\widehat h(\eta)|.
	\]
	Young's inequality for sequences, together with
	$\|\langle p\rangle_M^{-s_0}\|_{\ell^1}\leq
	\|\langle p\rangle^{-s_0}\|_{\ell^1}<\infty$, then gives
	\begin{equation}\label{eq:R_cutoff_bound}
		\|R_{\mathrm{cutoff}} h\|_{s-m+2,M} \lesssim |a|_{\mathcal{N}^{m}_{s_0+2,M,0}} \|h\|_{s,M}\,.
	\end{equation}
We now analyze the \emph{Taylor} remainder $R_{\mathrm{Taylor}}$.  The Fourier transform reads:
	\begin{equation}\label{eq:diff_Fourier_Taylor}
		\widehat{(R_{\mathrm{Taylor}}h)}(\xi) = \frac{1}{2\pi}\sum_{\eta\in\mathbb{Z}^{2}} \chi_{\eps}\left(\frac{|\xi-\eta|}{\langle\xi+\eta\rangle}\right) \Big[ \widehat{a}(\xi-\eta, \eta) - \widehat{a}\Big(\xi-\eta, \frac{\xi+\eta}{2}\Big) \Big] \widehat{h}(\eta)\,.
	\end{equation}
We expand the  difference by means of the second order Taylor's formula around the point $\frac{\xi+\eta}{2}$. Setting $p = \xi-\eta$, we have:
	\begin{equation}\label{eq:taylor_exp_2nd}
		\widehat{a}(p, \eta) - \widehat{a}\Big(p, \eta + \frac{p}{2}\Big) = - \sum_{|\alpha|=1} \frac{p^\alpha}{2} \partial_\xi^\alpha \widehat{a}\Big(p, \eta + \frac{p}{2}\Big) + \sum_{|\beta|=2} \frac{p^\beta}{2\beta!} \int_0^1 t \, \partial_\xi^\beta \widehat{a}\Big(p, \eta + t \frac{p}{2}\Big) \, dt\,.
	\end{equation}
	Using the  identity $p^\gamma \widehat{a}(p, \zeta) = (-\mathrm{i})^{|\gamma|} \widehat{\partial_x^\gamma a}(p, \zeta)$, we rewrite \eqref{eq:taylor_exp_2nd} as:
	\begin{equation}\label{eq:taylor_exp_2nd_fourier}
		\widehat{a}(p, \eta) - \widehat{a}\Big(p, \eta + \frac{p}{2}\Big) = \frac{\mathrm{i}}{2} \widehat{(\nabla_x \cdot \nabla_\xi a)}\Big(p, \eta + \frac{p}{2}\Big) - \frac{1}{4} \sum_{|\beta|=2} \frac{2!}{\beta!} \int_0^1 t \, \widehat{(\partial_x^\beta \partial_\xi^\beta a)}\Big(p, \eta + t \frac{p}{2}\Big) \, dt\,.
	\end{equation}
	Plugging the first term of \eqref{eq:taylor_exp_2nd_fourier} into the definition of $R_{\mathrm{Taylor}}h$, we exactly recover the Weyl quantization of the first-order correction:
	\begin{equation}
		\frac{\mathrm{i}}{2} \frac{1}{2\pi}\sum_{\eta\in\mathbb{Z}^{2}} \chi_{\eps}\left(\frac{|\xi-\eta|}{\langle\xi+\eta\rangle}\right) \widehat{(\nabla_x \cdot \nabla_\xi a)}\Big(\xi-\eta, \frac{\xi+\eta}{2}\Big) \widehat{h}(\eta) \equiv \frac{\mathrm{i}}{2} \widehat{\big( T_{\nabla_x \cdot \nabla_\xi a} h \big)}(\xi)\,.
	\end{equation}
	We call $\mathcal{E}_{\mathrm{Taylor}}h$ the second term in \eqref{eq:taylor_exp_2nd_fourier}. Defining $b_\beta := \partial_x^\beta \partial_\xi^\beta a$ for $|\beta|=2$, the assumption $a \in \mathcal{N}^{m}_{s_0+2,M,2}$ guarantees $b_\beta \in \mathcal{N}^{m-2}_{s_0,M,0}$. Hence, its Fourier coefficients satisfy:
	\begin{equation}
		\big| \widehat{b_\beta}(\xi-\eta, \zeta) \big| \lesssim \langle\zeta\rangle_M^{m-2} |a|_{\mathcal{N}^{m}_{s_0+2,M,2}} \langle\xi-\eta\rangle_M^{-s_0}\,.
	\end{equation}
	On the support of $\chi_\eps$, we have $\langle \eta + t \frac{\xi-\eta}{2} \rangle_M \sim \langle \xi+\eta \rangle_M \sim \langle \xi \rangle_M \sim \langle \eta \rangle_M$ uniformly for $t\in [0, 1]$. The weighted Sobolev norm of the remainder is bounded via discrete Young's convolution inequality:
	\begin{equation}
		\begin{aligned}
			\|\mathcal{E}_{\mathrm{Taylor}} h\|_{s-m+2, M}^{2}
			&\lesssim \sum_{|\beta|=2} \sum_{\xi\in\mathbb{Z}^{2}}
			\langle\xi\rangle^{2(s-m+2)}_M
			\Bigg| \sum_{\eta\in \mathbb{Z}^{2}}
			\chi_{\eps}\left(\frac{|\xi-\eta|}{\langle\xi+\eta\rangle}\right)\\
			&\hspace{35mm}\times
			\int_0^1 t \big| \widehat{b_\beta}\Big(\xi-\eta,
			\eta + t \frac{\xi-\eta}{2}\Big) \big| \,dt \,
			|\widehat{h}(\eta)| \Bigg|^{2} \\
			&\lesssim |a|_{\mathcal{N}^{m}_{s_0+2,M,2}}^2 \sum_{\xi\in\mathbb{Z}^{2}} \Bigg( \sum_{\eta\in \mathbb{Z}^{2}} \frac{\langle\xi\rangle^{s-m+2}_M \langle\xi+\eta\rangle_M^{m-2}}{\langle\xi-\eta\rangle_M^{s_0}} |\widehat{h}(\eta)| \Bigg)^{2} \\
			&\lesssim |a|_{\mathcal{N}^{m}_{s_0+2,M,2}}^2
			\big\| \langle\cdot\rangle_M^{-s_0} *
			\big(\langle\cdot\rangle_M^{s} |\widehat{h}|\big) \big\|_{\ell^2}^2 \\
			&\lesssim |a|_{\mathcal{N}^{m}_{s_0+2,M,2}}^2 \|h\|_{s,M}^2\,.
		\end{aligned}
	\end{equation}
	We conclude the proof of item (i) by defining the total remainder $R^{(1)}(a) := R_{\mathrm{cutoff}} + \mathcal{E}_{\mathrm{Taylor}}$.

We now prove item (ii). Since the two quantizations are equals on constants, we may replace $a$ by $a-c$. If $a = a(x)$ is independent of $\xi$, its Fourier transform $\widehat{a}(p, \zeta)$ does not depend on the second variable $\zeta$. Therefore, the evaluation at $\eta$ and $\frac{\xi+\eta}{2}$ in \eqref{eq:diff_Fourier_Taylor} are algebraically identical:
	\begin{equation*}
		\widehat{a}(\xi-\eta, \eta) - \widehat{a}\Big(\xi-\eta, \frac{\xi+\eta}{2}\Big) \equiv 0\,.
	\end{equation*}
	Consequently, $R_{\mathrm{Taylor}} \equiv 0$ and the difference reduces entirely to $R_{\mathrm{cutoff}}$. On the support of the difference of the two cut-offs, the $x$-frequency of $a$ is comparable with the frequency of $h$, apart from a bounded low-frequency region. The same weighted convolution argument used in \eqref{eq:R_cutoff_bound}, extracting only $\rho$ derivatives, gives
	\[
	\langle\xi\rangle_M^{s+\rho}
	\big|\widehat{(a-c)}(\xi-\eta)\widehat h(\eta)\big|
	\lesssim
	\langle\xi-\eta\rangle_M^\rho
	\big|\widehat{(a-c)}(\xi-\eta)\big|
	\langle\eta\rangle_M^s|\widehat h(\eta)|.
	\]
	Moreover,
	\begin{equation}\label{eq:spatial-cutoff-l1}
	\begin{aligned}
	\big\|\langle p\rangle_M^\rho\widehat{(a-c)}(p)\big\|_{\ell^1_p}
	&\leq
	\big\|\langle p\rangle_M^{-(\sigma-\rho)}\big\|_{\ell^2_p}
	\|a-c\|_{\sigma,M}\\
	&\lesssim \|a-c\|_{\sigma,M},
	\end{aligned}
	\end{equation}
	uniformly in $M$, since $\sigma-\rho>1$ in dimension two.
	Young's inequality therefore gives
	\[
		\|R_{\mathrm{cutoff}}h\|_{s+\rho,M}
		\lesssim \|a-c\|_{\sigma,M}\|h\|_{s,M},
	\]
	which proves (ii).
\end{proof}
}

\begin{cor}[\emph{Bony--Weyl paralinearization and exact homogeneity}]
\label{cor:weyl_para}
Let $s\geq\sigma>3$ and let $u\in H^s(\T^2)$ be real valued.
Let $F\in C^\infty(\R;\R)$ satisfy $F(0)=0$. Then
\begin{equation}\label{eq:weyl_para_base}
 \|F(u)-\opbw(F'(u))u\|_{s+2,M}
 \lesssim C(F,s,\sigma,\|u\|_{H^\sigma})
 \|u\|_{\sigma,M}\|u\|_{s,M}.
\end{equation}
If, for some $N\in\mathbb N$, $N\geq1$,
\[
 F^{(j)}(0)=0,\qquad j=0,\ldots,N,
\]
then
\begin{equation}\label{eq:weyl_para_homogeneous}
 \|F(u)-\opbw(F'(u))u\|_{s+2,M}
 \lesssim C(F,N,s,\sigma,\|u\|_{H^\sigma})
 \|u\|_{\sigma,M}^{N}\|u\|_{s,M}.
\end{equation}
Both constants are uniform for $M\geq1$.
\end{cor}

\begin{proof}
Since $\sigma>3$, we have $2<\sigma-1$. Theorem
\ref{thm:formulaBony}, with $\rho=2$, gives
\[
 \|F(u)-\mathcal T_{F'(u)}u\|_{s+2,M}
 \lesssim C(F,s,\sigma,\|u\|_{H^\sigma})
 \|u\|_{\sigma,M}\|u\|_{s,M}.
\]
The symbol $F'(u)$ is independent of $\xi$. Applying item (ii) of
Lemma \ref{lem:cambiamo} with $\rho=2$ and
$c=F'(0)$, and then applying \eqref{Moser2:base} to
$F'(z)-F'(0)$, we obtain
\[
 \|(\mathcal T_{F'(u)}-T_{F'(u)})u\|_{s+2,M}
 \lesssim \|F'(u)-F'(0)\|_{\sigma,M}\|u\|_{s,M}
 \lesssim C(F,\sigma,\|u\|_{H^\sigma})
 \|u\|_{\sigma,M}\|u\|_{s,M}.
\]
This proves \eqref{eq:weyl_para_base}.

Assume now the additional vanishing conditions and set
$\lambda:=\|u\|_{H^\sigma}$. If $\lambda=0$, then $u=0$ and the
claim is immediate. Otherwise define
\[
 v:=\lambda^{-1}u,
 \qquad
 F_\lambda(z):=\lambda^{-(N+1)}F(\lambda z).
\]
Then $\|v\|_{H^\sigma}=1$, and hence
$\|v\|_{L^\infty}\leq C_\sigma$. By linearity of the Bony--Weyl
quantization in the symbol and in the argument,
\begin{equation}\label{eq:rescaling_para_identity}
 F(u)-T_{F'(u)}u
 =\lambda^{N+1}\big(F_\lambda(v)-T_{F_\lambda'(v)}v\big).
\end{equation}
Set $R_0:=\max\{1,\lambda\}$. Taylor's formula and the vanishing
assumptions imply that the family
$F_\mu(z)=\mu^{-(N+1)}F(\mu z)$, $0<\mu\leq R_0$, is uniformly
bounded in the finitely many $C^K$ seminorms entering
\eqref{eq:weyl_para_base} on the fixed interval
$|z|\leq C_\sigma$. Therefore the constants obtained by applying
the base estimate to $(F_\lambda,v)$ are uniform.

Using \eqref{eq:rescaling_para_identity}, we obtain
\begin{equation}\label{eq:rescaling_para_bound}
\begin{aligned}
 \|F(u)-T_{F'(u)}u\|_{s+2,M}
 &\leq C(F,N,s,\sigma,R_0)\lambda^{N+1}
 \|v\|_{\sigma,M}\|v\|_{s,M}\\
 &=C(F,N,s,\sigma,R_0)\lambda^{N-1}
 \|u\|_{\sigma,M}\|u\|_{s,M}.
\end{aligned}
\end{equation}
Finally,
\[
 \lambda=\|u\|_{H^\sigma}\leq\|u\|_{\sigma,M}.
\]
Since $N\geq1$, inserting this estimate into
\eqref{eq:rescaling_para_bound} proves
\eqref{eq:weyl_para_homogeneous}.
\end{proof}

\subsection{On the regularity threshold}\label{sec:s_0}

We show here that the condition $s_0>4$ is sufficient for the pseudo-differential bounds used to prove the energy estimates in Proposition \ref{modified_energy}. Moreover, this restriction on $s_0$ appears to be essentially sharp for the techniques employed here. In particular, reaching the finite energy threshold $s_0>1$ would  require new ideas, rather than a  refinement of the present argument.

Let us consider the principal symbol of the commutator driving the energy estimate,
{
\begin{equation}
	c(x, \xi) = \{ \mu^s \langle\xi\rangle_M^{2s}, \, \mu |\xi|^2 \}
	= -2s M^2 \mu^s \langle\xi\rangle_M^{2s-2} (\xi \cdot \nabla_x \mu)\,.
\end{equation}
To evaluate its action on the modified Sobolev space $H^{s}$, we normalize the symbol by the principal weight $\langle\xi\rangle_M^{2s}$. Up to the constant factor $-2s$, which does not affect the derivative count, the normalized symbol of order zero is given by:
\begin{equation}\label{normalized_symb}
	a(x, \xi) = M^2 \mu(x)^s\frac{\xi}{\langle\xi\rangle_M^2} \cdot \nabla_x \mu(x)\,.
\end{equation}
}
According to \cite[Theorem 2.1 and Remark on p.~350]{Taylor_Pseudo}, in dimension two the $L^2$-operator norm of $\opbw(a)$ is bounded by a finite sum of the supremum norms of the symbol's derivatives, both in space and in frequency, up to order two.

The bound reads 
\begin{equation}\label{CV_bound}
\| \opbw(a) \|_{\mathcal{L}(L^2 \to L^2)} \le C \sum_{|\alpha| \le 2} \sum_{|\beta| \le 2} \sup_{(x,\xi) \in \mathbb{T}^2 \times \mathbb{R}^2} \left| \partial_\xi^\alpha \partial_x^\beta a(x, \xi) \right|\,.
\end{equation}
Let us split in the symbol $a(x, \xi)$ in \eqref{normalized_symb} the spatial and frequency variables.

Up to the constant factor $-2s$, we write
$a(x, \xi) = F(x) \cdot G(\xi)$, where $F(x) = \mu(x)^s\nabla_x \mu(x)$ and $G(\xi) = M^2 \xi \langle\xi\rangle_M^{-2}$. 
Then, we can write 
\begin{equation}
	\partial_\xi^\alpha \partial_x^\beta a(x, \xi) = \left( \partial_x^\beta F(x) \right) \cdot \left( \partial_\xi^\alpha G(\xi) \right)\,.
\end{equation}

\noindent \textbf{1. Frequency Derivatives} \par\smallskip
The function $G(\xi) = M^2 \xi (M^2 + |\xi|^2)^{-1}$ is $C^\infty$ for any $\xi \in \mathbb{R}^2$  ($M \gg 1$). Its supremum is bounded by $M/2$. Moreover, any derivative $\partial_\xi^\alpha G(\xi)$ produces a function with a stronger decay at infinity. Therefore, for any multi-index $\alpha$, we have
\begin{equation}
	\sup_{\xi \in \mathbb{R}^2} \left| \partial_\xi^\alpha G(\xi) \right| \le C_\alpha M\,.
\end{equation}
No need to impose additional regularity on $\mu(x)$.\par\smallskip

\noindent \textbf{2. Spatial Derivatives and $s_0 > 4$} \par\smallskip
The bound \eqref{CV_bound} requires taking the supremum of spatial derivatives up to order $|\beta| = 2$.
{
\begin{itemize}
	\item For $|\beta| = 0$: $\sup_{x} |F(x)| \leq \|F\|_{L^\infty}$.
	\item For $|\beta| = 1$: $\sup_{x} |\partial_x F(x)| \leq \|\nabla_xF\|_{L^\infty}$.
	\item For $|\beta| = 2$: $\sup_{x} |\partial_x^2 F(x)| \leq \|\nabla_x^2F\|_{L^\infty}$.
\end{itemize}
Summing these contributions and recalling that $\mu$ is close to $1$, we obtain
\begin{equation}
	\|\opbw(a)\|_{\mathcal L(L^2)}
	\lesssim M\|\mu^s\nabla\mu\|_{W^{2,\infty}}
	\lesssim M\|\mu-1\|_{W^{3,\infty}}
	\lesssim M\|\Phi(Z)\|_{s_0,M}^6.
\end{equation}
Since by definition $\mu-1$ is a smooth function of $\Phi(Z)(x)$ vanishing to order six, the last estimate follows from the tame composition bounds. Controlling its $W^{3,\infty}$ norm requires $\Phi(Z) \in W^{3,\infty}(\mathbb{T}^2)$. 
}
Finally, by the two-dimensional Sobolev embedding theorem,
	\(H^{s_0}(\mathbb{T}^2) \hookrightarrow W^{k,\infty}(\mathbb{T}^2)\)
	for \(s_0>k+1.\)
Setting $k=3$, we obtain 
	\(s_0>4\,.\)

\end{document}